\documentclass[12pt]{amsart}
\usepackage{amssymb}
\usepackage[all]{xy}
\usepackage{graphicx}
\usepackage{mathrsfs}
\usepackage{hyperref}
\input xy
\xyoption{all}

\usepackage{amscd}
\newtheorem{lemma}{Lemma}[section]
\newtheorem{corollary}[lemma]{Corollary}
\newtheorem{theorem}[lemma]{Theorem}

\theoremstyle{definition}
\newtheorem{remark}[lemma]{Remark}
\newtheorem{definition}[lemma]{Definition}

\newtheorem{example}[lemma]{Example}

\DeclareMathOperator{\Hom}{Hom}
\DeclareMathOperator{\Ext}{Ext}
\DeclareMathOperator{\End}{End}

\DeclareMathOperator{\modd}{mod}
\DeclareMathOperator{\ind}{ind}
\DeclareMathOperator{\Gr}{Gr}
\DeclareMathOperator{\proj}{proj}
\DeclareMathOperator{\coind}{coind}
\DeclareMathOperator{\add}{add}
\DeclareMathOperator{\rad}{rad}
\DeclareMathOperator{\dimm}{dim}

\usepackage{color}

\begin{document}

\title{Cluster algebras generated by projective
cluster variables}

\author{Karin Baur}
\address{School of Mathematics, University of Leeds, Leeds, LS2 9JT, United Kingdom}
\address{On leave from the University of Graz}
\email{pmtkb@leeds.ac.uk}

\author{Alireza Nasr-Isfahani}
\address{Department of Pure Mathematics,
Faculty of Mathematics and Statistics,
University of Isfahan,
Isfahan 81746-73441, Iran and School of Mathematics, Institute for Research in Fundamental Sciences (IPM), P.O. Box: 19395-5746, Tehran,
Iran}
\email{nasr$_{-}$a@sci.ui.ac.ir / nasr@ipm.ir}



\begin{abstract}
We introduce the notion of a lower bound cluster algebra generated by projective cluster variables
as a polynomial ring over the initial cluster variables and the so-called projective cluster variables.
We show that under an acyclicity assumption, the cluster algebra and the lower bound cluster
algebra generated by projective cluster variables coincide.
In this case we use our results to construct a basis for the cluster algebra.
We also show that any coefficient-free cluster algebra of types $A_n$ or $\widetilde{A}_n$ is equal
to the corresponding lower bound cluster algebra generated by projective cluster variables.
\end{abstract}

\maketitle

\section{Introduction}
Cluster algebras were introduced and first investigated by Fomin and
Zelevinsky \cite{FZ1} in order to better understand Lusztig's theory of total positivity and canonical bases in quantum groups \cite{L, L1}. The rings of functions of many important varieties, such as Grassmannians, semisimple Lie groups, flag varieties and Teichm\"{u}ller spaces have a cluster algebra structure. There are two general approaches towards cluster algebras:
\begin{itemize}
\item[1)] Study subalgebras of cluster algebras whose behavior
is known better~\cite{BFZ, M, MRZ}.
\item[2)] Study related larger algebras with better behavior, for a given cluster algebras~\cite{BFZ}.
\end{itemize}
The notion of a
lower bound cluster algebra was introduced by Berenstein, Fomin and Zelevinsky \cite{BFZ} as
a finitely generated subalgebra of the associated cluster algebra.
They proved that the cluster algebra $\mathcal{A}(\Sigma)$ associated with a totally mutable seed
$\Sigma$ is equal to the lower bound $\mathcal{L}(\Sigma)$ if and only if $\Sigma$ acyclic. In this case
they constructed a basis for the cluster algebra $\mathcal{A}(\Sigma)$, and relations
among the generators of the cluster algebra \cite{BFZ}. Combining these results we
have an explicit description of cluster algebras with acyclic seeds. Muller, Rajchgot and Zykoski
in~\cite{MRZ} used the cycle relations and gave an explicit presentation for each lower bound cluster
algebra. By using this presentation they proved that each lower bound cluster algebra is Cohen-Macaulay
and normal.

Cluster categories were introduced
in \cite{BMRRT} (also for type $A_n$ in \cite{CCS}) as a
categorical model for cluster algebras.
The cluster-tilted algebras as introduced by Buan, Marsh and
Reiten \cite{BMR1} have a key role in  the study of cluster
categories. Also an important connection between cluster algebras
and cluster-tilted algebras was established in \cite{BMR2, CCS2}.
Caldero and Chapoton defined in \cite{CC},
for each object $M$ of a cluster category a fraction $X_M$, i.e. a cluster variable of
an associated cluster algebra.
They showed that for every cluster category, the corresponding map (the so-called
the Caldero-Chapoton map) from the set of (rigid)
indecomposable factors of cluster-tilting objects to the set of cluster variables is bijective
(see also \cite{CK}). The bijectivity of the Caldero-Chapoton map for acyclic cluster algebras was
proved in \cite{BMRT}. For an arbitrary cluster-tilting object $T$ in a $2$-Calabi-Yau triangulated
category $\mathcal C$ over an algebraically closed field, Palu defined a fraction $X(T,L)$ in the
cluster algebra for every
indecomposable object $L$ of $\mathcal C$ and showed that the so-called cluster character,
i.e. the map $L\mapsto X(T,L)$, satisfies a certain multiplicative formula, \cite{Pa}.
Palu proved furthermore that in the finite and acyclic case, this map induces a bijection between
the indecomposable rigid objects of the cluster category and the cluster variables of the cluster algebra,
thus confirming conjecture 2 of~\cite{CK}.

Let $\Sigma$ be a seed, $\mathcal{A}(\Sigma)$ be an acyclic cluster algebra and $T=\oplus_{i=1}^{n}T_i$ (with
indecomposable summands $T_i$)
be the corresponding cluster-tilting object in the associated cluster category as in~\cite{BMRRT}.
We call the $x_{P_i}=X(T, T_i)$, $i=1,\dots, n$, the {\em projective cluster variables}.
In this paper we define the {\em lower bound cluster algebra generated by projective cluster variables},
denoted by $\mathcal{L}_P(\Sigma)$, as the subalgebra of the
field of rational functions in the initial cluster variables generated by the union of the
initial cluster variables with the projective cluster variables.
By definition, $\mathcal{L}_P(\Sigma)$ is contained in the cluster algebra $\mathcal{A}(\Sigma)$
and is finitely generated.
We prove that if
$\Sigma$ is an acyclic seed, then $\mathcal{L}_P(\Sigma)=\mathcal{A}(\Sigma)$ (Theorem~\ref{thm1}).
In this case we provide a basis for the cluster algebra by the so-called standard monomials
in the sets $\{x_i\}_i$ and $\{x_{P_i}\}_i$, Theorem~\ref{thm:standard}.
We also show that the two algebras $\mathcal{L}_P(\Sigma)$ and $\mathcal{A}(\Sigma)$ coincide
for any seed in the types $A_n$ and $\widetilde{A}_n$ (Theorem~\ref{thm3} and Theorem~\ref{thm4}).
Thus we give a large class of seeds $\Sigma$ where the lower bound cluster algebra
generated by projective cluster variables  is equal to the
cluster algebra $\mathcal A(\Sigma)$ but $\mathcal L(\Sigma)$ is not equal to the cluster algebra.
In particular, this proves that
$\mathcal{L}(\Sigma)\ne \mathcal{L}_P(\Sigma)$ in general (see Example~\ref{ex:cyclic}).
Our results illustrate that the lower bound cluster algebras generated by projective cluster
variables have good properties and they provide new information about cluster algebras.

The paper is organized as follows. In Section 2, we recall background and
results that will
be needed later in the paper. We also introduce our notion of projective cluster variables and the
lower bound cluster algebra generated by them.
In Section 3, we show that for any totally mutable acyclic seed
$\Sigma$, $\mathcal{A}(\Sigma)=\mathcal{L}_P(\Sigma)$ and that the standard monomials
in the initial cluster variables and in the projective cluster variables form a basis of the
cluster algebra. In Section 4, we study the behavior of projective cluster variables in acyclic coefficient-free cluster algebras under mutation. In Section 5, using the classification of quivers of type $A_n$ up to mutation equivalence
of~\cite{BV}, we show that any coefficient free cluster algebra of type $A_n$ is equal to
the corresponding lower bound cluster algebra generated by projective cluster variables.
Finally in the last section we show that in type $\widetilde{A}_n$, the two algebras $\mathcal{A}(\Sigma)$ and $\mathcal{L}_P(\Sigma)$
coincide.


\section{Background} \label{sec:background}

We first recall the definition of a cluster algebra.
Let $\mathbb{P}$ be an abelian group without torsion, called the coefficient group. Let $n$ be a
positive integer and let $[1, n]$ stand for the set $\lbrace 1, 2, \cdots, n\rbrace$.
Let $\mathcal{F}$ be the field of rational functions in $n$ independent variables with coefficients in the field
of fractions of the integer group ring $\mathbb{ZP}$. A {\em seed} in $\mathcal{F}$ is a triple
$\Sigma=(\mathbf{x}, \mathbf{p}, B)$, where $\mathbf{x}=\lbrace x_1, x_2, \cdots, x_n\rbrace$ is a transcendence basis of $\mathcal{F}$ over the field of fraction of
$\mathbb{ZP}$, $\mathbf{p}=(p_1^{\pm}, \cdots,p_n^{\pm})$ is a $2n$-tuple of elements of $\mathbb{P}$ and $B =(b_{ij})_{ij}$ is an
$n\times n$ sign-skew-symmetric integer matrix. The set $\mathbf{x}$ is called a {\em cluster} and its elements $x_i$ are the {\em initial cluster variables}.
The tuple $\mathbf{p}$ is called the coefficient tuple
and $B$ is called the {\em exchange matrix} of the seed $\Sigma$.
Note that we will mainly restrict to the coefficient-free case.
For $j\in [1, n]$, the {\em adjacent cluster} $\mathbf{x}_j$ is
$\mathbf{x}_j=\mathbf{x}-\lbrace x_j\rbrace\cup\lbrace x'_j\rbrace$, where the cluster variables $x_j$ and $x'_j$
are related by the exchange relation
$$
x_jx'_j:=P_j(\mathbf{x}):=p_j^{+}\prod_{b_{ij}>0} x_i^{b_{ij}} + p_j^{-}\prod_{b_{ij}<0} x_i^{-b_{ij}} .
$$
Let $\mathbb{ZP}[\mathbf{x}]=\mathbb{ZP}[x_1, \cdots, x_n]$ (resp.,
$\mathbb{ZP}[\mathbf{x}^{\pm 1}]=\mathbb{ZP}[x_1^{\pm 1}, \cdots, x_n^{\pm 1}]$)
denote the polynomial ring
(resp., Laurent polynomial ring) in $x_1, \cdots, x_n$ with coefficients in $\mathbb{ZP}$.

For a given seed $\Sigma$, the upper bound cluster algebra
$\mathcal{U}(\Sigma)$ is the
$\mathbb{ZP}$-subalgebra of $\mathcal{F}$ given by
$\mathcal{U}(\Sigma)=\mathbb{ZP}[\mathbf{x}^{\pm 1}]\cap\mathbb{ZP}[\mathbf{x}_1^{\pm 1}]
\cap\cdots \cap \mathbb{ZP}[\mathbf{x}_n^{\pm 1}]$, as
defined by Berenstein, Fomin and Zelevinsky in~\cite{BFZ}.
Let $B=(b_{ij})$ and $B'=(b_{ij}')$ be real square matrices of the same size.
$B'$ is called the {\em mutation of the matrix $B$ in direction
$k$} (denoted $B'=\mu_k(B)$) if
\begin{align*}
b_{ij}' = \begin{cases}
-b_{ij} & \text{if }\ i=k\ \text{or}\ j=k,\\
b_{ij}+\frac{|b_{ik}|b_{kj}+b_{ik}|b_{kj}|}{2}& \text{otherwise. }\
\end{cases}
\end{align*}
Let $\Sigma=(\mathbf{x}, \mathbf{p}, B)$ be a seed and $k\in [1, n]$. A seed
$\Sigma'=(\mathbf{x}', \mathbf{p}', B')$ is called {\em the mutation of $\Sigma$ in direction $k$} if,
$\mathbf{x}'=\mathbf{x}_k$, $B'=\mu_k(B)$ and $\mathbf{p}'=(p'^{\pm}_1, \cdots, p'^{\pm}_n)$, where
$p'^{\pm}_k=p^{\mp}_k$ and for $i\neq k$,
\begin{align*}
p'^{+}_i/p'^{-}_i = \begin{cases}
(p^{+}_k)^{b_{ki}}p^{+}_i/p^{-}_i & \text{if }\ b_{ki}\geq 0,\\
(p^{-}_k)^{b_{ki}}p^{+}_i/p^{-}_i& \text{if }\ b_{ki}\leq 0.\
\end{cases}
\end{align*}
Seed mutations (resp., matrix mutations) give rise to an equivalence relation among seeds (resp.,
square matrices), which is called mutation equivalence and denoted by $\Sigma'\sim \Sigma$
(resp., $B'\sim B$). A seed $\Sigma=(\mathbf{x}, \mathbf{p}, B)$ is called {\em totally mutable} if every
matrix mutation equivalent to $B$ is sign-skew-symmetric.
Let $\Sigma_0$ be a totally mutable seed. The {\em upper cluster algebra}
$\overline{\mathcal{A}}=\overline{\mathcal{A}}(\Sigma_0)$ defined by $\Sigma_0$ is the intersection of
the subalgebras
$\mathcal{U}(\Sigma)\subseteq \mathcal{F}$ for all seeds $\Sigma\sim\Sigma_0$. Upper cluster algebra
were defined and first investigated by Berenstein, Fomin and
Zelevinsky in~\cite{BFZ}. They proved that the coordinate ring of any double Bruhat cell in a
semisimple complex Lie group is isomorphic to its
upper cluster algebra defined in terms of relevant combinatorial data.

The {\em lower bound cluster algebra}
$\mathcal{L}(\Sigma)$ associated with a given seed $\Sigma$ is defined as
$\mathcal{L}(\Sigma)=\mathbb{ZP}[x_1, x'_1, \cdots, x_n, x'_n]$. In fact $\mathcal{L}(\Sigma)$ is the
$\mathbb{ZP}$-subalgebra of $\mathcal{F}$ generated by the union of the $n+1$ clusters
$\mathbf{x}, \mathbf{x}_1, \cdots, \mathbf{x}_n$. Finally, the {\em cluster algebra}
$\mathcal{A}=\mathcal{A}(\Sigma_0)$ associated with a totally mutable seed $\Sigma_0$
is the $\mathbb{ZP}$-subalgebra of $\mathcal{F}$ generated by the union of all lower bounds
$\mathcal{L}(\Sigma)$ for $\Sigma\sim\Sigma_0$. For any $\Sigma\sim\Sigma_0$ we have
$\mathcal{L}(\Sigma)\subseteq\mathcal{A}(\Sigma_0)\subseteq\overline{\mathcal{A}}(\Sigma_0)\subseteq\mathcal{U}(\Sigma)$ \cite{BFZ}.
Let $\Sigma=(\mathbf{x}, \mathbf{p}, B)$ be a seed. The directed graph of $\Sigma$, is the
graph $\Gamma(\Sigma)=\Gamma(B)$ with the vertices $1, \cdots, n$ and the directed
edges $(i, j)$ whenever $b_{ij}>0$.
A seed $\Sigma$ is called {\em acyclic} if $\Gamma(\Sigma)$ has no oriented cycle. A monomial in
$x_1, x_1', \cdots, x_n, x_n'$ is called {\em standard} if it contains no product of the form $x_jx_j'$.
Let $\mathcal{L}(\Sigma)$ be the lower cluster algebra associated with a seed $\Sigma$
and let $I$ be the ideal of relations among its generators $x_1, x_1', \cdots, x_n, x_n'$
for $\mathcal{L}(\Sigma)$.
Berenstein, Fomin and
Zelevinsky showed in \cite{BFZ} that if $\Sigma$ is acyclic, then the polynomials
$x_jx_j'-P_j(\mathbf{x})$ with $j\in [1, n]$ generate the ideal $I$.
Moreover, these polynomials form a Gr\"{o}bner basis for $I$
with respect to any term order in which $x_1', \cdots, x_n'$ are much more expensive than
$x_1, \cdots, x_n$ (see~\cite[Corollary 1.17]{BFZ}).
The authors also proved that the cluster algebra $\mathcal{A}(\Sigma)$ associated with a totally mutable seed
$\Sigma$ is equal to the lower bound $\mathcal{L}(\Sigma)$ if and only if $\Sigma$ is
acyclic, \cite[Theorem 1.20]{BFZ}.
The collection of elements in all clusters obtained
through arbitrary sequences of mutations are called the {\em cluster variables} for $\Sigma$.

It is often convenient to work with quivers instead of matrices in the seeds. Let us recall this notion here.
A quiver is a quadruple $Q=(Q_0, Q_1, s, t)$
formed by a set of vertices $Q_0$, a set of arrows $Q_1$ and two maps $s$ and $t$ from $Q_1$ to $Q_0$ which send
an arrow $a$ respectively to its source $s(a)$ and its target $t(a)$. By a quiver, we always mean a finite quiver without loops nor
$2$-cycles. There is a $1$-$1$ correspondence between quivers and skew-symmetric integer matrices
(up to reordering the columns). A quiver $Q$ is given by the skew-symmetric matrix $B$ whose coefficient
$b_{ij}$ is the difference between the number of arrows from $i$ to $j$ and the number of arrows from $j$ to
$i$. The {\em mutation of the quiver $Q$ in direction
$k$} (denoted $\mu_k(B)$) is the quiver $Q'$ obtained from $Q$ in the following three steps:
\begin{itemize}
\item[1)] For every path $i\rightarrow k\rightarrow j$ add one arrow $i\rightarrow j$,
\item[2)] Reverse all arrows incident with $k$,
\item[3)] Delete $2$-cycles.
\end{itemize}
Quiver mutations give rise to an equivalence relation among quivers, which is called mutation
equivalence and denoted by $Q'\sim Q$. In this paper we work with the seeds of the form
$\Sigma=(\mathbf{x}, \mathbf{p}, Q)$, where $Q$ is a quiver which is mutation equivalent to
an acyclic quiver.

Let us now recall the definition of a cluster category.
Throughout the paper, let $k$ be an algebraically closed field.
Let $H$ be a hereditary finite dimensional $k$-algebra and let $\mathcal{D} = D^b({\rm mod} H)$ be
the bounded derived category of finitely generated right $H$-modules
with shift functor $[1]$. Also, let $\tau$ be the $AR$-translation
in $\mathcal{D}$. For background, we refer to the book~\cite{happel} by Happel.
The {\em cluster category} $\mathcal{C}_H$ is defined as the orbit
category $\mathcal{C}_H = \mathcal{D} /F$, where $F = \tau^{-1}[1]$.
The objects of $\mathcal{C}_H$ are the same as the objects of
$\mathcal{D}$, but maps are given by $\Hom_{\mathcal{C}_H}(X, Y )
=\bigoplus_{i\in \mathbb{Z}} \Hom_\mathcal{D}(X, F^iY)$. An object $\widetilde{T}$ in
$\mathcal{C}_H$ is called {\em cluster-tilting} provided for any object
$X$ of $\mathcal{C}_H$, we have $\Ext^1_{\mathcal{C}_H}(\widetilde{T}, X) = 0$ if
and only if $X$ lies in the additive subcategory ${\rm add}(\widetilde{T})$ of
$\mathcal{C}_H$ generated by $\widetilde{T}$. Let $\widetilde{T}$ be a cluster-tilting
object in $\mathcal{C}_H$.
The {\em cluster-tilted algebra associated to $\widetilde{T}$} is the algebra
$\End_{\mathcal{C}_H}(\widetilde{T})^{op}$ \cite{BMRRT, BMR1}.
It was proved in \cite{BMR2, CCS2} that the quivers
of the cluster-tilted algebras of a given simply-laced Dynkin type
are precisely the quivers of the exchange matrices of the cluster
algebra of the same type. An object $M\in\mathcal{C}_H$ is called rigid if $\Ext^1_{\mathcal{C}_H}(M, M) = 0$.
Let $Q$ be an acyclic quiver with $n$ vertices,
let $kQ$ be its path algebra. This is a hereditary finite dimensional
$k$-algebra and we can thus take its cluster
category $\mathcal C=\mathcal{C}_{kQ}$ as above.

For the moment we consider coefficient-free case; let
$\mathcal{A}=\mathcal{A}(\mathbf{x}, Q)$ be the cluster algebra associated
to $Q$. Then there is a map $\varphi:\mathcal C \to \mathcal A$ inducing bijections (see \cite{BMRRT, CK1}):
\[
\begin{array}{clc}
\{\text{indecomposable rigid objects in $\mathcal{C}$}\} & \stackrel{\varphi}{\longrightarrow}
 &\{\text{cluster variables in $\mathcal{A}$}\}\\
\{\text{cluster-tilting objects in $\mathcal{C}$}\} & \longrightarrow &\{\text{clusters in $\mathcal{A}$}\}\\
T=T_1\oplus \cdots\oplus T_n & \longmapsto & \mathbf{x}_T=\{x_{T_1}, \cdots, x_{T_n}\}\\
\{\text{seeds $(\mathbf{x}, Q)$}\} & \longrightarrow &\{\text{cluster-tilting objects $T$}\}
\end{array}
\]
We can thus associate a cluster-tilted algebra $\End_{\mathcal{C}}(T)^{op}$ to any seed
$(\mathbf{x}, Q)$. The cluster-tilted algebra $\End_{\mathcal{C}}(T)^{op}$ is the quotient
$kQ/I$ of the path algebra $kQ$ by an ideal $I$ where the quiver $Q$ is the same as the quiver
of the seed $(\mathbf{x}, Q)$.
When the cluster algebra is of finite type, a system of relations for the ideal $I$ explicitly described in \cite{CCS2} and \cite{BMR3}.

We now introduce the notion of projective cluster variables, related to the projective
indecomposable objects of $kQ$.

\begin{definition}\label{def:proj-cl-var}
Let $Q$ be an acyclic quiver, $\Sigma=(\mathbf{x}, Q)$ be a seed
and $kQ=\oplus_{i=1}^{n}P_i$, where each $P_i$ is an indecomposable projective $kQ$-module and $P_i\ncong P_j$ for each $i\neq j$.
Then we define the {\em projective cluster variables with respect to} $\Sigma$ to be the
cluster variables $x_{P_i}:=\varphi(P_i)$, $1\leq i\leq n$.
We will sometimes say that the $x_{P_i}$, $1\leq i\leq n$, are the projective cluster variables
of the cluster algebra $\mathcal A(\mathbf{x}, Q)$.
\end{definition}

\begin{example} \label{ex1}
Let $\Sigma=(\mathbf{x}, Q)$, where $\mathbf{x}=\{x_1, x_2\}$ and $Q$ is the
quiver $$\hskip.5cm \xymatrix{
{2} \ar[r]&{1}.\\
}
\hskip.5cm$$
By \cite[Theorem 2.2]{BMRT}, $x_{P_i}=f(x_1, x_2)/m(x_1, x_2)$, as a reduced fraction, where
$m(x_1, x_2)$ is the monomial $x_1^{d_1}x_2^{d_2}$
for $(d_1, d_2)$ the dimension vector $P_i$ and where $f$ is a polynomial in the $x_i$'s.
The projective indecomposable $kQ$ modules are $P_1=S_1$ (the simple module at vertex 1)
with dimension vector $(0,1)$ and $P_2$ the indecomposable module with dimension vector $(1,1)$.
Therefore $x_{P_1}=\dfrac{1+x_2}{x_1}$ and $x_{P_2}=\dfrac{1+x_1+x_2}{x_1x_2}$.
\end{example}

Let $H$ be a hereditary algebra, $\mathcal{C}_H$ be the corresponding cluster category and $T$ be
a cluster-tilting object in $\mathcal{C}_H$. We associate with $T$ the pair $(T,Q_T)$, called a
{\em tilting seed}, where
$Q_T$ is the quiver of $\End_{\mathcal{C}_H}(T)^{op}$.

For each $1\leq i\leq n$, the {\em mutation of $(T, Q_T)$ in direction $i$} is defined to be
$\delta_i(T, Q_T)=(T', Q_T')$, where $T'$ is the cluster-tilting object of $\mathcal{C}_H$ obtained
by exchanging the direct summand $T_i$ with $T_i^*$
(see \cite{BMRRT} and \cite{BMR2} for more details).

Let $\mathcal{A}$ be an acyclic cluster algebra. Then there exists a seed $\Sigma=(\mathbf{u}, Q)$,
where $Q$ is an acyclic quiver and
$\mathbf{u}=\{u_1, u_2, \cdots, u_n\}\subseteq \mathbb{Q}(x_1, \cdots, x_n)$ is a transcendence
basis of $\mathbb{Q}(x_1, \cdots, x_n)$ over $\mathbb{Q}$.
We take such a seed $\Sigma$ as the initial seed.
Buan, Marsh and Reiten have defined a map $\alpha$ in~\cite{BMR2}
associating tilting seeds to seeds for acyclic cluster algebras,
Let $(\mathbf{u}', Q')$ be an arbitrary seed. Then
$(\mathbf{u}', Q')=\mu_{i_t}\cdots \mu_{i_1}(\mathbf{u}, Q)$ for some ordered sequence
$(i_1, \cdots, i_t)$. This sequence in not unique in general and we choose one of minimal length.
The sequence of length zero is the empty set $\varnothing$ - it is associated with the initial seed.
Let $H=kQ$. The map $\alpha$ from seeds to tilting seeds is defined as follows:
$\alpha((\mathbf{u}, Q), \varnothing)=(H[1], Q)$ and
$\alpha((\mathbf{u}', Q'), (i_1, \cdots, i_t))=\delta_{i_t}\cdots \delta_{i_1}(H[1], Q)=(T', Q_{T'})$.
We choose an indexing for the summands $H_i$ of $H$
such that cluster variable $u_i$ (of the cluster $\mathbf{u}$) corresponds to $H_i$.
The induced correspondence between the cluster
variables $u_i'$ in the cluster $\mathbf{u}'$ and the indecomposable direct summands $T_i'$ in $T'$,
is also denoted by $\alpha$.
By \cite[Theorem 6.1]{BMR2}, for every $j\in \{1, 2, \cdots, n\}$
there is a commutative diagram
$$
\xymatrix{
{((\mathbf{u}', Q'), (i_1, \cdots, i_t))}\ar[d]_{\mu_j} \ar[r]^-{\alpha}&{(T', Q_{T'})}\ar[d]^{\delta_j}\\
{((\mathbf{u}'', Q''), (i_1, \cdots, i_t, j))}\ar[r]_-{\alpha}&{(T'', Q_{T''})}\\
}
$$
where $\mathbf{u}''$ is the cluster obtained from $\mathbf{u}'$ by replacing $u_j'\in \mathbf{u}'$ by
$u_j''$ using $\mu_j$ and $T''$ is the cluster tilting object in $\mathcal{C}_H$ obtained by
exchanging the indecomposable summand $T_j'$ by $T_j''$ (using $\delta_j$).

In the following, we extend the notion of projective cluster variables to acyclic cluster algebras
before generalising it
to cluster algebras with coefficients.

\begin{definition}\label{def:proj-cl-var-acyclic}
Let $\mathcal{A}(\mathbf{x}, Q)$ be an acyclic cluster algebra and let $(\mathbf{u}, Q')$ be an
acyclic seed such that $\mathcal{A}(\mathbf{x}, Q)= \mathcal{A}(\mathbf{u}, Q')$. Let
$\mathcal C=\mathcal C_{Q'}$ be the
cluster category corresponding to the seed $(\mathbf{u},Q')$ and $(T', Q)$ be a tilting seed corresponding to the seed $(\mathbf{x}, Q)$.
Then we define the {\em projective cluster variables} with respect to the seed $(\mathbf{x},Q)$
to be the cluster variables $x_{P_i}:=x_{T_i'}=\varphi(T_i')$, where $T'=\oplus T_i'$ and $\varphi$ is the bijection described earlier.
\end{definition}

\begin{remark}
Note that in the above definition, the acyclic seed $(\mathbf{u}, Q')$ is not unique, but \cite[Theorem 6.1]{BMR2} shows that the definition is independent of the choice of acyclic seed $(\mathbf{u}, Q')$.
\end{remark}

\begin{remark}\label{rem:projective variables1}
The motivation to call these variables projective cluster variables arises from the following.
In the setting of Definition~\ref{def:proj-cl-var-acyclic}, the functor $\Hom_{\mathcal{C}}(T',-)$ induces an equivalence of
categories $\mathcal{C}/ \add(\tau T')\rightarrow \modd$-$\End_{\mathcal{C}}(T')^{op}$, \cite{BMR1}. Under this equivalence, every $\Hom_{\mathcal{C}}(T',T'_i)=:P_i$ is an indecomposable projective
$\End_{\mathcal{C}}(T')^{op}$-module ($i=1,\dots, n$).
Consider the projective cluster variable $x_{P_i}=\varphi(T_i')$, an element of
$\mathcal A(\mathbf{x},Q)=\mathcal A(\mathbf{u},Q')$.
By the Laurent phenomenon~\cite[Theorem 3.1]{FZ1},
this can be expressed as a reduced fraction whose denominator is a monomial in the
$(u_1,\dots, u_n)$ of $\mathcal A(\mathbf{u},Q')$.

Let $H=kQ'$. By \cite[Theorem 2.2]{BMRT}, there is a unique indecomposable $H$-module $M$ such that
$x_{P_i}=f(u_1, \cdots, u_n)/m(u_1, \cdots, u_n)$, as a reduced fraction, where
$m(u_1, \cdots, u_n)$ is the monomial $u_1^{d_1}u_2^{d_2}\cdots u_n^{d_n}$
for $(d_1,\dots,d_n)$ the dimension vector $M$
and where $f$ is a polynomial in the $u_i$'s.

In case $Q$ is acyclic itself, then $H=kQ$ and the module $M$ for $x_{P_i}$ is just $P_i$. So
$x_{P_i}=f(x_1, \cdots, x_n)/x_1^{d_1}x_2^{d_2}\cdots x_n^{d_n}$ where
$(d_1, d_2, \cdots, d_n)$ is the dimension vector of the indecomposable projective $H$-module $P_i$. We will later use these expressions to calculate projective cluster variables.
\end{remark}

\begin{example} \label{ex2}
Let $\Sigma=(\mathbf{x}, Q)$, where $\mathbf{x}=\{x_1, x_2, x_3\}$ and $Q$ is the
quiver
$$
\hskip.5cm
\xymatrix{
&{2} \ar[dr]&\\
{1}\ar[ur]\ar@{<-}[rr]&&{3}
}
\hskip.5cm
$$
If we mutate at 2, we get an acyclic seed $(\mathbf{u},Q')=\mu_2(\mathbf{x},Q)$ where
$Q'$ is the quiver
$$
\xymatrix{
&{2} \ar[dl]&\\
{1}&&{3}\ar[ul]
}
$$
and $\mathbf{u}=\{u_1, u_2, u_3\}$ with
$u_1=x_1, u_2=\dfrac{x_1+x_3}{x_2}, u_3=x_3$.
Let $T=P_1\oplus P_2\oplus P_3$, where $P_i$ is an indecomposable projective $kQ'$-module, then $(T, Q')$ is a tilting seed corresponding to the seed $(\mathbf{u}, Q')$. Therefore $\delta_2(T, Q')=(T', Q)$ is a tilting seed corresponding to $(\mathbf{x}, Q)$, where $T'=P_1\oplus S_3\oplus P_3$ and $S_3$ is the simple $kQ'$-module corresponding to the vertex 3 of the quiver $Q'$. The cluster variable $x_{P_1}$ is
$\mu_1(u_1)$ in the seed $(\mathbf{u}, Q')$, and hence $x_{P_1}=\dfrac{1+u_2}{x_1}=\dfrac{1+\dfrac{x_1+x_3}{x_2}}{x_1}=\dfrac{x_1+x_2+x_3}{x_1x_2}$.  The variable $x_{P_2}$ is $\mu_3(u_3)$ in the
seed $(\mathbf{u}, Q')$, and hence $x_{P_2}=\dfrac{1+u_2}{x_3}=\dfrac{1+\dfrac{x_1+x_3}{x_2}}{x_3}=\dfrac{x_1+x_2+x_3}{x_2x_3}$. Lastly, $x_{P_3}$ is $\mu_3(u_3)$ in the
seed $\mu_2\mu_1(\mathbf{u}, Q')$, and one computes that
$x_{P_3}=\dfrac{x_1+x_2+x_3}{x_1x_3}$.
\end{example}

In the next step, we extend the definition of projective cluster variables to the case of cluster algebras
with coefficients, first for acyclic seeds and then for acyclic cluster algebras with coefficients. The
two parts (i) and (ii) of the definition specialise to the notions from Definition~\ref{def:proj-cl-var} and
\ref{def:proj-cl-var-acyclic} respectively, in the coefficient-free case.

\begin{definition}\label{def:proj-cl-var1}
Let $\Sigma=(\mathbf{x},\mathbf{p},Q)$ be a seed.
\begin{itemize}
\item[(i)]
Assume that $Q$ is an acyclic quiver. Let $(d_1,\dots, d_n)$ be the dimension vector
of the indecomposable projective $kQ$-module $P_i$. Then the {\em projective cluster variable}
$x_{P_i}$ {\em with respect to the seed $\Sigma$}
is defined to be the cluster variable in $\mathcal A(\mathbf{x},\mathbf{p},Q)$
which in its reduced form has the denominator $x_1^{d_1}x_2^{d_2}\cdots x_n^{d_n}$.
We do this for every $i=1,\dots,n$.
\item[(ii)]
Assume that $\Sigma=(\mathbf{x},\mathbf{p},Q)$ is mutation equivalent to an
acyclic seed. Then $(\mathbf{x}, Q)$ is also mutation equivalent to an acyclic seed. Let $x_{P_i'}$ be the projective cluster variables of $(\mathbf{x}, Q)$. There is a sequence of mutations of minimal length $\underline{\mu}=\mu_t\cdots \mu_1$ such that $x_{P_i'}$ is a cluster variable in the seed $\underline{\mu}(\mathbf{x},Q)$.
The {\em projective cluster variable $x_{P_i}$ with respect to $\Sigma$} is the corresponding cluster variable in
$\underline{\mu}(\Sigma)$. We do this for every $i=1,\dots,n$.
\end{itemize}
\end{definition}

In this paper, we will only deal with coefficient-free cases and with acyclic seeds in the case
with coefficients. We included part (ii) of the definition as a natural extension.

\begin{example} \label{ex3}
\begin{itemize}
\item[(i)]
Let $\Sigma=(\mathbf{x}, \mathbf{p}, Q)$, where $\mathbf{x}=\{x_1, x_2\}$, $\mathbf{p}=(p^{\pm}_1, p^{\pm}_2)$ and $Q$ is the quiver $$\hskip.5cm \xymatrix{
{2} \ar[r]&{1}.\\
}
\hskip.5cm$$
By definition $x_{P_1}=x'_1=\dfrac{p^{+}_1x_2+p^{-}_1}{x_1}$. Let $\mu_1(\Sigma)=\Sigma^{(1)}=(\mathbf{x}^{(1)}, \mathbf{p}^{(1)}, \mu_1(Q))$,
where $\mathbf{x}^{(1)}=\{x^{(1)}_1, x^{(1)}_2\}$ and
$\mathbf{p}=(p^{(1)\pm}_1, p^{(1)\pm}_2)$. Then by definition $x_{P_2}=\\\dfrac{p_{2}^{(1)+}(p^{+}_1x_2+p^{-}_1)+p_2^{(1)-}x_1}{x_1x_2}$.
\item[(ii)]
Let $\Sigma=(\mathbf{x}, \mathbf{p}, Q)$, where $\mathbf{x}=\{x_1, x_2, x_3\}$, $\mathbf{p}=(p^{\pm}_1, p^{\pm}_2, p^{\pm}_3)$ and $Q$ is the quiver $$\hskip.5cm \xymatrix{
&{2} \ar[dr]&\\
{1}\ar[ur]\ar@{<-}[rr]&&{3}
}
\hskip.5cm$$
We have $x'_1=x^{(1)}_1=\dfrac{p^{+}_1x_3+p^{-}_1x_2}{x_1}, x'_2=x^{(2)}_2=\dfrac{p^{+}_2x_1+p^{-}_2x_3}{x_2}$
and $x'_3=x^{(3)}_3=\dfrac{p^{+}_3x_2+p^{-}_3x_1}{x_3}$.
Let $\mu_2\mu_1(\Sigma)=\Sigma^{(12)}=(\mathbf{x}^{(12)}, \mathbf{p}^{(12)}, \mu_2\mu_1(Q))$,
where $\mathbf{x}^{(12)}=\\\{x^{(12)}_1, x^{(12)}_2, x^{(12)}_3\}$ and
$\mathbf{p}=(p^{(12)\pm}_1, p^{(12)\pm}_2, p^{(12)\pm}_3)$. By definition
$x_{P_1}=x_2^{(12)}=\dfrac{p_{2}^{(1)+}x_1+p_2^{(1)-}(p^{+}_1x_3+p^{-}_1x_2)}{x_1x_2}$,
$x_{P_2}=x_2^{(32)}=\dfrac{p_{2}^{(3)+}(p^{+}_3x_2+p^{-}_3x_1)+p_2^{(3)-}x_3}{x_2x_3}$ and
$x_{P_3}=x_3^{(13)}=\dfrac{p_{3}^{(1)+}(p^{+}_1x_3+p^{-}_1x_2)+p_3^{(1)-}x_1}{x_3x_1}$.
\end{itemize}
\end{example}

\begin{definition}\label{def:lower-bound-proj}
Let $\Sigma=(\mathbf{x}, \mathbf{p}, Q)$ be a seed of an acyclic cluster algebra, with $\mathbf{x}=\{x_1,\dots, x_n\}$.
The {\em lower bound cluster algebra generated by projective cluster variables
$\mathcal{L}_P(\Sigma)$ associated with $\Sigma$
is defined as} $\mathcal{L}_P(\Sigma)=\mathbb{ZP}[x_1, x_{P_1}, \cdots, x_n, x_{P_n}]$,
where the $x_{P_i}$ are the projective cluster variables from Definition~\ref{def:proj-cl-var1}.
\end{definition}

In the coefficient-free case, the lower bound cluster algebra generated by projective cluster variables
is defined accordingly using Definitions~\ref{def:proj-cl-var} or~\ref{def:proj-cl-var-acyclic}.
By definition, $\mathcal{L}_P(\Sigma)$ is a finitely generated subalgebra of the
cluster algebra $\mathcal{A}(\Sigma)$.

\subsection{Related work}

The notion of projective cluster variables has potential applications to green-to-red sequences and in the context of
PBW basis for a quantum unipotent subgroup. We explain this briefly.

(1) If there exists a green-to-red sequence for the
cluster algebra $\mathcal A(\Sigma)$, the projective cluster variables are the same as the cluster
variables in the ``all-red'' seed: In unpublished work, in the acyclic case,
Greg Muller has found a version of Theorem~\ref{thm1}. In the language of green-to-red sequences,
the lower bound algebra (in the acyclic case) can be generated
the cluster variables in a maximal green sequence, obtained by mutating at a sequence of sources. This is then
precisely the initial cluster variables with the projective cluster variables.
Motivated by this, Greg Muller posed the question whether the cluster variables in a maximal green
sequence always generate a cluster algebra.

(2)
Fan Qin pointed out to us that for acyclic seed with $z$-pattern (in the language of \cite{KQ}),
the standard monomial basis we introduce should agree with an associated dual PBW basis;
One chooses appropriated geometric coefficients for an acyclic seed such that the
(quantized) cluster algebra
(where the unfrozen variables are not invertible) corresponds to the quantum unipotent subgroup
$\mathcal A_q [c^2]$
associated to the Weyl group element $c^2$, where $c$ is a Coxeter word.

Fan Qin also points out that the standard monomials we introduce - or a version with injectives, in the convention of~\cite{KQ}
are important for defining the triangular basis. These standard monomials do not necessarily form a basis but form
a topological basis in the topology given in the work~\cite[\S 2.2.2]{DM} of Davison and Mandel.

\section{Cluster algebras with initial acyclic seeds}\label{sec 3}

In this section we work in the setting of cluster algebras with coefficients. We
show that for any acyclic seed $\Sigma=(\mathbf{x}, \mathbf{p}, Q)$,
the lower bound cluster algebra generated by projective cluster variables $\mathcal{L}_P(\Sigma)$
and the cluster algebra $\mathcal{A}(\Sigma)$ coincide.

\begin{lemma}\label{lemm1}
Let $\Sigma=(\mathbf{x}, \mathbf{p}, Q)$ be an acyclic seed, let $\{1,2,\cdots,n\}$ be the vertices of $Q$.
Let $v$ be a source in $Q$,
$v_1, \cdots, v_t$ be direct successors of the vertex $v$ and $a_i$ be the number of arrows from $v$ to $v_i$ in $Q$. Then $x'_v=b(x_{v_1}^{a_1}\cdots x_{v_t}^{a_t}x_{P_v}-f)$, where $b\in \mathbb{ZP}$, $f\in \mathbb{ZP}[x_v, x_{v_1}, \cdots, x_{v_t}, x_{P_{v_1}}, x_{P_{v_2}}, \cdots, x_{p_n}]$.
\end{lemma}

\begin{proof} We label the vertices of $Q$ from
$1$ to $n$ such that if there exists a path from $i$ to $j$, then $i> j$.

Locally, the quiver $Q$ is of the form
\begin{center}
$
\includegraphics[width=3cm]{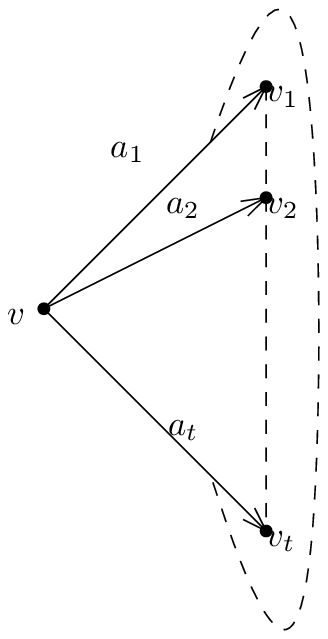}
$
\end{center}
(with $v>v_1>v_2>\cdots v_t$)
and we have $x'_v=\dfrac{p_v^{+}+p_v^{-}x_{v_1}^{a_1}x_{v_2}^{a_2}\cdots x_{v_t}^{a_t}}{x_v}$.
Consider one of these vertices $v_i$, $1\leq i\leq t$,
let $v, v_{i_1}, \cdots, v_{i_{s_i}}$ be the direct predecessors of $v_i$ and
$w_{i_1}, \cdots, w_{i_{t_i}}$ the direct successors of $v_i$. Label the $w_{i_j}$ in such a way that
$\{w_{i_1}, \cdots, w_{i_{r_i}}\}\subseteq \{v_{i+1}, \cdots, v_{t}\}$ for some $r_i$ (possibly, there are no such),
let $a_{i_j}$ be the number of arrows from $v_{i_j}$ to $v_i$ in $Q$ and
$b_{i_j}$ be the number of arrows from $v_i$ to $w_{i_j}$ in $Q$.
Since $Q$ is acyclic, there is a finite set of mutations such that after applying these mutations to the
seed $\Sigma=(\mathbf{x}, \mathbf{p}, Q)$, we obtain a seed
$\Sigma^{(i+1)}=(\mathbf{x}^{(i+1)}, \mathbf{p}^{(i+1)}, Q^{(i+1)})$ where
$\mathbf{x}^{(i+1)}=\lbrace x^{(i+1)}_1, x^{(i+1)}_2, \cdots, x^{(i+1)}_n\rbrace$, with
$x^{(i+1)}_v=x_v$, $x^{(i+1)}_{w_{i_j}}=x_{P_{w_{i_j}}}$ for each $1\leq j\leq t_i$,
$x^{(i+1)}_{v_{i_j}}=x_{v_{i_j}}$ for each $1\leq j\leq s_i$,
$\mathbf{p}^{(i+1)}=(p^{(i+1)\pm}_1, \cdots,p^{(i+1)\pm}_n)$ and $Q^{(i+1)}$ (locally)
is the following quiver (all successor variables for $i$ have changed, all predecessor variables remain
unchanged)
$$
\includegraphics[width=5cm]{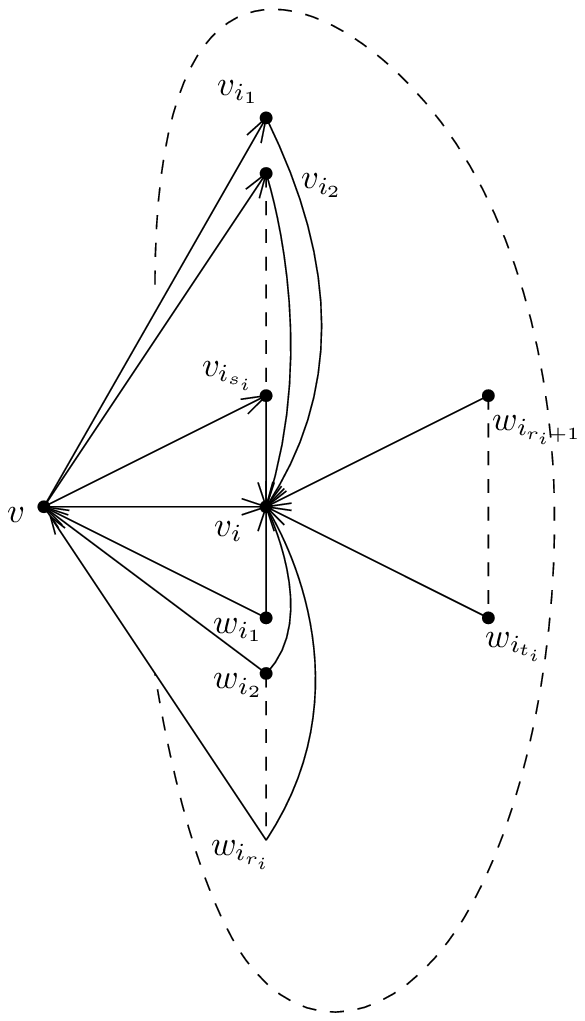}
$$
Also, $p^{(i+1)+}_v/p^{(i+1)-}_v=(p^{-}_{w_{i_{r_i}}})^{-c_{i_{r_i}}}\cdots (p^{-}_{w_{i_{1}}})^{-c_{i_{1}}} p^{+}_i/p^{-}_i$,
where $c_{i_j}$ is the number of arrows from $v$ to $w_{i_j}$.
If we then mutate at the vertex $v_i$, we get
\begin{equation}\label{eq1}
x_{P_{v_i}}=\dfrac{p^{(i+1)+}_{v_i}x_v^{a_i}x_{v_{i_1}}^{a_{i_1}}\cdots x_{v_{i_{s_i}}}^{a_{i_{s_i}}}x_{P_{w_{i_1}}}^{b_{i_1}}\cdots x_{P_{w_{i_{t_i}}}}^{b_{i_{t_i}}}+p^{(i+1)-}_{v_i}}{x_{v_i}},
\end{equation}
where
$p^{(i+1)+}_{v_i}/p^{(i+1)-}_{v_i}=(p^{-}_{w_{i_{t_i}}})^{-b_{i_{t_i}}}\cdots
(p^{-}_{w_{i_{1}}})^{-b_{i_{1}}} p^{+}_{v_i}/p^{-}_{v_i}$.

We now apply this process to all the vertices $v_t,\cdots, v_1$ incident with the source $v$, starting
with $v_t$.
So first we get the seed $\Sigma^{(t+1)}=(\mathbf{x}^{(t+1)}, \mathbf{p}^{(t+1)}, Q^{(t+1)})$.

Then we mutate at the vertex $v_t$ and obtain the seed
$\Sigma^{(t)}=(\mathbf{x}^{(t)}, \mathbf{p}^{(t)}, Q^{(t)})$.
Now we mutate at the vertex $v_{t-1}$ and get the seed $\Sigma^{(t-1)}=(\mathbf{x}^{(t-1)}, \mathbf{p}^{(t-1)}, Q^{(t-1)})$. After applying this process for all the vertices $v_t, \cdots, v_1$ we obtain the
seed $\Sigma^{(1)}=(\mathbf{x}^{(1)}, \mathbf{p}^{(1)}, Q^{(1)})$ where
$\mathbf{x}^{(1)}=\lbrace x^{(1)}_1, x^{(1)}_2, \cdots, x^{(1)}_n\rbrace$, $x^{(1)}_v=x_v$, $x^{(1)}_{v_{i}}=x_{P_{v_{i}}}$ for each $1\leq i\leq t$, $\mathbf{p}^{(1)}=(p^{(1)\pm}_1, \cdots,p^{(1)\pm}_n)$ and $Q^{(1)}$
(locally) is the following quiver.
 $$\includegraphics[width=3.5cm]{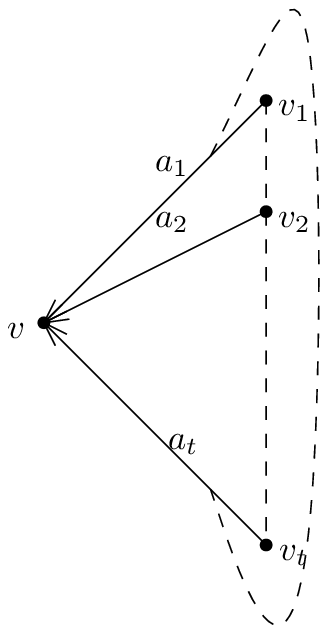}$$
If we mutate $\Sigma^{(1)}=(\mathbf{x}^{(1)}, \mathbf{p}^{(1)}, Q^{(1)})$ at the vertex $v$ we have $x_{P_v}=\dfrac{p_{v}^{(1)+}x_{P_{v_1}}^{a_{1}}\cdots x_{P_{v_t}}^{a_t}+p^{(1)-}_{v}}{x_{v}}$. By using \eqref{eq1} we have $x_{v_1}^{a_1}\cdots x_{v_t}^{a_t}x_{P_v}=\dfrac{p_{v}^{(1)+}(p_{v_1}^{(2)-})^{a_{1}}\cdots(p_{v_t}^{(t+1)-})^{a_t}+p_{v}^{(1)-}x_{v_1}^{a_1}\cdots x_{v_t}^{a_t}}{x_v}+f$, where $f\in \mathbb{ZP}[x_v, x_{v_1}, \cdots, x_{v_t}, x_{P_{v_1}}, x_{P_{v_2}}, \cdots, x_{p_n}]$. Also we have  $p_{v}^{(1)+}/p_{v}^{(1)-}=\\(p_{v_1}^{(2)-})^{-a_{1}}\cdots(p_{v_t}^{(t+1)-})^{-a_t}p_{v}^{+}/p_{v}^{-}$. Therefore $x_{v_1}^{a_1}\cdots x_{v_t}^{a_t}x_{P_v}=(\dfrac{p_{v}^{(1)-}}{p_{v}^{-}})\dfrac{p_v^{+}+p_v^{-}x_{v_1}^{a_1}x_{v_2}^{a_2}\cdots x_{v_t}^{a_t}}{x_v}+f$. Thus $x'_v=\dfrac{p_{v}^{-}}{p_{v}^{(1)-}}(x_{v_1}^{a_1}\cdots x_{v_t}^{a_t}x_{P_v}-f)$ and the result follows.
\end{proof}

\begin{theorem}\label{thm1}
Let $\Sigma=(\mathbf{x}, \mathbf{p}, Q)$ be an acyclic seed.
Then $\mathcal{L}_P(\Sigma)=\mathcal{A}(\Sigma)$.
\end{theorem}

\begin{proof} We use induction on the rank of $\mathcal{A}(\Sigma)$.
Let $\mathcal{A}(\Sigma)$ be a cluster algebra of rank 2. Then $Q$ is a quiver as follows
$$
\includegraphics[width=3cm]{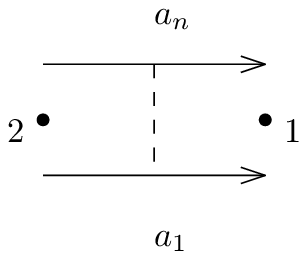}
$$
We have $x_{P_1}=x'_1=\dfrac{p_1^{+}x_2^{n}+p_1^{-}}{x_1}$,
$x'_2=\dfrac{p_2^{+}+p_2^{-}x_1^{n}}{x_2}$ and
$x_{P_2}=\dfrac{p_2'^{+}(\dfrac{p_1^{+}x_2^{n}+p_1^{-}}{x_1})^n+p_2'^{-}}{x_2}$.
Since by Theorem 1.20 of \cite{BFZ}, $\mathcal{A}(\Sigma)$ is generated
by $x_1, x_2, x'_1$ and $x'_2$, it is enough to show that $x'_2$ is contained in
$\mathcal{L}_P(\Sigma)$.

Now $x_1^nx_{P_2}=\dfrac{p_2'^{+}(p_1^{-})^n+p_2'^{-}x_1^{n}}{x_2}+f$,
where $f\in \mathcal{L}_P(\Sigma)$. Since
$p_2'^{+}(p_1^{-})^n+p_2'^{-}x_1^{n}=\dfrac{p_2'^{-}}{p_2^{-}}(p_2^{+}+p_2^{-}x_1^n)$ (using
the mutation formula for coefficients from Section~\ref{sec:background}),
we get  $x'_2=\dfrac{p_2^{-}}{p_2'^{-}}x_1^nx_{P_2}-\dfrac{p_2^{-}}{p_2'^{-}}f\in \mathcal{L}_P(\Sigma)$.
Therefore in this case, $\mathcal{L}_P(\Sigma)=\mathcal{A}(\Sigma)$.

Now assume that $\mathcal{A}(\Sigma)$ is of rank $n\geq 3$. We label the vertices of $Q$ from
$1$ to $n$ such that, if there exists a path from $i$ to $j$, then $i> j$. In particular, $n$ is a source.
Let $n-1, \cdots, n-t$ be direct successors of the vertex $n$, for some $t>0$,
and for $i=1,\dots, t$ let $a_i$ be the number of arrows from $n$ to $n-i$ in $Q$.
The quiver $Q$ looks as follows
\begin{center}
\includegraphics[width=5cm]{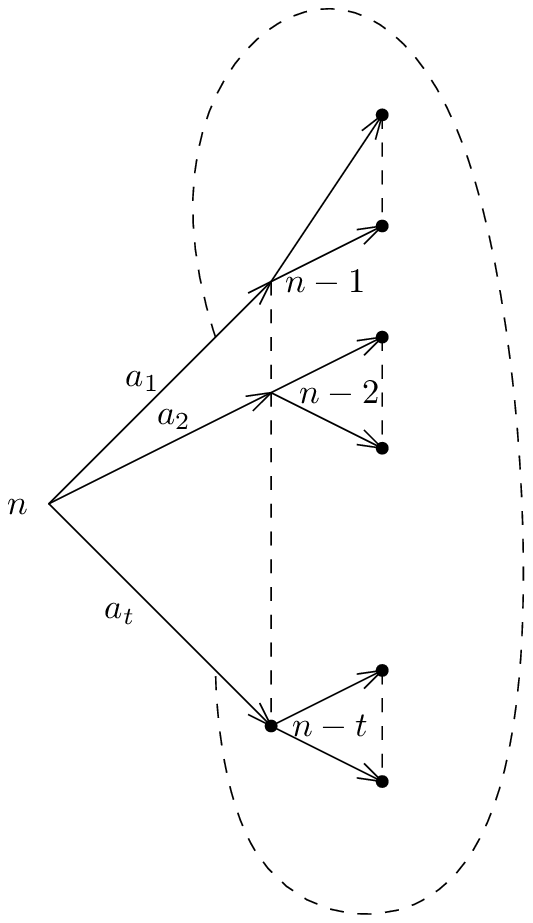}
\end{center}
By Lemma \ref{lemm1},
$x'_n=\dfrac{p_n^{+}+p_n^{-}x_{n-1}^{a_1}x_{n-2}^{a_2}\cdots x_{n-t}^{a_t}}{x_n}\in \mathcal{L}_P(\Sigma)$.
Now we delete the vertex $n$ and all arrows incident with it. Then we have a seed
$\widetilde{\Sigma}=(\tilde{\mathbf{x}}, \tilde{\mathbf{p}}, \tilde{Q})$,
where $\tilde{\mathbf{x}}=\lbrace \tilde{x}_1, \tilde{x}_2, \cdots, \tilde{x}_{n-1}\rbrace$ and
$\tilde{\mathbf{p}}=(\tilde{p}^{\pm}_1, \cdots,\tilde{p}^{\pm}_{n-1})$ with $\tilde{x}_i=x_i$ and
$\tilde{p}^{\pm}_{i}=p^{\pm}_{i}$ for $i=1,\dots, n-1$.
The rank of the associated cluster algebra $\mathcal{A}(\widetilde{\Sigma})$ is $n-1$ and
so by the induction
assumption $\mathcal{A}(\widetilde{\Sigma})=\mathcal{L}_P(\widetilde{\Sigma})$. Since $Q$ is acyclic,
for each $i>n-t$, $\tilde{x}'_i=x'_i$ and $\tilde{x}_{P_i}=x_{P_i}$. Therefore for each $i>n-t$,
$x'_i\in \mathcal{L}_P(\Sigma)$. It is enough to show that $x'_{i}\in \mathcal{L}_P(\Sigma)$ for each
$n-1\leq i\leq n-t$. Put $v=n$ and $v_i=n-i$ for each $1\leq i\leq t$. By the proof of the
Lemma~\ref{lemm1} we have
$x_{P_{v_i}}=\dfrac{p^{(i+1)+}_{v_i}x_v^{a_i}x_{v_{i_1}}^{a_{i_1}}\cdots
x_{v_{i_{s_i}}}^{a_{i_{s_i}}}x_{P_{w_{i_1}}}^{b_{i_1}}\cdots
x_{P_{w_{i_{t_i}}}}^{b_{i_{t_i}}}+p^{(i+1)-}_{v_i}}{x_{v_i}}$.
Also $x'_{v_i}=\dfrac{p^{+}_{v_i}x_v^{a_i}x_{v_{i_1}}^{a_{i_1}}\cdots
x_{v_{i_{s_i}}}^{a_{i_{s_i}}}+p^{-}_{v_i}x_{w_{i_1}}^{b_{i_1}}\cdots x_{w_{i_{t_i}}}^{b_{i_{t_i}}}}{x_{v_i}}$.
By the same argument as in the proof of the Lemma \ref{lemm1} we can compute
$x_{P_{w_{i_1}}}, \cdots, x_{P_{w_{i_{t_i}}}}$. Then the same argument as in the proof of the
Lemma~\ref{lemm1} shows that $x'_{i}\in \mathcal{L}_P(\Sigma)$.
\end{proof}

Combining Theorem 1.20 of \cite{BFZ} and Theorem \ref{thm1}, we obtain the following corollary.

\begin{corollary}\label{cor1} Let $\Sigma=(\mathbf{x}, \mathbf{p}, Q)$ be a seed. Then $\mathcal{L}_P(\Sigma)=\mathcal{A}(\Sigma)=\mathcal{L}(\Sigma)$ if and only if $Q$ is acyclic.
\end{corollary}

The converse to Theorem~\ref{thm1} is not true in general as the following example shows:
It is an example of a
cluster algebra given by a cyclic seed $\Sigma$ where $\mathcal{L}_P(\Sigma)=\mathcal{A}(\Sigma)$.

\begin{example} \label{ex:cyclic}
Let $\Sigma=(\mathbf{x}, \mathbf{p}, Q)$, where $\mathbf{x}=\{x_1, x_2, x_3\}$, $\mathbf{p}=(p^{\pm}_1, p^{\pm}_2, p^{\pm}_3)$ and $Q$ is the quiver $$\hskip.5cm \xymatrix{
&{2} \ar[dr]^{\beta}&\\
{1}\ar[ur]^{\alpha}\ar@{<-}[rr]_{\gamma}&&{3}
}
\hskip.5cm$$
We have $x'_1=x^{(1)}_1=\dfrac{p^{+}_1x_3+p^{-}_1x_2}{x_1}, x'_2=x^{(2)}_2=\dfrac{p^{+}_2x_1+p^{-}_2x_3}{x_2}$
and $x'_3=x^{(3)}_3=\dfrac{p^{+}_3x_2+p^{-}_3x_1}{x_3}$. By part $(ii)$ of Example \ref{ex3} we have
$x_{P_1}=\dfrac{p_{2}^{(1)+}x_1+p_2^{(1)-}(p^{+}_1x_3+p^{-}_1x_2)}{x_1x_2}$, $x_{P_2}=\\\dfrac{p_{2}^{(3)+}(p^{+}_3x_2+p^{-}_3x_1)+p_2^{(3)-}x_3}{x_2x_3}$ and $x_{P_3}=\dfrac{p_{3}^{(1)+}(p^{+}_1x_3+p^{-}_1x_2)+p_3^{(1)-}x_1}{x_3x_1}$.
Then $x_2x_{P_1}=p_{2}^{(1)+}+p_{2}^{(1)-}x'_1$ and so $x'_1\in \mathcal{L}_P(\Sigma)$. Also $x_3x_{P_2}=p_{2}^{(3)+}p^{+}_3+\dfrac{p_{2}^{(3)-}}{p_{2}^{-}}x'_2$ and hence
$x'_2\in \mathcal{L}_P(\Sigma)$.
Lastly, $x_1x_{P_3}=p_{3}^{(1)+}p^{+}_1+\dfrac{p_{3}^{(1)-}}{p_{3}^{-}}x'_3$ and hence
$x'_3\in \mathcal{L}_P(\Sigma)$.
The cluster algebra $\mathcal{A}(\Sigma)$ is the $\mathbb{ZP}$-subalgebra of the field of
rational functions in the three independent variables $x_1,x_2,x_3$
with coefficients in $\mathbb{ZP}$ generated by the cluster
variables $x_1, x_2, x_3, x'_1, x'_2, x'_3, x_{P_1}, x_{P_2}, x_{P_3}$. Since the $x_i'$ all belong to
$\mathcal{L}_P(\Sigma)$, we obtain
$\mathcal{L}_P(\Sigma)=\mathcal{A}(\Sigma)$. Note that by Theorem 1.20 of \cite{BFZ},
$\mathcal{L}(\Sigma)
\subsetneq \mathcal{A}(\Sigma)$.
\end{example}

Let $\Sigma=(\mathbf{x}, \mathbf{p}, Q)$ be a seed. Recall that a monomial in
$x_1, x'_1, \cdots, x_n, x'_n$ is standard if it contains no product of the form $x_jx'_j$ \cite{BFZ}.
The standard monomials span $\mathcal{L}(\Sigma)$ as a $\mathbb{ZP}$-module.
Berenstein, Fomin and Zelevinsky proved that the standard monomials in $x_1, x'_1, \cdots, x_n, x'_n$
are linearly independent over $\mathbb{ZP}$ if and only if $\Sigma$ is acyclic \cite{BFZ}.
Let $\Sigma=(\mathbf{x}, \mathbf{p}, Q)$ be a seed.

\begin{definition}
We call a monomial {\em standard in}
$x_1, x_{P_1}, \cdots, x_n, x_{P_n}$ if it contains no product of the form $x_jx_{P_j}$.
\end{definition}

If $\Sigma=(\mathbf{x}, \mathbf{p}, Q)$ is an acyclic seed, then it is easy to see that for any $j\in Q_0$, $x_jx_{P_j}=p_{j}^{+}\prod_{i\rightarrow j}x_i\prod_{j\rightarrow l}x_{P_l}+p_{j}^{-}$. Therefore in this case standard monomials in $x_1, x_{P_1}, \cdots, x_n, x_{P_n}$ span $\mathcal{L}_P(\Sigma)=\mathbb{ZP}[x_1, x_{P_1}, \cdots, x_n, x_{P_n}]$ as a $\mathbb{ZP}$-module.

\begin{theorem} \label{thm:standard}
Let $\Sigma=(\mathbf{x}, \mathbf{p}, Q)$ be an acyclic seed. Then the standard monomials in $x_1, x_{P_1}, \cdots, x_n, x_{P_n}$ form a $\mathbb{ZP}$-basis of $\mathcal{L}_P(\Sigma)$.
\end{theorem}

\begin{proof} We use the same idea as in \cite[proof of Theorem 1.16]{BFZ}. Since $Q$ is acyclic, we can assume that for any arrow $j\rightarrow i$ in $Q$, $i>j$. Label standard monomials in $x_1, x_{P_1}, \cdots, x_n, x_{P_n}$ by the points $\mathbf{m}=(m_1, \cdots, m_n)\in \mathbb{Z}^n$, where $\mathbf{x}^{<\mathbf{m}>}=x_1^{<m_1>}\cdots x_n^{<m_n>}$ and \begin{align*}
x_j^{<m_j>} = \begin{cases}
x_j^{m_j} & \text{if }\ m_j\geq 0 ,\\
(x_{P_j})^{-m_j} & \text{if }\ m_j< 0.\
\end{cases}
\end{align*}
Let $\prec$ be the lexicographic order in $\mathbb{Z}^n$. For $\mathbf{m}, \mathbf{n}\in \mathbb{Z}^n$, $\mathbf{m}\prec\mathbf{n}$ if the first nonzero difference $m_j-n_j$ is positive. We order both the Laurent monomials $\mathbf{x}^{\mathbf{m}}=x_1^{m_1}\cdots x_n^{m_n}$ and the standard monomials $\mathbf{x}^{<\mathbf{m}>}=x_1^{<m_1>}\cdots x_n^{<m_n>}$ lexicographically. For each $m_j<0$, $x_j^{<m_j>}=(x_{P_j})^{-m_j}=x_j^{m_j}(p_{j}^{+}\prod_{i\rightarrow j}x_i\prod_{j\rightarrow l}x_{P_l}+p_{j}^{-})^{-m_j}$. This implies that if $\mathbf{m}\prec\mathbf{m'}$, then the first monomial in $\mathbf{x}^{<\mathbf{m}>}$ precedes the first monomial in $\mathbf{x}^{<\mathbf{m}'>}$. Therefore the standard monomials in $x_1, x_{P_1}, \cdots, x_n, x_{P_n}$ are linearly independent over $\mathbb{ZP}$.
\end{proof}

It is not clear whether the converse of Theorem~\ref{thm:standard} holds:
Let $\Sigma=(\mathbf{x}, \mathbf{p}, Q)$ be a seed and assume that the standard monomials in
$x_1, x_{P_1}, \cdots, x_n, x_{P_n}$ form a $\mathbb{ZP}$-basis of $\mathcal{L}_P(\Sigma)$.
Then we do not know whether $\Sigma$ is an acyclic seed.

\section{Coefficient-free acyclic cluster algebras}\label{sec 4}

Let $\mathcal{A}(\mathbf{x}, Q)$ be an acyclic coefficient-free cluster algebra such that
$\mathcal{A}(\mathbf{x},Q)=\mathcal L_P(\Sigma)$. In this section we provide a generating set of
cluster variables for the cluster algebra $\mathcal{A}(\mathbf{x}, \mu_j(Q))$, for any $1\le j\le n$.
This is a crucial result in the rest of the paper.

Let $Q$ be an acyclic quiver with $n$ vertices and $H=kQ$ its finite-dimensional hereditary $k$-algebra. Let $\mathcal{C}_H$ be the corresponding cluster category, $T=T_1\oplus\cdots\oplus T_n$ be a
basic\footnote{i.e. for $i\ne j$, the summands $T_i$ and $T_j$ are not isomorphic}
cluster-tiling object of
$\mathcal{C}_H$ and $B=\End_{\mathcal{C}_H}(T)^{op}$. The functor $F:\mathcal{C}_H\rightarrow \modd(B)$ given by $X\rightarrow \mathcal{C}_H(T, X)$ induces an equivalence of
categories $\mathcal{C}_H/(T[1])\rightarrow \modd(B)$, where $[1]$ denotes the shift
functor and $(T[1])$ denotes the ideal of morphisms in $\mathcal{C}_H$ which factor through a direct
sum of copies of $T[1]$ (\cite{BMR1}, \cite{KR}). Let $N\in \modd(B)$, $e\in K_0(\modd(B))$ and $\Gr_e(N)$ the variety of submodules $N'$ of $N$ whose class in $K_0(\modd(B))$
is $e$. Then $\Gr_e(N)$ is a closed projective subvariety of the classical Grassmannian
of subspaces of $N$. Let $\chi(\Gr_e(N))$ denote its Euler-Poincar\'{e} characteristic with
respect to the \'{e}tale cohomology with proper support. Let $K^{sp}_0(\modd(B))$ be the quotient of the free abelian group on
the set of isomorphism classes $[N]$ of finite-dimensional $B$-modules $N$, modulo the
subgroup generated by all elements $[N_1\oplus N_2]-[N_1]-[N_2]$, called the
split Grothendieck group of $\modd(B)$.
Palu in \cite{Pa} defined a bilinear form on $K^{sp}_0(\modd(B))$ as follows
$$
< , >:K^{sp}_0(\modd(B))\times K^{sp}_0(\modd(B))\rightarrow \mathbb{Z}
$$
$$
<N, N'>=[N, N']-^1[N,N'],
$$
where $N, N'\in \modd(B)$, $[N, N']=dim_k\mathcal{C}_H(N, N')$ and
$^1[N, N']=dim_k\mathcal{C}_H(N, N'[1])$.
He used this to define an antisymmetric bilinear form on it:
$$
< , >_a:K^{sp}_0(\modd(B))\times K^{sp}_0(\modd(B))\rightarrow \mathbb{Z}
$$
$$
<N, N'>_a=<N, N'>-<N', N>.
$$
Let $\ind\mathcal{C}_H$ be a set of representatives for the isomorphism classes of indecomposable objects of $\mathcal{C}_H$ and let $S_i$ be
the top of the projective $B$-module $P_i = F(T_i)$ for each $1\leq i\leq n$.
With this, Palu defined a {\em Caldero-Chapoton map}
(see also \cite{CK}),
$X^T_{?}: \ind\mathcal{C}_H \rightarrow \mathbb{Q}(x_1, \cdots, x_n)$ by
\begin{align*}
X^T_M = \begin{cases}
x_i & \text{if }\ M\cong T_i[1] ,\\
\sum_e\chi(Gr_eFM)\Pi_{i=1}^{n}x_i^{<S_i, e>_a-<S_i, F(M)>}& \text{otherwise. }\
\end{cases}
\end{align*}
He showed that the map $X^T_{?}$ is a cluster character \cite[Theorem 1.4]{Pa}.

Let $\mathcal{T}$ be a set of representatives for the isomorphism classes of indecomposable rigid
objects of $\mathcal{C}_H$.
For any basic cluster-tilting object $T$ in $\mathcal{C}_H$, let $Q_T$ be
the quiver of $\End_{\mathcal{C}_H}(T)^{op}$. Then $X^T_{?}$ induces a bijection from the set $\mathcal{T}$ to the set of cluster variables of the associated cluster algebra $\mathcal{A}_{Q_T}:=\mathcal{A}(\mathbf{x}, Q_T)$,
sending basic
cluster-tilting objects to clusters, \cite[Corollary 5.4]{Pa}.
Furthermore, there is a commutative diagram
\begin{equation}\label{eq:diagram-Palu}
\xymatrix{
&{\mathcal{T}} \ar[dr]^{X^T_{?}} \ar[dl]_{X_{?}}&\\
{\mathcal{A}_{Q}}&&{\mathcal{A}_{Q_T}}\ar@{->}[ll]_{f}
}
\end{equation}
where $X_{?}$ is the Caldero-Chapoton map \cite{CC} and
$f$ is the isomorphism of the cluster algebras $\mathcal{A}_{Q}:=\mathcal{A}(\mathbf{x}, Q)$
and $\mathcal{A}_{Q_T}$ defined as follows:
There exists $i_1, i_2, \cdots, i_t$ such that $Q_T=\mu_{i_1}\cdots \mu_{i_t}(Q)$. Then $f$ sends $x_j$ to $y_j=\mu_{i_t}\cdots \mu_{i_1}(x_j)$, $j=1,\dots,n$, and send any other cluster variable obtained by a sequence of mutations, say $\mu_{l_1}\cdots \mu_{l_m}(x_j)$, to the cluster variable $\mu_{l_1}\cdots \mu_{l_m}(y_j)$.

\begin{remark}\label{rem:induced-isom}
If $(\{x_1,\dots, x_n\},Q)$ and $(\{y_1,\dots, y_n\},Q)$ are seeds with the same quiver, then
the map sending $x_i$ to $y_i$ and any other cluster variable obtained by a sequence
$\mu_{i_r}\dots\mu_{i_1}$ from $(\mathbf{x},Q)$ to the corresponding cluster variable obtained
from applying the sequence $\mu_{i_r}\dots\mu_{i_1}$ to $(\mathbf{y},Q)$ gives an isomorphism of
cluster algebras $g:\mathcal A(\mathbf{x},Q)\to\mathcal A(\mathbf{y},Q)$.
This isomorphism is induced from the isomorphism
$\mathbb{Q}(x_1,\dots, x_n)\cong \mathbb{Q}(y_1,\dots, y_n)$.
\end{remark}

\begin{theorem}\label{thm2} Let $\mathcal{A}(\mathbf{x}, Q)$ be an acyclic cluster algebra
with cluster $\mathbf{x}=\{x_1, x_2, \cdots, x_n\}$ and seed $\Sigma=(\mathbf{x},Q)$.
Assume that $\mathcal{A}(\mathbf{x},Q)=\mathcal L_P(\Sigma)$.
Let $1\le j\le n$ be arbitrary and set $\mathbf{x}'=\mathbf{x}\cup \{x_j'\}\backslash \{x_j\}$.
Then
$\mathcal{A}(\mathbf{x}, \mu_j(Q))$ is generated by the cluster variables in $\mathbf{x}'$ and
by the variables
$x_{P'_1}, \cdots, x_{P'_{j-1}},g(x_{P_j}), x_{P'_{j+1}}, \cdots, x_{P'_n}$,
where $g:\mathcal{A}(\mathbf{x}, Q)\rightarrow \mathcal{A}(\mathbf{x}', Q)$ is
the isomorphism of cluster algebras from Remark~\ref{rem:induced-isom}.
\end{theorem}

\begin{proof}
Since $\mathcal{A}(\mathbf{x}, Q)$ is an acyclic cluster algebra, there exists an acyclic seed
$\Sigma=(\mathbf{u}, Q')$ such that $\mathcal{A}(\mathbf{x}, Q)= \mathcal{A}(\mathbf{u}, Q')$.
The cluster algebra $\mathcal{A}(\mathbf{x}, \mu_j(Q))$ is equal to the cluster algebra
$\mathcal{A}(\mathbf{x}', Q)$, where $\mathbf{x}'=\mathbf{x}\cup \{x_j'\}\backslash \{x_j\}$.
Let $(i_1, \cdots, i_t)$ be an ordered sequence of minimal length such that
$(\mathbf{x}, Q)=\mu_{i_t}\cdots \mu_{i_1}(\mathbf{u}, Q')$.
Let $(T', Q_{T'})$ be the tilting seed obtained from $(\mathbf{x},Q)$ through this sequence,
i.e. $\alpha((\mathbf{x}, Q), (i_1, \cdots, i_t))=(T', Q_{T'})$.

Then $\alpha(\mu_j(\mathbf{x}, Q), (i_1, \cdots, i_t, j))=(T'', Q_{T''})$,
where $T''$ is a cluster-tilting object of $\mathcal{C}_H$ obtained by exchanging $T'_j$
with $T''_j=T_j^*$,  by [6, Theorem 6.1].
By the same result, the following diagram commutes:
$$
\xymatrix{
{((\mathbf{x}, Q), (i_1, \cdots, i_t))}\ar[d]_{\mu_j} \ar[r]^-{\alpha}&{(T', Q_{T'})}\ar[d]^{\delta_j}\\
{((\mu_j(\mathbf{x}), \mu_j(Q)), (i_1, \cdots, i_t, j))}\ar[r]_-{\alpha}&{(T'', Q_{T''})}\\
}
$$
We have the following commutative diagrams, see (\ref{eq:diagram-Palu}):
$$\xymatrix{
&{\mathcal{T}} \ar[dr]^{X^{T'}_{?}} \ar[dl]_{X_{?}}&\\
{\mathcal{A}_{Q'}}&&{\mathcal{A}_{Q_{T'}}}\ar@{->}[ll]_{f_1}
}\hskip1cm
\xymatrix{
&{\mathcal{T}} \ar[dr]^{X^{T''}_{?}} \ar[dl]_{X_{?}}&\\
{\mathcal{A}_{Q'}}&&{\mathcal{A}_{Q_{T''}}}\ar@{->}[ll]_{f_2}
}$$
where $f_1, f_2$ are isomorphisms of cluster algebras.
Since $Q_{T'}=\mu_{i_t}\cdots \mu_{i_1}(Q')$, $f_1$ sends $x_l$ to
$y_l=\mu_{i_1}\cdots \mu_{i_t}(x_l)$ and for any sequence of mutations
$\mu_{l_1}\cdots \mu_{l_m}$, the cluster variable $\mu_{l_1}\cdots \mu_{l_m}(x_l)$
is sent to $\mu_{l_1}\cdots \mu_{l_m}(y_l)$. In the same way, since $Q_{T''}=\mu_j\mu_{i_t}\cdots \mu_{i_1}(Q')$, the isomorphism $f_2$ sends $x_l$ to $y_l=\mu_{i_1}\cdots \mu_{i_t}\mu_j(x_l)$ it sends $\mu_{l_1}\cdots \mu_{l_m}(x_l)$ to
$\mu_{l_1}\cdots \mu_{l_m}(y_l)$. Combining these diagrams, we get
the following commutative diagram:
$$
\xymatrix{
&{\mathcal{T}} \ar[dr]^{X^{T''}_{?}} \ar[dl]_{X^{T'}_{?}}&\\
{\mathcal{A}_{Q_{T'}}}\ar@{->}[r]_{f_1}&\mathcal{A}_{Q'}\ar@{->}[r]_{f_2^{-1}}&{\mathcal{A}_{Q_{T''}}}
}
$$
The composition $g=f_2^{-1}f_1$ is therefore an isomorphism of cluster algebras such that $f_2^{-1}f_1(x_i)=x_i$ for each $i\neq j$ and with
$g(x_j)=f_2^{-1}f_1(x_j)=\mu_j(x_j)$.
We have $T'=T'_1\oplus \cdots\oplus T'_n$,
$T''=T'_1\oplus \cdots\oplus T'_{j-1}\oplus T''_j\oplus  T'_{j+1}\oplus \cdots  T'_n$ and so
$X^{T'}_{T'_i}=x_{P_i}$ for each $1\leq i\leq n$ and $X^{T''}_{T'_i}=x_{P'_i}$ for each $i\neq j$.
Therefore the cluster algebra $\mathcal{A}(\mathbf{x}, \mu_j(Q))$ is generated by cluster variables
$x_1, \cdots, x_{j-1}, x_j', x_{j+1}, \cdots, x_n, x_{P'_1},
\cdots, x_{P'_{j-1}}, g(x_{P_j}), x_{P'_{j+1}}, \cdots, x_{P'_n}$.
\end{proof}

For every object $X$ $\in \mathcal{C}_H$ there exists a triangle
$$
T_1^X\rightarrow T_0^X\rightarrow X\rightarrow T_1^X[1]
$$
with $T_1^X$ and $T_0^X$ in $\add(T)$ (see \cite{KR1}). The {\em index} $\ind(X)$ of $X$ is
defined as the class of $[F(T_0^X)]-[F(T_1^X)]$ in $K_0(\proj B)$, \cite{Pa}. Similarly, the {\em coindex}
of $X$ is the class $[F(T^0_X)]-[F(T^1_X)]$ in $K_0(\proj B)$, where
$$
X\rightarrow T^0_X[2]\rightarrow T^1_X[2]\rightarrow X[1]
$$
is a triangle in $\mathcal{C}_H$ with $T^0_X, T^1_X\in \add(T)$.
Then, for any
any object $M\in \mathcal{C}_H$, we have

\begin{equation}\label{eq}
X^{T}_{M}=\underline{x}^{-\coind(M)}\sum_e\chi(Gr_e(F(M)))\prod_{i=1}^{n}x_i^{<S_i, e>_a}
\end{equation}
where $\underline{x}^e=\prod_{i=1}^{n}x_i^{[e:P_i]}$ with
$e\in K_0(\proj B)$ and $[e:P_i]$ is the $i$th coefficient of $e$ in the basis
$[P_1], [P_2], \cdots, [P_n]$ (see \cite{Pa}).

\begin{lemma}\label{lemm2}
Let $Q$ be an acyclic quiver, $\mathcal{A}=\mathcal{A}(\mathbf{x}, Q)$ and $Q=Q_T$ for some
cluster tilting object $T=T_1\oplus \cdots\oplus T_n$.
For any $1\le j\le n$ consider $\rad(P_j)$:
Let $M$ be a direct sum of rigid objects of $\mathcal{C}_H$
such that $\rad(P_j)=\mathcal{C}_H(T,M)$. Furthermore, if $P^1\to P^0\to \rad(P_j)\to 0$ is a
minimal projective resolution of $\rad(P_j)$ in $\modd(B)$, let $r_1,\dots, r_n$ be given by
$P^1=\oplus_{i=1}^nP_i^{r_i}$.
Then
$$
x_jx_{P_j}=x_jX^{T}_{T_j}=1+X^{T}_{M}\prod x_i^{-r_i}\prod_{l\rightarrow j} x_l.
$$
\end{lemma}

\begin{proof}
Let $\xymatrix{0\ar@{->}[r]&\rad(P_j)\ar@{->}[r]^{i}& P_j\ar@{->}[r]^{p}& S_j\ar@{->}[r]&0}$ be an
exact sequence in $\modd(B)$. By \cite[Lemma 3.1]{Pa} there exist triangles
$T_1^X\rightarrow T_0^X\rightarrow X\rightarrow T_1^X[1]$ and
$X\rightarrow Y\oplus T_1^X[1]\rightarrow Z\rightarrow X[1]$ in $\mathcal{C}_H$ that
$F(X)=\mathcal{C}_H(T, X)\cong \rad(P_j)$, $F(Y)\cong P_j$ and $F(Z)\cong S_j$.
By \cite[Proposition 2.2]{Pa},
$\coind(Y\oplus T_1^X[1])=\coind(Y)+\coind(T_1^X[1])=\coind(X)+\coind(Z)$.
We have a minimal projective resolution $P^1\rightarrow P^0\rightarrow F(X)\rightarrow 0$, where
$F(X)\cong \rad(P_j)$. By \cite[Lemma 2.1]{Pa}, $\coind(T_1^X[1])=-\ind(T_1^X)=-\sum n_i[P_i]$,
where $T_1^X=\oplus_{j=1}^{n}T_j^{n_j}$.

Now, $X^{T}_{T_j}
=\underline{x}^{-\coind(T_j)}\sum_e\chi(Gr_e(F(T_j)))\prod_{i=1}^{n}x_i^{<S_i, e>_a}
=\underline{x}^{-\coind(T_j)}\prod_{i=1}^{n}x_i^{<S_i, P_j>_a}+\\\underline{x}^{-\coind(T_j)}\sum_{e\neq \underline{dim}(P_j)}\chi(Gr_e(F(T_j)))\prod_{i=1}^{n}x_i^{<S_i, e>_a}=\underline{x}^{-\coind(T_j)}\prod_{i=1}^{n}x_i^{<S_i, P_j>_a}+\\\underline{x}^{-\coind(T_j)+\coind(X)}X^{T}_{X}$.

Furthermore, $\coind(Z)=\sum_{i=1}^{n}<S_i, F(Z)>[P_i]=\sum_{i=1}^{n}<S_i, S_j>[P_i]=[P_j]-\sum_{l\rightarrow j}[P_l]$ and so
$-\coind(T_j)+\coind(X)=-\sum_{i=1}^{n}r_i[P_i]-[P_j]+\sum_{l\rightarrow j}[P_l]$.
Also, we have $\underline{x}^{-\coind(T_j)}\prod_{i=1}^{n}x_i^{<S_i, P_j>_a}$
$=\prod_{i=1}^{n}x_i^{-<P_j, S_i>}=x_j^{-1}$. Therefore
$x_{P_j}=X^{T}_{T_j}=x_j^{-1}+x_j^{-1}X^{T}_{X}\prod x_i^{-r_i}\prod_{l\rightarrow j} x_l$
and the result follows.
\end{proof}

\begin{remark}\label{rem:projective variables}
Let $\mathcal{A}(\mathbf{x}, Q)$ be an acyclic cluster algebra. Then there exists an acyclic seed
$\Sigma=(\mathbf{u}, Q')$ such that $\mathcal{A}(\mathbf{x}, Q)= \mathcal{A}(\mathbf{u}, Q')$.
Let $(T', Q_{T'})$ be the tilting seed such that  $Q_{T'}=Q$. Recall that we denote by
$\mathcal T$ a set of representatives of the isomorphism classes of the indecomposable objects of
$\mathcal C$. As before, see (\ref{eq:diagram-Palu}),
we have a commutative diagram

$$
\xymatrix{
&{\mathcal{T}} \ar[dr]^{X^{T'}_{?}} \ar[dl]_{X_{?}}&\\
{\mathcal{A}_{Q'}}&&\mathcal A_Q \ar@{->}[ll]_{f}
}
$$
Let $T'_i$ be an indecomposable direct summand of $T'$.
Then the projective cluster variable $x_{P_i}$ with respect to the seed $(\mathbf{x},Q)$
is equal to $X^{T'}_{T'_i}=f^{-1}X_{T'_i}$.
Using this fact and Remark~\ref{rem:projective variables1} we can calculate the projective
cluster variables with respect to $(\mathbf{x},Q)$.
\end{remark}

\begin{example} \label{ex:projective-cl-var}
Let $\Sigma=(\mathbf{x}, Q)$ be the seed with $\mathbf{x}=\{x_1, x_2, x_3, x_4\}$ and $Q$
the following quiver.
$$
\xymatrix{
&{1} \ar[dl]&\\
{2}\ar@/^-.3pc/[rr]\ar@/^.3pc/[rr]&&{3}\ar[ul]\ar[dl]\\
&{4}\ar[ul]&
}
$$
Then
$Q''=\mu_4\mu_1(Q)$ is the following acyclic quiver
$$\hskip.5cm \xymatrix{
&{1} \ar[dr]&\\
{2}\ar[ur]\ar[dr]&&{3}\\
&{4}\ar[ur]&
}
$$
Let $\mathcal C=\mathcal C_H$ be the cluster category of the hereditary algebra $H=kQ''$.
For $i=1,\dots, 4$ let $P_i$ be the indecomposable projective $H$-module corresponding to the
vertex $i$ of $Q''$.
Then
$T'':=P_1\oplus P_2\oplus P_3\oplus P_4$ is a cluster tilting object in $\mathcal{C}$ corresponding to the quiver $Q''$.
The module $P_1\oplus P_2\oplus P_3$ is a sincere (i.e. each simple module occurs as a
composition factor) almost complete basic tilting $H$-module and has two complements by~\cite{HU}.
One of them is $P_4$. Using approximations (see section 6 of \cite{BMRRT}) we find
that the other one is the regular indecomposable $H$-module $M$ with dimension vector
$(1, 1, 1, 0)$.
Then $T':=P_1\oplus P_2\oplus P_3\oplus M$
is the cluster tilting object in $\mathcal{C}$ corresponding to the quiver $Q'=\mu_1(Q)$.
Let $N$ be the regular indecomposable $H$-module with dimension vector $(0,1,1,1)$.
By a similar argument we can see that
$T= N\oplus P_2\oplus P_3\oplus M$ is the cluster tilting object in $\mathcal{C}$ corresponding to the quiver $Q$.
So by Remark \ref{rem:projective variables}, $x_{P_i}=X^{T}_{T_i}=f^{-1}X_{T_i}$ for all $i$.
We calculate these and obtain
$x_{P_1}=\dfrac{x_1x_4+(x_2+x_3)^2}{x_1x_2x_3}$,
$x_{P_2}=\dfrac{x_1x_3^2x_4+(x_1x_4+x_2(x_2+x_3))^2}{x_1x_2x_3^2x_4}$,
$x_{P_3}=\dfrac{x_1x_4+(x_2+x_3)^2}{x_1x_3x_4}$ and
$x_{P_4}=\dfrac{x_1x_4+(x_2+x_3)^2}{x_2x_3x_4}$.
Note that as an alternative way we can also calculate the index and the coindex of each $T_i$, and compute
$x_{P_i}=X^{T}_{T_i}$ using formula \eqref{eq}.

\end{example}

\section{Cluster algebras of type $A_n$}\label{sec 5}
In this section we show that in type $A$, {\em without coefficients},
all cluster algebras are equal to their lower bound cluster algebras
generated by projective cluster variables. In other words,
if $\Sigma=(\mathbf{x}, Q)$ is a seed where the quiver $Q$ is mutation equivalent to a quiver $Q'$
of type $A_n$, the lower bound cluster algebra generated by projective cluster variables
$\mathcal{L}_P(\Sigma)$ and the cluster algebra $\mathcal{A}(\Sigma)$ coincide. We show the above claim
using the characterization of quivers which are mutation-equivalent to a type $A_n$
quiver and by considering the position of the three-cycles which can arise in this case.

Quivers of mutation type $A_n$ can be characterized as follows:

Let $\mathcal{Q}_n$ be the class of quivers with $n$ vertices which satisfy the following:
\begin{itemize}
\item[i)] All non-trivial cycles are oriented and of length $3$,
\item[ii)] Any vertex has at most four neighbors,
\item[iii)] If a vertex has four neighbors, then two of its adjacent arrows belong to one $3$-cycle, and
the other two belong to another $3$-cycle,
\item[iv)] If a vertex has exactly three neighbors, then two of its adjacent arrows belong to a $3$-cycle,
and the third arrow does not belong to any $3$-cycle.
\end{itemize}
Note that by a cycle we mean a cycle in the underlying graph, not passing
through the same edge twice.

\begin{lemma}\label{lemm4}(\cite[Proposition 2.4]{BV}) A quiver $Q$ is mutation equivalent to
a quiver of type $A_n$ if and only if $Q\in  \mathcal{Q}_n$.
\end{lemma}

Note that in particular, the cluster algebras given by the quivers $Q\in \mathcal{Q}_n$ are
acyclic as the 3-cycles can be removed by mutations.
We recall the head and tail functions: if $\alpha:i\to j$ is an arrow in a quiver,
the {\em head or target $t(\alpha)$ of $\alpha$} is the vertex $j$ and the
{\em tail or source $s(\alpha)$ of $\alpha$}
is the vertex $i$.

In Lemmas~\ref{lemm3} and \ref{lemm5}, we study how
mutations on type $A$ subquivers attached to a triangle
in the quivers $Q\in \mathcal{Q}_n$ keep the property "being generated by projectives".
For this, we fix the notation as in Figure~\ref{fig:quiverQ}.
We assume that such a quiver has an oriented triangle on three vertices $i,j,k$
with two "legs" attached to the vertices $i$ and $k$, both containing at least edge
(i.e. $r\ge 1$ and $s\ge 1$). At the third vertex, the subquiver is not
specified.

\begin{figure}
\includegraphics[width=10cm]{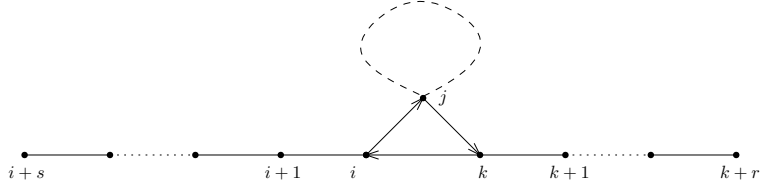}
\caption{Type A quiver with two legs at an oriented triangle}\label{fig:quiverQ}
\end{figure}

\begin{lemma}\label{lemm3}
Let $Q\in \mathcal{Q}_n$ be a quiver as in Figure~\ref{fig:quiverQ}. Assume that
$r\ge 2$ or that $s\ge 2$
and that there exists a sink or source vertex $l\in \{k+2, \cdots, k+r, i+2, \cdots, i+s\}$ in $Q$. Let
$\Sigma=(\mathbf{x}, Q)$ and let $\Sigma'=(\mathbf{x}, \mu_l(Q))$. If
$\mathcal{L}_P(\Sigma)=\mathcal{A}(\Sigma)$, then $
\mathcal{L}_P(\Sigma')=\mathcal{A}(\Sigma')$.
\end{lemma}

\begin{proof}
Let $l\in \{k+2, \cdots, k+r, i+2, \cdots, i+s\}$ be a sink or a source.
Without loss of generality we can assume that $l\in \{i+2, \cdots, i+s\}$ as mutation at $l$
does not affect the rest of the quiver.
We prove the claim by
looking at the two cases where $l$ is a sink or a source, by considering the position of $l$ on the
leg starting at $i$ and the orientations of the arrows incident with $l$ in $Q$ and in the quiver
$Q':=\mu_l(Q)$. Note that by the definition of $Q$, the vertex $l$ has either one or two arrows
incident with it.

Let $g:\mathcal{A}(\mathbf{x}, Q)\rightarrow \mathcal{A}(\mathbf{x'}, Q)$ be the isomorphism from
Theorem~\ref{thm2}, with $\mathbf{x'}=\mathbf{x}\cup \{\mu_l(x_l)\}\backslash \{x_l\}$. Then by Theorem \ref{thm2}, the cluster algebra $\mathcal{A}(\mathbf{x'}, Q)$
$=\mathcal{A}(\mathbf{x}, \mu_l(Q))$ is generated by the cluster variables
$x_1, \cdots, x_{l-1}, x_l', x_{l+1}, \cdots, x_n$,
$x_{P'_1}, \cdots, x_{P'_{l-1}}, g(x_{P_l}), x_{P'_{l+1}}, \cdots,\\ x_{P'_n}$
where $x_{P_i'}:=x_{T_i'}$ is a projective cluster variable of the seed
$\Sigma'=(\mathbf{x},\mu_l(Q))$.
Therefore it is enough to show that the cluster variables $x_l'$ and $g(x_{P_l})$ are generated
by $x_1, x_2, \cdots, x_n,$ $x_{P'_1}, x_{P'_2}, \cdots, x_{P'_n}$. \\
\underline{(I) Let $l$ be a sink.}
In this case, $g(x_{P_l})=x_l$.
We have to show that
$x_l'$ is generated by the cluster variables
$x_1, x_2, \cdots, x_n, x_{P'_1}, x_{P'_2}, \cdots, x_{P'_n}$. \\
(A) First assume that there exists only one arrow $\alpha \in Q_1$ that $t(\alpha)=l$.
Then the start of $\alpha$ is $l-1$. (In this case, $l$ is the vertex $i+s$).
In this case $x_l'=\dfrac{1+x_{l-1}}{x_l}$ and by Lemma \ref{lemm2},
$x_{P'_l}=\dfrac{1+x_{\rad(P'_{l})}}{x_l}$. The radical of $P'_l$ is $P'_{l-1}$
and so
$x_{P'_l}=\dfrac{1+\dfrac{1+x_{\rad(P'_{l-1})}\prod x_j^{-r_j}\prod_{i\rightarrow l-1} x_i}{x_{l-1}}}{x_l}$,
where $P^1\rightarrow P^0\rightarrow \rad(P'_{l-1})\rightarrow 0$ is a minimal projective resolution of
$\rad(P'_{l-1})$ and $P^1=\oplus_{j=1}^{n}P_j^{'r_j}$. Let $Q'=\mu_l(Q)$. \\
(A.1) If $l-1$ is a sink in $Q'$, then
$x_{l-1}x_{P'_l}=\dfrac{1+x_{l-1}}{x_l}+x_{l-2}$ and the result follows. \\
(A.2) If $l-1$ is not a sink in $Q'$,
then $x_{l-1}x_{P'_l}=\dfrac{1+x_{l-1}}{x_l}+x_{\rad(P'_{l-1})}$. In this case $\rad(P'_{l-1})=P'_{l-2}$
and the result follows. \\
(B) Next assume that there exists two arrows $\alpha, \beta \in Q_1$ with tail
$t(\alpha)=l=t(\beta)$ and heads $s(\alpha)=l-1$ and $s(\beta)=l+1$.
In this case $x_l'=\dfrac{1+x_{l-1}x_{l+1}}{x_l}$, $x_{P'_{l}}=\dfrac{1+x_{\rad(P'_{l})}}{x_l}$ and
$\rad(P'_{l})=P'_{l+1}\oplus P'_{l-1}$. Since $x_{\rad(P'_{l})}=X^{T'}_{M}$, where
$F(M)=rad(P'_{l})$, $x_{\rad(P'_{l})}=x_{P'_{l+1}}x_{P'_{l-1}}$. Then
$x_{P'_{l}}=\dfrac{1+x_{P'_{l+1}}x_{P'_{l-1}}}{x_l}$. \\
(B.1) If $l-1$ and $l+1$ are both sinks in
$Q'$, then $x_{P'_{l-1}}=\dfrac{1+\prod_{i\rightarrow l-1}x_i}{x_{l-1}}$ and
$x_{P'_{l+1}}=\dfrac{1+\prod_{i\rightarrow l+1}x_i}{x_{l+1}}$.
Then the result follows in this case, as we have
$x_{l-1}x_{l+1}x_{P'_l}=\dfrac{1+x_{l-1}x_{l+1}}{x_l}+\dfrac{\prod_{i\rightarrow l+1}x_i
+\prod_{i\rightarrow l-1}x_i+\prod_{i\rightarrow l+1}x_i\prod_{i\rightarrow l-1}x_i}{x_l}$. \\
(B.2) If $l+1$ is a sink and $l-1$ is not a sink in $Q'$, then
$x_{P'_{l+1}}=\dfrac{1+\prod_{i\rightarrow l+1}x_i}{x_{l+1}}$ and
$x_{P'_{l-1}}=\dfrac{1+x_{P'_{l-2}}x_l}{x_{l-1}}$. So
$x_{l-1}x_{l+1}x_{P'_l}=\dfrac{1+x_{l-1}x_{l+1}}{x_l}$
$+\dfrac{\prod_{i\rightarrow l+1}x_i+x_{P'_{l-2}}x_l+x_{P'_{l-2}}x_l\prod_{i\rightarrow l+1}x_i}{x_l}$
and the result follows. \\
(B.3)
If $l-1$ is a sink and $l+1$ is not a sink in $Q'$, then $x_{P'_{l+1}}=\dfrac{1+x_{P'_{l+2}}x_l}{x_{l+1}}$
and $x_{P'_{l-1}}=\dfrac{1+\prod_{i\rightarrow l-1}x_i}{x_{l-1}}$. So
$x_{l-1}x_{l+1}x_{P'_l}=\dfrac{1+x_{l-1}x_{l+1}}{x_l}+\dfrac{\prod_{i\rightarrow l-1}x_i
+x_{P'_{l+2}}x_l+x_{P'_{l+2}}x_l\prod_{i\rightarrow l-1}x_i}{x_l}$ and the result follows. \\
(B.4) If both $l+1$ and $l-1$ are not sinks in $Q'$, then
$x_{P'_{l+1}}=\dfrac{1+x_{P'_{l+2}}x_l}{x_{l+1}}$
and $x_{P'_{l-1}}=\dfrac{1+x_{P'_{l-2}}x_l}{x_{l-1}}$. Then
$x_{l-1}x_{l+1}x_{P'_l}=\dfrac{1+x_{l-1}x_{l+1}}{x_l}
+\dfrac{x_{P'_{l-2}}x_l+x_{P'_{l+2}}x_l+x_{P'_{l-2}}x_lx_{P'_{l+2}}x_l}{x_l}$
and the result follows. \\
\underline{(II) Let $l$ be a source}. We show that the cluster variables $x_l'$ and
$g(x_{P_l})$ are generated by $x_1, x_2, \cdots, x_n, x_{P'_1}, x_{P'_2}, \cdots, x_{P'_n}$. \\
(A)
Assume that there exists only one arrow $\alpha \in Q_1$ that $s(\alpha)=l$. Then
$t(\alpha)=l-1$ (note that $l=i+s$).
In this case $x_l'=\dfrac{1+x_{l-1}}{x_l}$ and $x_{P'_1}=\dfrac{1+x_{l-1}}{x_l}$.
So it is enough to show that $g(x_{P_l})$ is generated by
$x_1, x_2, \cdots, x_n, x_{P'_1}, x_{P'_2}, \cdots, x_{P'_n}$. \\
(A.1)
If $l-1$ is a sink in $Q$, then $x_{P_l}$ is a cluster variable in $\mu_l\mu_{l-1}(\mathbf{x})$
and so $g(x_{P_l})$ is the corresponding cluster variable in the cluster
$\mu_l\mu_{l-1}(\mu_l\mathbf{x})$. An easy calculation shows that
$g(x_{P_l})=\dfrac{x_l+x_{l-2}}{x_{l-1}}$. In this case
$x_lx_{P'_{l-1}}=1+x_{l-2}x_{P'_l}=\dfrac{x_l+x_{l-2}}{x_{l-1}}+x_{l-2}$ and the result follows. \\
(A.2)
Now assume that $l-1$ is not a sink in $Q$. Let $\beta$ be the (unique) arrow starting at $l-1$,
with $t(\beta)=l-2$. Then there exists a seed $(\mathbf{y}, Q'')$ which is
mutation equivalent to $(\mathbf{x}, Q)$, $y_l=x_l, y_{l-1}=x_{l-1}, y_{l-2}=x_{P_{l-2}}$ and $x_{P_l}$
is the corresponding cluster variable in the seed $\mu_l\mu_{l-1}(\mathbf{y}, Q'')$.
Then $g(x_{P_l})$ is the corresponding cluster variable in the seed
$\mu_l\mu_{l-1}\mu_l(\mathbf{y}, Q'')$ and so $g(x_{P_l})=\dfrac{x_l+x_{P'_{l-2}}}{x_{l-1}}$.
Also in this case $x_lx_{P'_{l-1}}=\dfrac{x_l+x_{P'_{l-2}}}{x_{l-1}}+x_{P'_{l-2}}$ and the result follows. \\
(B)
Next assume that there are two arrows $\alpha, \beta \in Q_1$ with $s(\alpha)=l=s(\beta)$ with
$t(\alpha)=l-1$ and $t(\beta)=l+1$. In this case
$x_l'=\dfrac{1+x_{l-1}x_{l+1}}{x_l}$, $x_{P'_{l}}=\dfrac{1+x_{l-1}x_{l+1}}{x_l}$.
Then it is enough to show that $g(x_{P_l})$ is generated by
$x_1, x_2, \cdots, x_n, x_{P'_1}, x_{P'_2}, \cdots, x_{P'_n}$. In this situation,
$g(x_{P_l})=\dfrac{1+x_{P'_{l+1}}x_{P'_{l-1}}}{\dfrac{1+x_{l-1}x_{l+1}}{x_l}}$.
Therefore $g(x_{P_l})=x_lx_{P'_{l+1}}x_{P'_{l-1}}-x_{l-1}x_{l+1}g(x_{P_l})+x_l$.
So it is enough to show that $x_{l-1}x_{l+1}g(x_{P_l})$ is generated by
$x_1, x_2, \cdots, x_n, x_{P'_1}, x_{P'_2}, \cdots, x_{P'_n}$.
We have six cases. \\
(B.1): Assume that $l+1$ is a sink in $Q$, there is exactly one arrow in
$Q_1$ with the target $l+1$
Assume that $l+1$ is a sink in $Q$ and there is no other arrow of $Q$
ending at $l+1$ and $l-1$ is a sink in $Q$. In this case
$x_{P'_{l+1}}=\dfrac{x_l+1+x_{l-1}x_{l+1}}{x_lx_{l+1}}$ and $x_{P'_{l-1}}$
$=\dfrac{x_l+x_{l-2}(1+x_{l-1}x_{l+1})}{x_lx_{l-1}}$.
An easy calculation shows that $x_{l-1}x_{l+1}g(x_{P_l})=x_l+x_{l-2}+1+x_{l-2}x_{P'_l}$
and the result follows. \\
(B.2): Assume that $l+1$ is a sink in $Q$, there is exactly one arrow in $Q_1$ with the target
$l+1$ and $l-1$ is not a sink in $Q$. In this case
$x_{P'_{l+1}}=\dfrac{x_l+1+x_{l-1}x_{l+1}}{x_lx_{l+1}}$ and
$x_{P'_{l-1}}=\dfrac{x_l+(1+x_{l-1}x_{l+1})x_{P'_{l-2}}}{x_lx_{l-1}}$.
Then $x_{l-1}x_{l+1}g(x_{P_l})=x_l+x_{P'_{l-2}}+1+x_{P'_{l-2}}x_{P'_l}$ and the result follows. \\
(B.3): Assume that $l+1$ is a sink in $Q$, there are exactly two arrows in $Q_1$
with the target $l+1$ and $l-1$ is a sink in $Q$. In this case
$x_{P'_{l+1}}=\dfrac{x_l+x_{l+2}(1+x_{l-1}x_{l+1})}{x_lx_{l+1}}$ and
$x_{P'_{l-1}}=\dfrac{x_l+x_{l-2}(1+x_{l-1}x_{l+1})}{x_lx_{l-1}}$. Then
$x_{l-1}x_{l+1}g(x_{P_l})=x_l+x_{l-2}+x_{l+2}+x_{l+2}x_{l-2}x_{P'_l}$ and the result follows. \\
(B.4): Assume that $l+1$ is a sink in $Q$, there are exactly two arrows in $Q_1$ with the
target $l+1$ and $l-1$ is not a sink in $Q$. In this case
$x_{P'_{l+1}}=\dfrac{x_l+x_{l+2}(1+x_{l-1}x_{l+1})}{x_lx_{l+1}}$ and
$x_{P'_{l-1}}=\dfrac{x_l+(1+x_{l-1}x_{l+1})x_{P'_{l-2}}}{x_lx_{l-1}}$. Then
$x_{l-1}x_{l+1}g(x_{P_l})=x_l+x_{P'_{l-2}}+x_{l+2}+x_{l+2}x_{P'_{l-2}}x_{P'_l}$ and the result
follows. \\
(B.5): Assume that $l+1$ is not a sink in $Q$ and that $l-1$ is a sink.
In this case
$x_{P'_{l+1}}=\dfrac{x_l+x_{P'_{l+2}}(1+x_{l-1}x_{l+1})}{x_lx_{l+1}}$ and
$x_{P'_{l-1}}=\dfrac{x_l+x_{l-2}(1+x_{l-1}x_{l+1})}{x_lx_{l-1}}$.
Then $x_{l-1}x_{l+1}g(x_{P_l})=x_l+x_{l-2}+x_{P'_{l+2}}+x_{P'_{l+2}}x_{l-2}x_{P'_l}$
and the result follows. \\
(B.6): Assume that $l+1$ is not a sink in $Q$,
and $l-1$ is not a sink in $Q$. In this case
$x_{P'_{l+1}}=\dfrac{x_l+x_{P'_{l+2}}(1+x_{l-1}x_{l+1})}{x_lx_{l+1}}$ and
$x_{P'_{l-1}}=\dfrac{x_l+(1+x_{l-1}x_{l+1})x_{P'_{l-2}}}{x_lx_{l-1}}$.
Then $x_{l-1}x_{l+1}g(x_{P_l})=x_l+x_{P'_{l-2}}+x_{P'_{l+2}}+x_{P'_{l+2}}x_{P'{l-2}}x_{P'_l}$
and the result follows.
\end{proof}

\begin{lemma}\label{lemm5} Let $Q\in \mathcal{Q}_n$ be the quiver from Figure~\ref{fig:quiverQ}.
Assume that there exists a sink or source vertex $l\in \{k+1, i+1\}$ in $Q$. Let
$\Sigma=(\mathbf{x}, Q)$ and $\Sigma'=(\mathbf{x}, \mu_l(Q))$.
If  $\mathcal{L}_P(\Sigma)=\mathcal{A}(\Sigma)$,
then $\mathcal{L}_P(\Sigma')=\mathcal{A}(\Sigma')$.
\end{lemma}

\begin{proof}
As before, let $Q'=\mu_l(Q)$.
Without loss of generality we can assume that $l=i+1$ as mutation at $l$ does not
affect the rest of the quiver. In particular, mutation at $l$ does not change the arrows
incident with $j$ nor the arrows incident with $k$.

Let $g:\mathcal{A}(\mathbf{x}, Q)\rightarrow \mathcal{A}(\mathbf{x'}, Q)$ be the
isomorphism from Remark~\ref{rem:induced-isom},
where $\mathbf{x'}=\mathbf{x}\cup \{\mu_l(x_l)\}\backslash \{x_l\}$. By Theorem \ref{thm2}, it is
enough to show that the cluster variables $x_l'$ and $g(x_{P_l})$ are generated
by $x_1, x_2, \cdots, x_n, x_{P'_1}, x_{P'_2}, \cdots, x_{P'_n}$. \\
\underline{(I) Let $l$ be a sink}
Then $g(x_{P_l})=x_l$, so we only have to show
that $x_l'$ is generated by $x_1, x_2, \cdots, x_n, x_{P'_1}, x_{P'_2}, \cdots, x_{P'_n}$. \\
(A) Assume first that there is only one arrow in $Q$ ending at $l$ (i.e. $s=1$).
In this case $Q$ is as follows:
\begin{center}
$
\includegraphics[width=7cm]{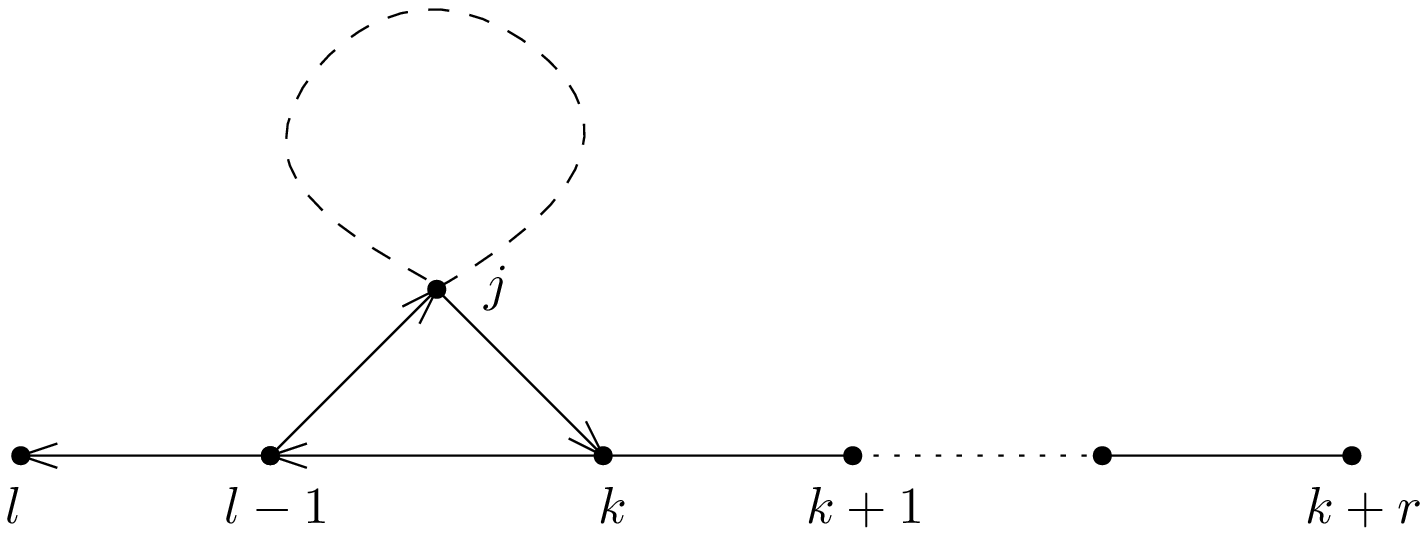}
$
\end{center}
We have $x_{P'_l}=\dfrac{1+x_{\rad(P'_l)}}{x_l}=\dfrac{1+x_{P'_{l-1}}}{x_l}$ and
$x_{P'_{l-1}}=\dfrac{1+x_lx_{\rad(P'_{l-1})}}{x_{l-1}}$, therefore
$x_{l-1}x_{P'_l}=\dfrac{1+x_{l-1}}{x_l}+x_{\rad(P'_{l-1})}$.
So it is enough to show that $x_{\rad(P'_{l-1})}$ is generated by the cluster variables
$x_1, x_2, \cdots, x_n, x_{P'_1}, x_{P'_2}, \cdots, x_{P'_n}$.
We distinguish four cases (A.1)-(A.4) for the subquiver attached to $j$: it can be only the vertex $j$, it
can start with an arrow towards $j$ or with an arrow from $j$ or it can have a triangle going
through $j$.

\noindent
(A.1) Assume that at $j$ there are only the two arrows $k\to j$ and $j\to k$ in
$Q$ and hence in $Q'$.
In this case
$x_{\rad(P'_{l-1})}=\dfrac{x_{l-1}+x_k}{x_j}$.
By the definition of $Q$, the vertex $k$ has 2 or 3 arrows incident with it.
If $Q$ (and so $Q'$) has only one arrow ending in $k$ and one
arrow starting from $k$ (in particular, the "leg at $k$" only consists of the vertex $k$), then
$x_kx_{P'_j}=\dfrac{x_{l-1}+x_k}{x_j}+1$ and the result follows (to calculate $x_{P_j'}$, we first
mutate $(x,Q')$ at $k$ and then at $j$, after this, the cluster variable corresponding to vertex $j$ is
$x_{P_j'}$.
If $Q'$ has two arrows ending at $k$ and one starting from $k$, then $x_kx_{P'_j}=\dfrac{x_{l-1}+x_k}{x_j}+x_{k+1}$,
and the result follows.
It remains to consider the situation where there are two arrows starting at $k$.
This is the trickiest situation as we also need to consider the vertex $k+1$ - whether it is a sink or
not.
If there are exactly two arrows in $Q_1'$ starting at $k$ (one of them $k\to k+1$)
and only one ending at $k$ {\em and}
the vertex $k+1$ has no other arrow ending at it,
$x_kx_{P'_j}=\dfrac{x_{l-1}+x_k}{x_j}+\dfrac{1+x_k}{x_{k+1}}=\dfrac{x_{l-1}+x_k}{x_j}+x_{P'_{k+1}}$,
and the result follows.
Assume now that there are exactly two arrows starting at $k$ and one ending at $k$.
Assume further that there are exactly two arrows incident with $k+1$. By Lemma~\ref{lemm3},
we can assume that in this case, $Q'$ looks as follows:
$$
\includegraphics[width=8cm]{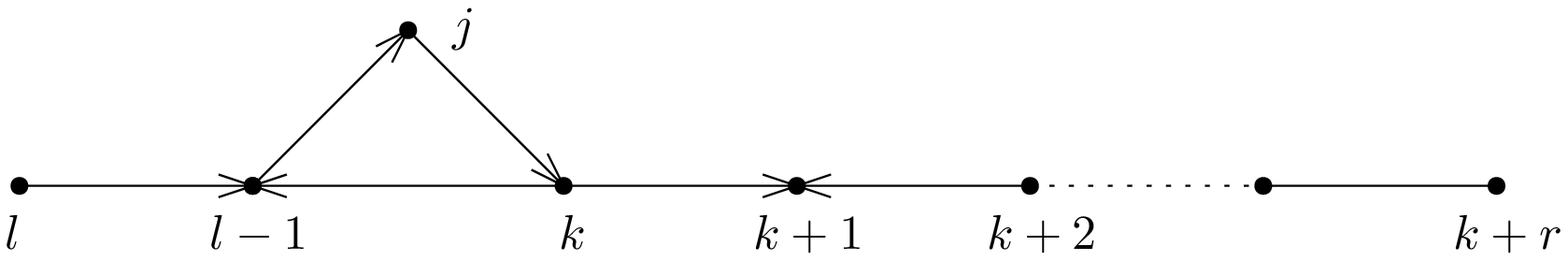}
$$
Then $x_kx_{P'_j}=\dfrac{x_{l-1}+x_k}{x_j}+\dfrac{1+x_kx_{k+2}}{x_{k+1}}
=\dfrac{x_{l-1}+x_k}{x_j}+x_{P'_{k+1}}$
and the result follows. \\
(A.2) Assume that there are three arrows in $Q_1'$ incident with $j$ and that the
one in the subquiver at $j$ starts at $m$ and ends at vertex $j$ (so $m\ne l-1$).
In this case $x_{\rad(P'_{l-1})}=\dfrac{x_mx_{l-1}+x_k}{x_j}$.
As before, we look at the vertex $k$ - it can have 2 or 3 arrows incident with it.
If there is only one arrow in $Q'$ ending at $k$ and only one arrow in $Q'$ starting at $k$, then
$x_kx_{P'_j}=\dfrac{x_{l-1}x_m+x_k}{x_j}+x_m$ (note that in this case, the leg of $Q$
at the vertex $k$ only consists of the vertex $k$).
If there are three arrows incident with $k$, two ending at $k$, one of them being
$k+1\to k$,
then $x_kx_{P'_j}=\dfrac{x_mx_{l-1}+x_k}{x_j}+x_mx_{k+1}$
and the result follows.
If there are two arrows in $Q'$ starting at $k$ (one of them being $k\to k+1$),
only one arrow ending at $k$ {\em and} there is
no other arrow incident with $k+1$, then
$x_kx_{P'_j}=\dfrac{x_{l-1}x_m+x_k}{x_j}+x_m\dfrac{1+x_k}{x_{k+1}}
=\dfrac{x_mx_{l-1}+x_k}{x_j}+x_mx_{P'_{k+1}}$, where $l-1=i$
and the result follows.
Lastly, assume that there are two arrows in $Q_1'$ starting at $k$ (one being $k\to k+1$),
one ending at $k$ and that there are exactly two arrows incident with $k+1$.
Then by Lemma~\ref{lemm3}, we can assume that the quiver $Q'$ looks as follows:
$$
\includegraphics[width=8cm]{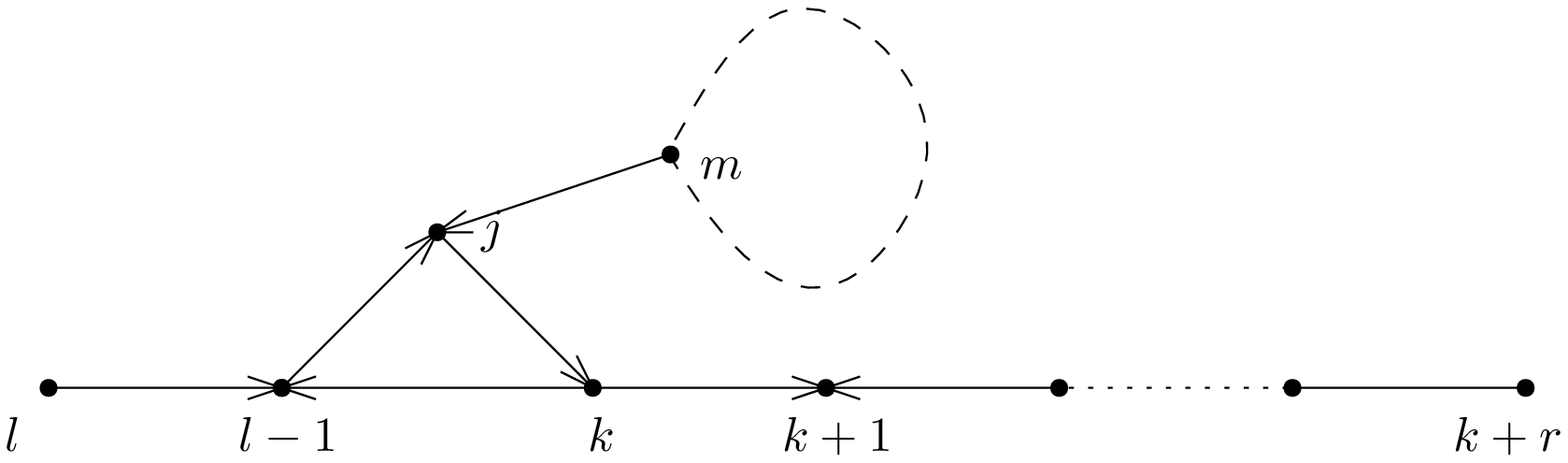}
$$
Then $x_kx_{P'_j}=\dfrac{x_mx_{l-1}+x_k}{x_j}+\dfrac{x_mx_j+x_mx_kx_{k+2}x_j}{x_jx_{k+1}}
=\dfrac{x_mx_{l-1}+x_k}{x_j}+x_mx_{P'_{k+1}}$ and the result follows. \\
(A.3) Assume that there are two arrows in $Q_1'$ starting at $j$ and only one arrow ending
at $j$. Let $j\to m$ be the arrow on the subquiver at $j$ (so $m\ne k$).
In this case
$x_{\rad(P'_{l-1})}=\dfrac{x_{P'_m}x_{l-1}+x_k}{x_j}$.
We look at the number and orientations of the arrows incident with $k$.
If in $Q'$ there is only one arrow ending and only one arrow starting at $k$ (i.e. $r=0$), then
$x_kx_{P'_j}=\dfrac{x_{l-1}x_{P'_m}+x_k}{x_j}+x_{P'_m}$ and
the result follows.
If there are two arrows in $Q'$ ending at $k$ (one of them being $k+1\to k$),
$x_kx_{P'_j}=\dfrac{x_{P'_m}x_{l-1}+x_k}{x_j}+x_{P'_m}x_{k+1}$,
and the result follows.
If there are two arrows starting at $k$ (one of them is $k\to k+1$) and one ending at $k$ {\em and}
there is no other arrow of $Q'$ incident with $k+1$,
then $x_kx_{P'_j}=\dfrac{x_{l-1}x_{P'_m}+x_k}{x_j}
+x_{P'_m}\dfrac{1+x_k}{x_{k+1}}=\dfrac{x_{P'_m}x_{l-1}+x_k}{x_j}+x_{P'_m}x_{P'_{k+1}}$
(with $l-1=i$) and the result follows.

Assume that there are two arrows starting at $k$ (one of them $k\to k+1$)
and one arrow ending at $k$ and that there are exactly two arrows incident with $k+1$.
Then by Lemma~\ref{lemm3} we can assume that in this case, $Q'$ is as follows:
$$
\includegraphics[width=8cm]{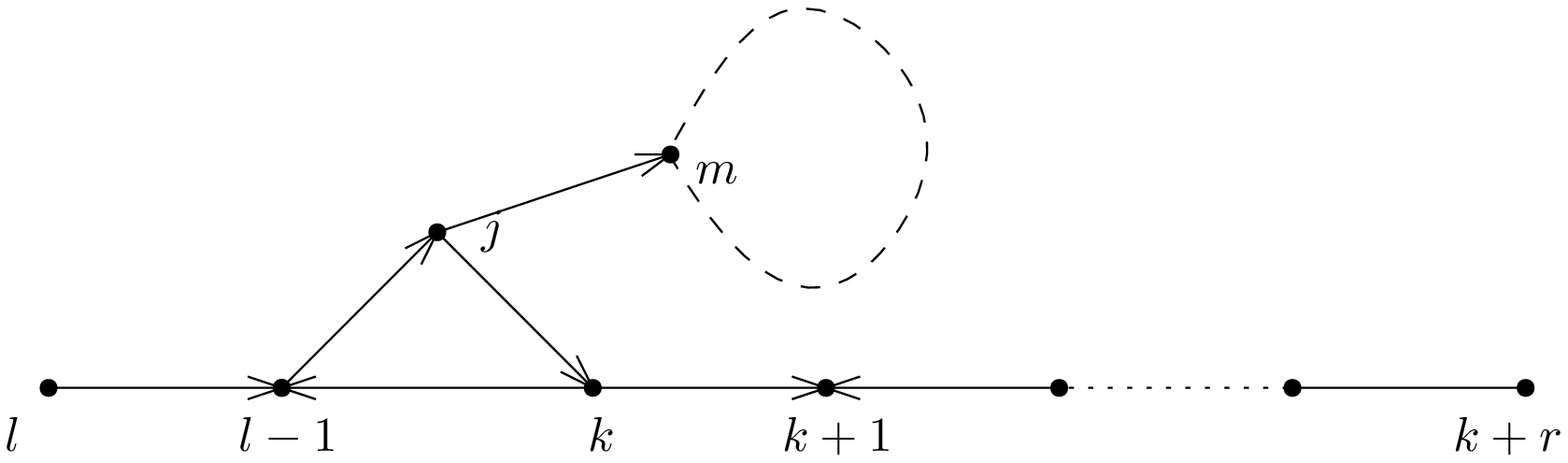}
$$
Then $x_kx_{P'_j}=\dfrac{x_{P'_m}x_{l-1}+x_k}{x_j}
+\dfrac{x_{P'_m}x_j+x_{P'_m}x_kx_{k+2}x_j}{x_jx_{k+1}}
=\dfrac{x_{P'_m}x_{l-1}+x_k}{x_j}+x_{P'_m}x_{P'_{k+1}}$ and the result follows. \\
(A.4)
Now assume that there are four arrows in $Q'$ which are incident with $j$.
Let $j\to m$ be the outgoing arrow in the subquiver at $j$ (i.e. $m\ne k$).
Then $x_{\rad(P'_{l-1})}=\dfrac{x_{M}x_{l-1}+x_k}{x_j}$ where $M$ is an indecomposable direct
summand of $\rad(P'_j)$ that the $m$'th coordinate of the dimension vector
$\underline{\dimm}(M)$ is nonzero.
As before, we have to consider the possibilities around the vertex $k$.
If there are only two arrows incident with $k$, then
$x_kx_{P'_j}=\dfrac{x_{l-1}x_{M}+x_k}{x_j}+x_{M}$.
Then it is enough to show that $x_{M}$ is generated by the variables
$x_1, x_2, \cdots, x_n, x_{P'_1}, x_{P'_2}, \cdots, x_{P'_n}$.
Similarly as before, we can show inductively, that $x_M$ is generated by the variables
$x_1, x_2, \cdots, x_n$, $x_{P'_1}, x_{P'_2}, \cdots, x_{P'_n}$ and the result
follows. We consider the quiver locally.
If there are three arrows incident with $k$, two incoming, one of them $k+1\to k$ then
$x_kx_{P'_j}=\dfrac{x_{M}x_{l-1}+x_k}{x_j}+x_{M}x_{k+1}$.
Inductively we can see that $x_{M}$ is generated by
the $x_i$ and the $x_{P'_i}$ and the result follows.
If there are three arrows incident with $k$, two outgoing, one of them $k\to k+1$ and
there is no other arrow incident with $k+1$
then
$x_kx_{P'_j}=\dfrac{x_{l-1}x_{M}+x_k}{x_j}+x_{M}\dfrac{1+x_k}{x_{k+1}}=
\dfrac{x_{M}x_{l-1}+x_k}{x_j}+x_{M}x_{P'_{k+1}}$ with $l-1=i$.
Inductively we can see that $x_{M}$ is generated by the $x_i$
and the $x_{P'_i}$ and the result follows.
Assume lastly that there are three arrows at $k$, two outgoing, one of them $k\to k+1$ and that
there are exactly two arrows at $k+1$. The quiver $Q'$ looks as follows:
$$
\includegraphics[width=8cm]{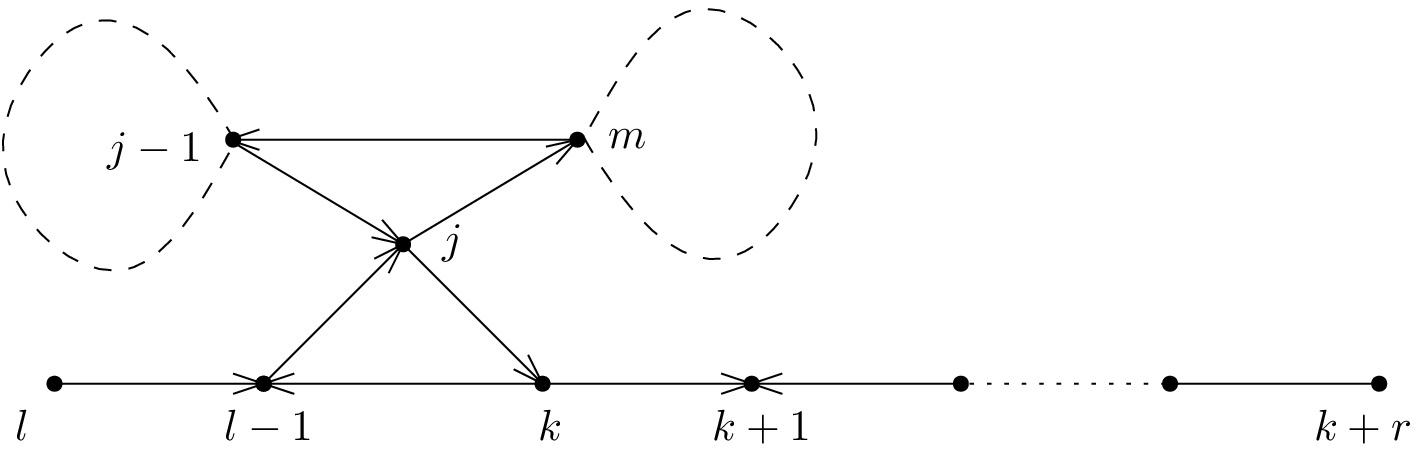}
$$
Then $x_kx_{P'_j}=\dfrac{x_{M}x_{l-1}+x_k}{x_j}
+\dfrac{x_{M}x_j+x_{M}x_kx_{k+2}x_j}{x_jx_{k+1}}=\dfrac{x_{M}x_{l-1}+x_k}{x_j}+x_{M}x_{P'_{k+1}}$.
Inductively we can see that $x_{M}$ is generated
by the $\{x_i\}_i$ and $\{x_{P'_i}\}_i$ and the result follows.

\noindent
(B) Now assume that there are two arrows ending at $l$, so $Q$ is of this form:
$$
\includegraphics[width=8cm]{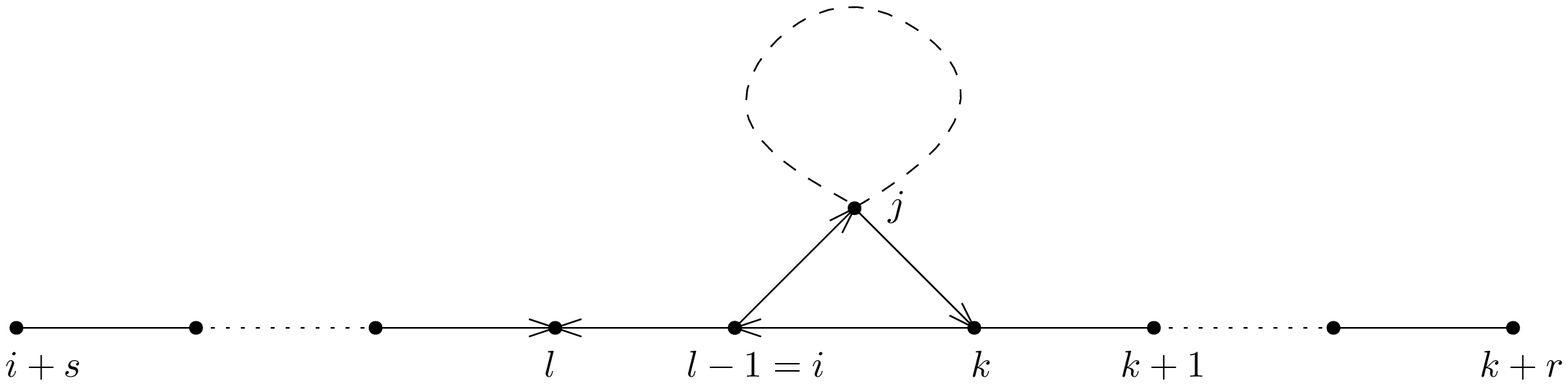}
$$
We show that $x_l'$ is generated by the $\{x_i\}_i$ and $\{x_{P'_i}\}_i$.
We have three cases. \\
(B.1) Assume that there is only one arrow in $Q$ incident with $l+1$ (and so $s=2$).
In this case
$x_{P'_l}=\dfrac{1+x_{\rad(P'_l)}}{x_l}=\dfrac{1+x_{P'_{l+1}}x_{P'_{l-1}}}{x_l}$ and so
$x_{l+1}x_{l-1}x_{P'_l}=\dfrac{1+x_{l-1}x_{l+1}}{x_l}+x_{\rad(P'_{l-1})}+1+x_lx_{\rad(P'_{l-1})}$.
The same argument as in the first part of the proof
shows that $x_{\rad(P'_{l-1})}$ is generated by the $\{x_i\}_i$ and $\{x_{P'_i}\}_i$ and so
the result follows. \\
(B.2) Assume that there are two arrows in $Q$ incident with $l+1$, one $l+1\to l$ and one
$l+2\to l+1$.
In this case $x_{l+1}x_{l-1}x_{P'_l}=\dfrac{1+x_{l-1}x_{l+1}}{x_l}+x_{\rad(P'_{l-1})}
+x_{l+2}+x_lx_{l+2}x_{\rad(P'_{l-1})}$.
The same argument as in the first part of the proof shows that $x_{\rad(P'_{l-1})}$ is generated
by the $\{x_i\}_i$ and by the $\{x_{P'_i}\}_i$ and so the result follows. \\
B.3) Assume that in $Q$ there are two arrows incident with $l+1$ and that $l+1$ is a source.
In this case $x_{l+1}x_{l-1}x_{P'_l}=\dfrac{1+x_{l-1}x_{l+1}}{x_l}
+x_{\rad(P'_{l-1})}+x_{P'_{l+2}}+x_lx_{P'_{l+2}}x_{\rad(P'_{l-1})}$.
The same argument as in the first part of the proof shows that $x_{\rad(P'_{l-1})}$ is generated by
the $\{x_i\}_i$ and the $\{x_{P'_i}\}_i$ and so the result follows. \\

\noindent
\underline{(II) Let $l$ be a source in $Q$}
We show that $g(x_{P_l})$ and $x_l'$ are generated by the variables $x_i$ and $x_{P'_i}$. \\
(A) Assume that there is only one arrow incident with $l$, so $Q$ looks as follows:
$$
\includegraphics[width=8cm]{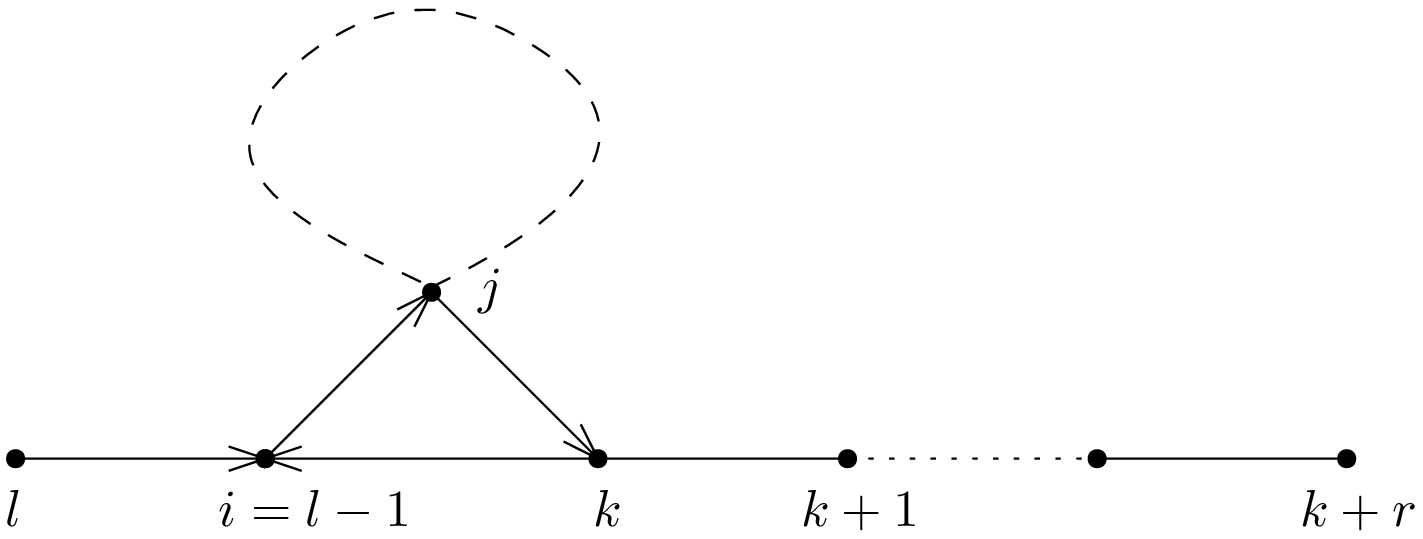}
$$
Here,
$x'_l=\dfrac{x_{l-1}+1}{x_l}=x_{P'_l}$ and if we write $\rad(P'_{l-1})=S_l\oplus M$,
we have $g(x_{P_l})=\dfrac{x_{l}+x_M}{x_{l-1}}$.
Then
$x_lx_{P'_{l-1}}=\dfrac{x_{l}+x_M}{x_{l-1}}+x_M$
The same argument as in the first part of the proof shows that $x_{M}$ is generated by
the variables $\{x_i\}_i$ and $x_{P'_i}$, so we are done. \\
(B) Assume next that two arrows start from $l$, $l\to l+1$ and $l\to l-1$.
In this case, $x'_l=\dfrac{x_{l-1}x_{l+1}+1}{x_l}=x_{P'_l}$.
We show that $g(x_{P_l})$ is generated by the $\{x_i\}_i$ and the $x_{P'_i}$.
Since
$g(x_{P_l})=\dfrac{1+x_{P'_{l+1}}x_{P'_{l-1}}}{\dfrac{x_{l-1}x_{l+1}+1}{x_l}}$ we get
$g(x_{P_l})=x_l+x_lx_{P'_{l+1}}x_{P'_{l-1}}-x_{l-1}x_{l+1}g(x_{P_l})$
and so it is enough to show that $x_{l-1}x_{l+1}g(x_{P_l})$ is generated by the
$\{x_i\}_i$ and by the $\{x_{P'_i}\}_i$.
We have three cases to consider. \\
(B.1) Assume that the arrow $l\to l+1$ is the only arrow incident with $l+1$.
In this case $g(x_{P_l})=\dfrac{x_l^2+x_lx_M+x_l+x_M(1+x_{l-1}x_{l+1})}{x_lx_{l-1}x_{l+1}}$,
where $M$ is defined through $\rad(P'_{l-1})=S_l\oplus M$.
Then $x_{l-1}x_{l+1}g(x_{P_l})=x_l+x_M+x_{l-1}x_{P'_{l-1}}$. The same argument as in the
first part of the proof shows that $x_{M}$ is generated by the $\{x_i\}_i$ and by the $\{x_{P'_i}\}_i$
and the result follows. \\
(B.2) Assume that $l+1$ is a sink with two arrows $l+2\to l+1$ and $l\to l+1$.
In this case
$g(x_{P_l})=\dfrac{x_l^2+x_lx_M+x_lx_{l+2}+x_{l+2}x_M(1+x_{l-1}x_{l+1})}{x_lx_{l-1}x_{l+1}}$,
where $M$ is defined through $\rad(P'_{l-1})=S_l\oplus M$.
Then $x_{l-1}x_{l+1}g(x_{P_l})=x_l+x_M+x_{l+2}x_{l-1}x_{P'_{l-1}}$.
The same argument as in the first part of the proof shows that $x_{M}$ is generated
by the $\{x_i\}_i$ and by the $\{x_{P'_i}\}_i$
and the result follows.

\noindent
(B.3) Assume thirdly that the arrows incident with $l+1$ are $l\to l+1$ and $l+1\to l+2$.
In this case,
$g(x_{P_l})=\dfrac{x_l^2+x_lx_M+x_lx_{P'_{l+2}}+x_{P'_{l+2}}x_M(1+x_{l-1}x_{l+1})}{x_lx_{l-1}x_{l+1}}$,
where $\rad(P'_{l-1})=S_l\oplus M$.
Then $x_{l-1}x_{l+1}g(x_{P_l})=x_l+x_M+x_{l-1}x_{P'_{l+2}}x_{P'_{l-1}}$.
The same argument as in the first part of the proof shows that $x_{M}$ is generated
by the $x_i, x_{P'_i}$
and the result follows.
\end{proof}

Next we have three lemmas considering the cases where the three vertices
$i,j,k$ are in a 2-path but not in a 3-cycle and where we consider what happens under mutation
at the vertex $i$.

\begin{lemma}\label{lemm6}
Let $Q\in \mathcal{Q}_n$ be a quiver as follows.
$$
\includegraphics[width=2cm]{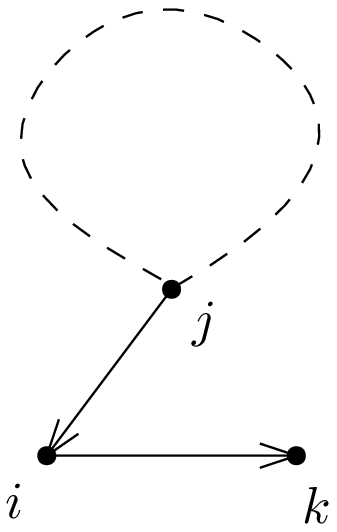}
$$
Let $\Sigma=(\mathbf{x}, Q)$ and $\Sigma'=(\mathbf{x}, \mu_i(Q))$. If $\mathcal{L}_P(\Sigma)=\mathcal{A}(\Sigma)$, then $\mathcal{L}_P(\Sigma')=\mathcal{A}(\Sigma')$.
\end{lemma}
\begin{proof}
Let $g:\mathcal{A}(\mathbf{x}, Q)\rightarrow \mathcal{A}(\mathbf{x'}, Q)$ be the
isomorphism from Theorem~\ref{thm2},
where $\mathbf{x'}=\mathbf{x}\cup \{\mu_i(x_i)\}\backslash \{x_i\}$. By Theorem \ref{thm2}, it
is enough to show that the two cluster variables $x_i'$ and $g(x_{P_i})$ are generated
by $x_1, x_2, \cdots, x_n, x_{P'_1}, x_{P'_2},\cdots, x_{P'_n}$. Now $x_i'=\dfrac{x_j+x_k}{x_i}$
and an easy calculation shows that $g(x_{P_i})=\dfrac{x_i+x_j}{x_k}$.
We have $x_{P'_k}=\dfrac{1+x_{\rad(P'_k)}}{x_k}=\dfrac{x_i+x_j+x_k}{x_ix_k}$.
Therefore $\dfrac{x_i+x_j}{x_k}=x_ix_{P'_k}-1$, $\dfrac{x_j+x_k}{x_i}=x_kx_{P'_k}-1$ and
the result follows.
\end{proof}

\begin{lemma}\label{lemm7}
Let $Q\in \mathcal{Q}_n$ be a quiver as follows (with arbitrary orientation between $k+1$ and $k+r$)
$$
\includegraphics[width=7cm]{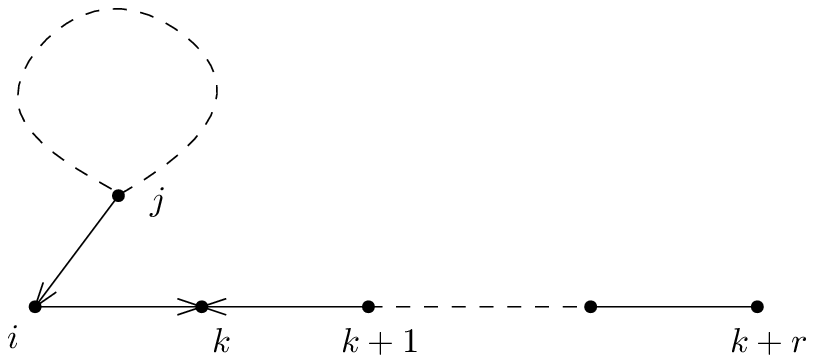}
$$
Let $\Sigma=(\mathbf{x}, Q)$ and $\Sigma'=(\mathbf{x}, \mu_i(Q))$. If
$\mathcal{L}_P(\Sigma)=\mathcal{A}(\Sigma)$, then $\mathcal{L}_P(\Sigma')=\mathcal{A}(\Sigma')$.
\end{lemma}
\begin{proof}
Let $\mathbf{x'}=\mathbf{x}\cup \{\mu_i(x_i)\}\backslash \{x_i\}$ and
$g:\mathcal{A}(\mathbf{x}, Q)\rightarrow \mathcal{A}(\mathbf{x'}, Q)$ be the isomorphism from
Remark~\ref{rem:induced-isom}. By
Theorem~\ref{thm2}, it is enough to show that the two cluster variables
$x_i'$ and $g(x_{P_i})$ are generated by the variables
$x_1, x_2, \cdots, x_n, x_{P'_1}, x_{P'_2}, \cdots, x_{P'_n}$.
Now
$x_i'=\dfrac{x_j+x_k}{x_i}$ and one can calculate that $g(x_{P_i})=\dfrac{x_i+x_jx_{k+1}}{x_k}$.
We have $x_{P'_k}=\dfrac{1+x_{k+1}x_{\rad(P'_k)}}{x_k}=\dfrac{x_i+x_{k+1}(x_j+x_k)}{x_ix_k}$.
Then $\dfrac{x_i+x_jx_{k+1}}{x_k}=x_ix_{P'_k}-x_{k+1}$ and $x_{P'_i}=\dfrac{1+x_{\rad(P'_i)}}{x_i}$.
We have four cases to consider. \\
Case (1): $j$ is a source vertex in $Q$ and has only one arrow (the arrow $j\to i$).
In this case $x_{\rad(P'_i)}=\dfrac{x_i+x_k}{x_j}$ and so $x_{P'_i}=\dfrac{x_j+x_i+x_k}{x_ix_j}$.
Then $\dfrac{x_j+x_k}{x_i}=x_jx_{P'_i}-1$ and the result follows. \\
Case (2): $j$ is a source vertex in $Q$, with two arrows $j\to i$ and $j\to j+1$ from $j$.
In this case $x_{\rad(P'_i)}=\dfrac{x_ix_{P'_{j+1}}+x_k}{x_j}$ and so
$x_{P'_i}=\dfrac{x_j+x_ix_{P'_{j+1}}+x_k}{x_ix_j}$.
Then $\dfrac{x_j+x_k}{x_i}=x_jx_{P'_i}-x_{P'_{j+1}}$ and the result follows. \\
Case (3): The vertex $j$ has two arrows, $j\to i$ and $j+1\to j$.
In this case $x_{\rad(P'_i)}=\dfrac{x_{j+1}x_i+x_k}{x_j}$,
and so $x_{P'_i}=\dfrac{x_j+x_{j+1}x_i+x_k}{x_ix_j}$.
Then $\dfrac{x_j+x_k}{x_i}=x_jx_{P'_i}-x_{j+1}$ and the result follows. \\
Case (4): There are three arrows incident with $j$.
In this case $x_{\rad(P'_i)}=\dfrac{x_ix_{M}+x_k}{x_j}$, where $M$ is defined by
$\rad(P'_j)=S_k\oplus M$.
Then $x_{P'_i}=\dfrac{x_j+x_{M}x_i+x_k}{x_ix_j}$ and so $\dfrac{x_j+x_k}{x_i}=x_jx_{P'_i}-x_{M}$.
The same argument as in the first part of the proof of Lemma \ref{lemm5} shows that $x_{M}$ is
generated by $x_1, x_2, \cdots, x_n, x_{P'_1}, x_{P'_2}, \cdots, x_{P'_n}$ and the result follows.
\end{proof}

\begin{lemma}\label{lemm8}
Let $Q\in \mathcal{Q}_n$ be a quiver as follows.
$$
\includegraphics[width=9cm]{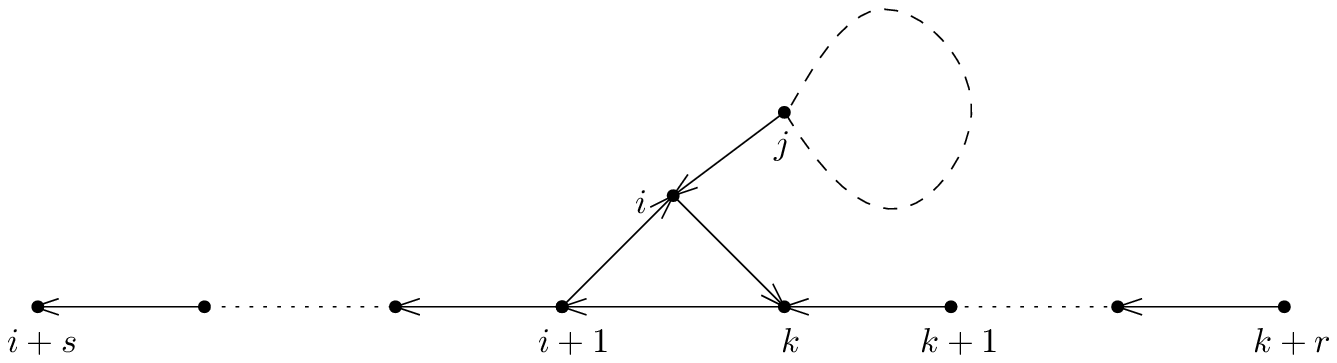}
$$
Let $\Sigma=(\mathbf{x}, Q)$ and $\Sigma'=(\mathbf{x}, \mu_i(Q))$. If
$\mathcal{L}_P(\Sigma)=\mathcal{A}(\Sigma)$, then
$\mathcal{L}_P(\Sigma')=\mathcal{A}(\Sigma')$.
\end{lemma}

\begin{proof}
Let $\mathbf{x'}=\mathbf{x}\cup \{\mu_i(x_i)\}\backslash \{x_i\}$ and
$g:\mathcal{A}(\mathbf{x}, Q)\rightarrow \mathcal{A}(\mathbf{x'}, Q)$ be the
isomorphism from Remark~\ref{rem:induced-isom}.

By Theorem \ref{thm2}, it is enough to show that the cluster variables $x_i'$ and $g(x_{P_i})$
are generated by $x_1, x_2, \cdots, x_n, x_{P'_1}, x_{P'_2}, \cdots, x_{P'_n}$.
Now $x_i'=\dfrac{x_{i+1}x_j+x_k}{x_i}$ and one can check that
$g(x_{P_i})=\dfrac{x_i+x_jx_{k+1}}{x_k}$.
Also
$x_{P'_k}=\dfrac{1+x_{k+1}x_{\rad(P'_k)}}{x_k}$ and $x_{\rad(P'_k)}=\dfrac{x_kx_{P'_{i+1}}+x_j}{x_i}$
Then $x_ix_{P'_k}=\dfrac{x_i+x_{k+1}x_j}{x_k}+x_{k+1}x_{P'_{i+1}}$ and so
$g(x_{P_i})=x_ix_{P'_k}-x_{k+1}x_{P'_{i+1}}$.
Now we show that $\dfrac{x_{i+1}x_j+x_k}{x_i}$ is generated
by $x_1, x_2, \cdots, x_n, x_{P'_1}, x_{P'_2}, \cdots, x_{P'_n}$.
Note that $x_{P'_i}=\dfrac{1+x_{\rad(P'_i)}}{x_i}$. We have four cases. \\
Case (1): $j$ is a source vertex in $Q$ and there are no other vertices incident with $j$.
In this case $x_{\rad(P'_i)}=\dfrac{x_i+x_k}{x_j}x_{P'_{i+1}}$.
Also, $x_{P'_{i+1}}=\dfrac{1+x_ix_{P'_{i+2}}}{x_{i+1}}$ and so we get
$x_{P'_i}=\dfrac{x_jx_{i+1}+x_i+x_k+(x_i+x_k)x_ix_{P'_{i+2}}}{x_ix_jx_{i+1}}$.
Thus $\dfrac{x_{i+1}x_j+x_k}{x_i}=x_jx_{i+1}x_{P'_i}-1-(x_i+x_k)x_{P'_{i+2}}$ and the result follows. \\
Case (2): $j$ is a source vertex in $Q$, with two arrows
$j\to i$ and $j\to j+1$.
In this case $x_{\rad(P'_i)}=\dfrac{x_ix_{P'_{j+1}}+x_k}{x_j}x_{P'_{i+1}}$ and so
$x_{P'_i}=\dfrac{x_jx_{i+1}+x_ix_{P'_{j+1}}+x_k+(x_ix_{P'_{j+1}}+x_k)x_ix_{P'_{i+2}}}{x_ix_jx_{i+1}}$.
Then $\dfrac{x_{i+1}x_j+x_k}{x_i}=x_jx_{i+1}x_{P'_i}-x_{P'_{j+1}}-(x_ix_{P'_{j+1}}+x_k)x_{P'_{i+2}}$
and the result follows. \\
Case (3): There are two arrows at $j$, one $j\to i$ and one $j+1\to j$.
In this case $x_{\rad(P'_i)}=\dfrac{x_{j+1}x_i+x_k}{x_j}x_{P'_{i+1}}$,
and
so $x_{P'_i}=\dfrac{x_jx_{i+1}+x_ix_{j+1}+x_k+(x_ix_{j+1}+x_k)x_ix_{P'_{i+2}}}{x_ix_jx_{i+1}}$.
Then $\dfrac{x_{i+1}x_j+x_k}{x_i}=x_jx_{i+1}x_{P'_i}-x_{j+1}-(x_ix_{j+1}+x_k)x_{P'_{i+2}}$ and
the result follows. \\
Case (4): There are three arrows at $j$, two starting at $j$, one ending at $j$.
In this case $x_{\rad(P'_i)}=\dfrac{x_ix_{M}+x_k}{x_j}x_{P'_{i+1}}$, where $M$ is defined through
$rad(P'_j)=S_k\oplus M$.
Then $x_{P'_i}=\dfrac{x_jx_{i+1}+x_{M}x_i+x_k+x_ix_{P'_{i+2}}(x_ix_M+x_k)}{x_ix_{i+1}x_j}$ and so
$\dfrac{x_{i+1}x_j+x_k}{x_i}=x_jx_{i+1}x_{P'_i}-x_{M}-(x_ix_{M}+x_k)x_{P'_{i+2}}$.
The same argument as in the first part of the proof of Lemma \ref{lemm5} shows that $x_{M}$
is generated by $x_1, x_2, \cdots, x_n, x_{P'_1}, x_{P'_2}, \cdots, x_{P'_n}$ and the result follows.
\end{proof}

\begin{theorem}\label{thm3}
Let $Q$ be a quiver which is mutation equivalent to a quiver of type $A_n$ and
$\Sigma=(\mathbf{x}, Q)$. Then $\mathcal{L}_P(\Sigma)=\mathcal{A}(\Sigma)$.
\end{theorem}

\begin{proof} We use induction on the number of $3$-cycles in $Q$. If $Q$ is acyclic, then the
result follows by Theorem \ref{thm1}. Assume that $Q$ has $m$ $3$-cycles and that
the result follows for any quiver $Q'\in \mathcal{Q}_n$ with $t<m$ $3$-cycles.
It is known that there exists at least one $3$-cycle in $Q$ which is connected to at most one other
3-cycle, \cite[Lemma 4.3]{BV}.
Let this 3-cycle be $i\to j\to k\to i$, as drawn here.
$$
\includegraphics[width=10cm]{pic7-kb}
$$
We also assume that for this $3$-cycle in $Q$, the sum $s+r$ of lengths of the legs at vertices
$i$ and $k$ is minimal.
By Lemmas \ref{lemm3} and \ref{lemm5} we can assume that $Q$ has a following orientation
(after suitable mutations)
$$
\includegraphics[width=10cm]{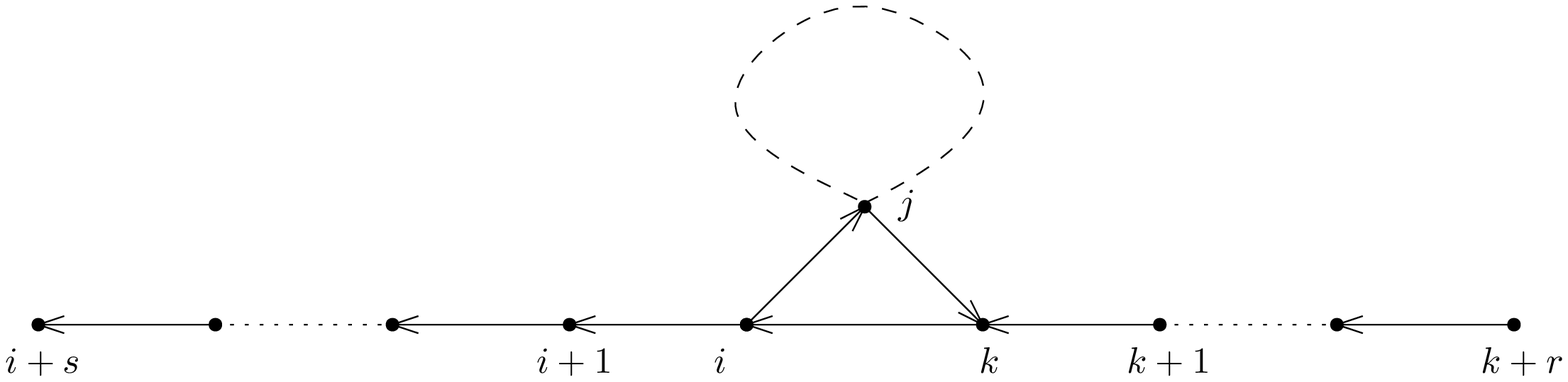}
$$
We will use induction on the number $r+s$ to show $\mathcal{L}_P(\Sigma)=\mathcal{A}(\Sigma)$.
Assume that $s+r=0$. We consider mutation at $i$: the quiver $\mu_i(Q)$ has $m-1$
3-cycles and by the induction assumption,
$\mathcal{L}_P(\Sigma')=\mathcal{A}(\Sigma')$, for $\Sigma'=(\mathbf{x}, \mu_i(Q))$.
By Lemma $\ref{lemm6}$, $\mathcal{L}_P(\Sigma)=\mathcal{A}(\Sigma)$ as we wanted.

Now assume that for any $Q\in \mathcal{Q}_n$ with $m$, $3$-cycles and $s'+r'<s+r$ the claim
is proved.
If $s=0$, then $\mu_i(Q)$ has $m-1$, $3$-cycles and by the induction assumption
$\mathcal{L}_P(\Sigma')=\mathcal{A}(\Sigma')$, where $\Sigma'=(\mathbf{x}, \mu_i(Q))$.
Then by Lemma $\ref{lemm7}$, $\mathcal{L}_P(\Sigma)=\mathcal{A}(\Sigma)$ as we wanted.
If $s\neq 0$, the mutated quiver $\mu_i(Q)$ still has $m$, $3$-cycles but the length
$s'+r'=s+r-1$ has decreased and by induction assumption
$\mathcal{L}_P(\Sigma')=\mathcal{A}(\Sigma')$, for $\Sigma'=(\mathbf{x}, \mu_i(Q))$.
Then by Lemma $\ref{lemm8}$, $\mathcal{L}_P(\Sigma)=\mathcal{A}(\Sigma)$ and the result follows.
\end{proof}

\section{Cluster algebras of type $\widetilde{A}_n$}\label{sec 5}

In this section we show that in type $\widetilde{A}_n$, the lower bound cluster algebras generated
by projectives and the cluster algebra coincide (Theorem~\ref{thm4}).
For this, we first recall the characterisation of Bastian~\cite{B} of
the mutation-equivalence class of quivers in type $\widetilde{A}_n$. These quivers all have one
non-oriented cycle to which a number of quivers of mutation type $A$ are attached along 3-cycles.
We will prove the theorem
by reducing the number of 3-cycles. For this, we will first analyse a few key examples.

\begin{definition}\label{def1}(\cite[Definition 3.3]{B})
Let $\mathscr{Q}_n$ be the class of connected quivers with $n+1$ vertices that
satisfy the following conditions:
\begin{itemize}
\item[(i)] There exists precisely one full subquiver that is a non-oriented cycle of length greater than or equal $2$. Thus, if the length is two, it is a double arrow.
\item[(ii)] For each arrow $\xymatrix{{x}\ar[r]^{\alpha}&{y}}$ in this non-oriented cycle, there may (or may not) be
a vertex $z_\alpha$ not on the non-oriented cycle and such that there is an oriented
$3$-cycle of the form $$\hskip.5cm \xymatrix{
&{z_{\alpha}} \ar[dl]&\\
{x}\ar[rr]_{\alpha}&&{y}\ar[ul]
}
\hskip.5cm$$
Apart from the arrows of these oriented 3-cycles there are no other arrows
incident to vertices on the non-oriented cycle.
\item[(iii)] If we remove all vertices in the non-oriented cycle and their incident arrows,
the result is a (possibly empty) disjoint union of quivers $Q_1, Q_2, \cdots$, one for each
$z_{\alpha}$. We call these $Q_{\alpha}$.
Each $Q_{\alpha}$ is of mutation type $A_{k_{\alpha}}$, with
$k_{\alpha}\ge 1$, and there are at most two arrows incident with $z_{\alpha}$ in $Q_{\alpha}$.
Furthermore, if $z_{\alpha}$ is incident with two arrows in $Q_{\alpha}$, then it is a vertex
in an oriented 3-cycle in $Q_{\alpha}$.
\end{itemize}

\end{definition}

If the non-oriented cycle consists of two arrows, then our convention is that any 3-cycle is only attached
to one of the pair of arrows (see Lemma~\ref{lemm11} for an illustration).

\begin{lemma}\label{lemm9}(\cite[Lemma 3.5 and Remark 3.6]{B}) A quiver $Q$ is mutation
equivalent to $\widetilde{A}_n$ if and only if $Q\in  \mathscr{Q}_n$.
\end{lemma}

Note that in particular, the cluster algebras of type $\widetilde{A}_n$ are acyclic, as
the set $\mathscr{Q}_n$ contains quivers which are non-oriented cycles with $n+1$ vertices.

\begin{lemma}\label{lemm11}
For $Q$ the quiver

\begin{center}
$
 \xymatrix{
&{1} \ar[dl]&\\
{2}\ar@/^-1pc/[rr]\ar@/^1pc/[rr]&&{3}\ar[ul]
}
$
\end{center}
and $\Sigma=(\mathbf{x}, Q)$ we have
$\mathcal{L}_P(\Sigma)=\mathcal{A}(\Sigma)$.
\end{lemma}
\begin{proof}
First note that $\mu_1(Q)$ is acyclic and so
$\mathcal{A}(\mathbf{x}, \mu_1(Q))=\mathcal{L}_P(\mathbf{x}, \mu_1(Q))$ by Theorem~\ref{thm1}.

Let $g:\mathcal{A}(\mathbf{x}, \mu_1(Q))\rightarrow \mathcal{A}(\mathbf{x'}, \mu_1(Q))$ be the
isomorphism of cluster algebras of Remark~\ref{rem:induced-isom} for
$\mathbf{x'}=\mathbf{x}\cup \{\mu_1(x_1)\}\backslash \{x_1\}=\{x_1',x_2,x_3\}$.
Then by Theorem~\ref{thm2}, we have that
$\mathcal{A}(\mathbf{x},\mu_1^2(Q))=\mathcal{A}(\mathbf{x},Q)$ is generated by the variables
$x_1',x_2,x_3$, $g(x_{P_1}), x_{P_2'},x_{P_3'}$.
We thus have to show that the cluster
variables $g(x_1)=x_1'$ and $g(x_{P_1})$ are generated by $x_1,x_2,x_3$, $x_{P_1'},x_{P_2'}, x_{P_3'}$.
Now $g(x_1)=\dfrac{x_2+x_3}{x_1}$ and  $g(x_{P_1})=\dfrac{x_1+x_2^2}{x_3}$;
$x_1x_{P'_3}=\dfrac{x_1+x_2^2}{x_3}+x_2$ and
$x_{P'_1}=\dfrac{x_1+(x_3+x_2)^2}{x_1x_2x_3}$.
Then $\dfrac{x_3+x_2}{x_1}=x_2x_{P'_1}-x_{P'_3}$ and the result follows.
\end{proof}

It follows from Lemma~\ref{lemm11} that for the quivers in $\mathscr{Q}_2$, the lower bound
cluster algebra generated by projectives is equal to the cluster algebra: up to isomorphism,
the only quivers in $\mathscr{Q}_2$ are a non-oriented 3-cycle and the quiver with a double
arrow and one oriented 3-cycle from Lemma~\ref{lemm11}.

\begin{lemma}\label{lemm10}
Let $Q\in \mathscr{Q}_n$, let $i$ be a vertex of the non-oriented cycle which is not incident with any
oriented 3-cycles, not a sink nor a source.
Let $\Sigma=(\mathbf{x}, Q)$ and $\Sigma'=(\mathbf{x}, \mu_i(Q))$.
Then we have
$\mathcal{L}_P(\Sigma)=\mathcal{A}(\Sigma)$ $\Longrightarrow$
$\mathcal{L}_P(\Sigma')=\mathcal{A}(\Sigma')$.
\end{lemma}

\begin{proof}
We can assume that $Q$ looks as follows, with arbitrary orientation of the
arrows around the non-oriented cycle (apart from the arrows incident with $i$).
\[
\includegraphics[width=9cm]{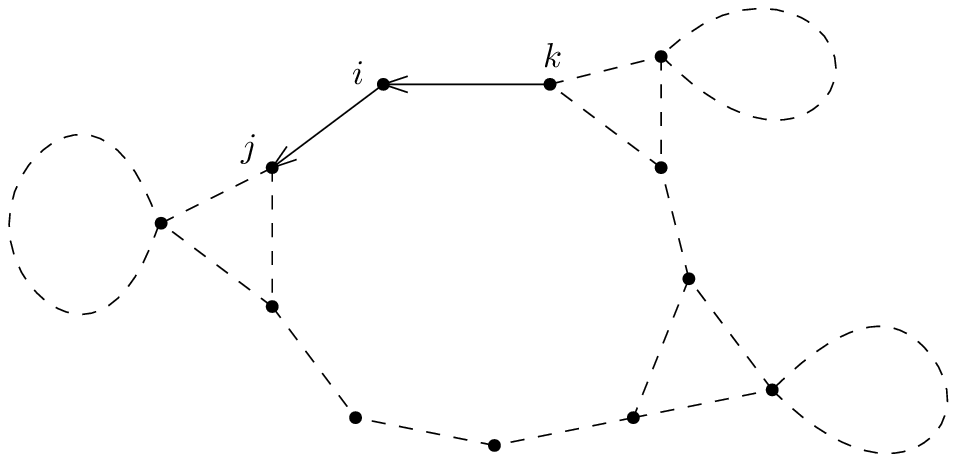}
\]
If the arrows incident with $i$ have the orientation $j\longrightarrow i\longrightarrow k$, the proof is similar.
Assume that $\mathcal{L}_P(\Sigma)=\mathcal{A}(\Sigma)$.
Let $\mathbf{x'}=\mathbf{x}\cup \{\mu_i(x_i)\}\backslash \{x_i\}$ and let
$g:\mathcal{A}(\mathbf{x}, Q)\rightarrow \mathcal{A}(\mathbf{x'}, Q)$ be the
isomorphism of cluster algebras from Remark~\ref{rem:induced-isom}.
Then we know from Theorem~\ref{thm2} that $\mathcal{A}(\mathbf{x}',\mu_i(Q))$ is
generated by the variables in $\mathbf{x}'$ and by the variables
$\{x_{P_j'}\}_j \setminus x_{P_i'}\cup g(x_{P_i})$.

To prove that $\mathcal{L}_P(\Sigma')=\mathcal{A}(\Sigma')$, we have to show that
the cluster variables $g(x_i)=x_i'$ and $g(x_{P_i})$ are generated by
$x_1, x_2, \cdots, x_{n+1}, x_{P'_1}, x_{P'_2}, \cdots, x_{P'_{n+1}}$.
We have $x_i'=\dfrac{x_j+x_k}{x_i}$ and an easy calculation shows that
$g(x_{P_i})=\dfrac{x_i+x_kx_{M}}{x_j}$, where $\rad(P'_j)=S_i\oplus M$.
Now
$x_{P'_i}=\dfrac{x_k+x_ix_{H}+x_j}{x_kx_i}$ and so $x_kx_{P'_i}=\dfrac{x_j+x_k}{x_i}+x_{H}$, where $\rad(P'_k)=N\oplus H$ that $rad(N)=M$.
We have
$x_{P'_j}=\dfrac{x_i+x_{M}x_k+x_{M}x_j}{x_ix_j}$ and hence
$x_ix_{P'_j}=\dfrac{x_i+x_kx_{M}}{x_j}+x_{M}$. The same argument as in the first part of the proof of
Lemma \ref{lemm5} shows that $x_{M}$ and $x_H$ are
generated by $x_1, x_2, \cdots, x_{n+1}, x_{P'_1}, x_{P'_2}, \cdots, x_{P'_{n+1}}$ and the result follows.
\end{proof}

We consider what happens if we mutate at a 3-valent
vertex $i$ on the non-oriented cycle. We distinguish three cases: the subquiver
for the oriented 3-cycle going through $i$ consists only of this oriented $3$-cycle (Lemma~\ref{lemm18}),
the arrow on the subquiver
for the oriented 3-cycle going through $i$ is pointing away from the 3-cycle (Lemma~\ref{lemm12})
or is pointing towards the 3-cycle (Lemma~\ref{lemm13}).

\begin{lemma}\label{lemm18}
Let $Q\in \mathscr{Q}_n$ be one of the following types of quivers.
$$\includegraphics[width=7cm]{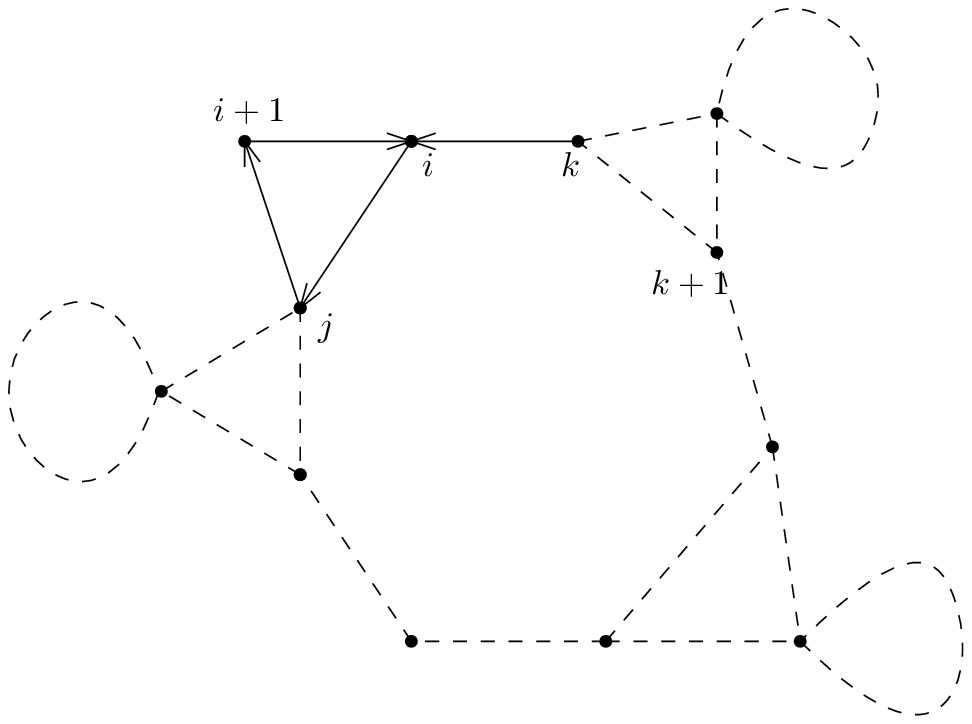}   \hspace{25pt} \includegraphics[width=7cm]{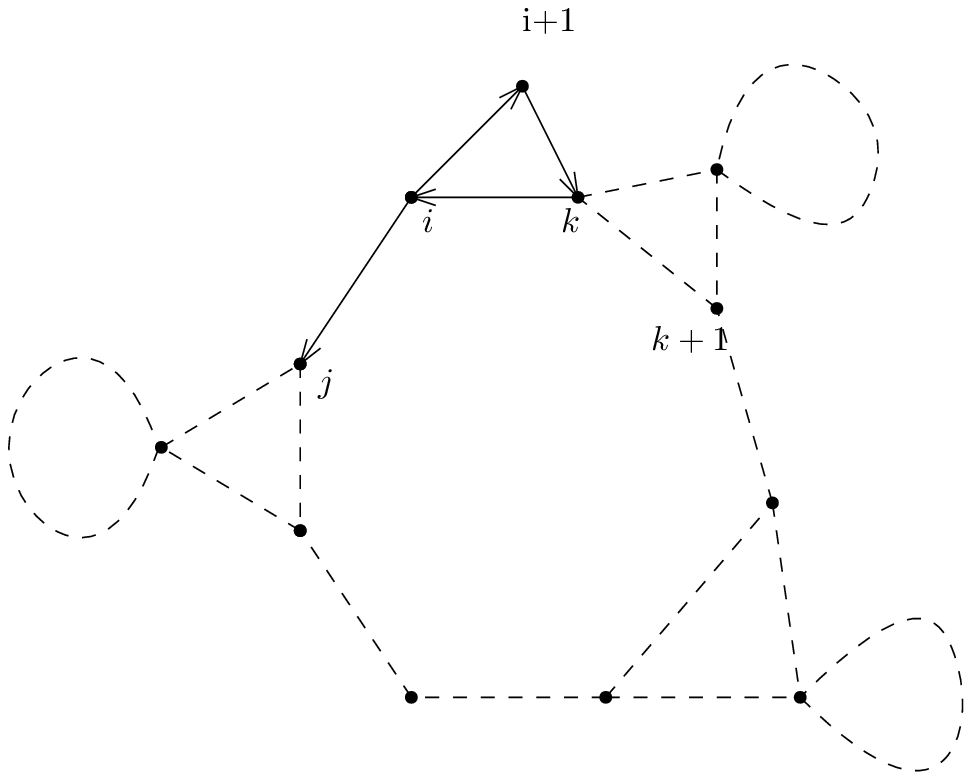}$$
Let $\Sigma=(\mathbf{x}, Q)$ and $\Sigma'=(\mathbf{x}, \mu_i(Q))$. Then we have:
$\mathcal{L}_P(\Sigma)=\mathcal{A}(\Sigma)$ $\Longrightarrow$
$\mathcal{L}_P(\Sigma')=\mathcal{A}(\Sigma')$.
\end{lemma}

\begin{proof}
We assume that $Q$ is a quiver as on the left.
The proof for quivers as on the right is similar.
Let $\mathbf{x'}=\mathbf{x}\cup \{\mu_i(x_i)\}\backslash \{x_i\}$ and let
$g:\mathcal{A}(\mathbf{x}, Q)\rightarrow \mathcal{A}(\mathbf{x'}, Q)$ be the isomorphism from
Remark~\ref{rem:induced-isom}.
As in the proof of Lemma~\ref{lemm10}, we have to show that the cluster variables $x_i'=g(x_i)$
and $g(x_{P_i})$ are generated by the $\{x_j\}_j$ and by the $\{x_{P'_j}\}_j$.
Here, $x_i'=\dfrac{x_j+x_kx_{i+1}}{x_i}$ and one can check that
$g(x_{P_i})=\dfrac{x_i+x_kx_{M}}{x_j}$, where $\rad(P'_j)=T\oplus M$ and $T$ is a representation of the subquiver $Q_{\alpha}$ of the quiver $\mu_i(Q)$ with $z_\alpha=i$.
We have $x_{P'_i}=\dfrac{x_kx_{i+1}+(1+x_i)(x_ix_{L}+x_j)}{x_ix_{i+1}x_k}$, where $\rad(P_k)=P_i\oplus L$. Also,
$x_{P'_j}=\dfrac{x_i+x_Mx_{P'_{i+1}}x_j+x_Mx_k}{x_ix_j}$. Then
$x_{i+1}x_kx_{P'_i}=\dfrac{x_j+x_kx_{i+1}}{x_i}+x_{L}+x_ix_L+x_j$ and
$x_ix_{P'_j}=\dfrac{x_i+x_kx_{M}}{x_j}+x_{M}x_{P'_{i+1}}$.
The same argument as in the first part of the proof of Lemma \ref{lemm5} shows that
$x_{L}$ and $x_M$ are generated by
$x_1, x_2, \cdots, x_{n+1}, x_{P'_1}, x_{P'_2}, \cdots, x_{P'_{n+1}}$ and the result follows.

\end{proof}

\begin{lemma}\label{lemm12} Let $Q\in \mathscr{Q}_n$ one of the following types of quivers.
$$\includegraphics[width=7cm]{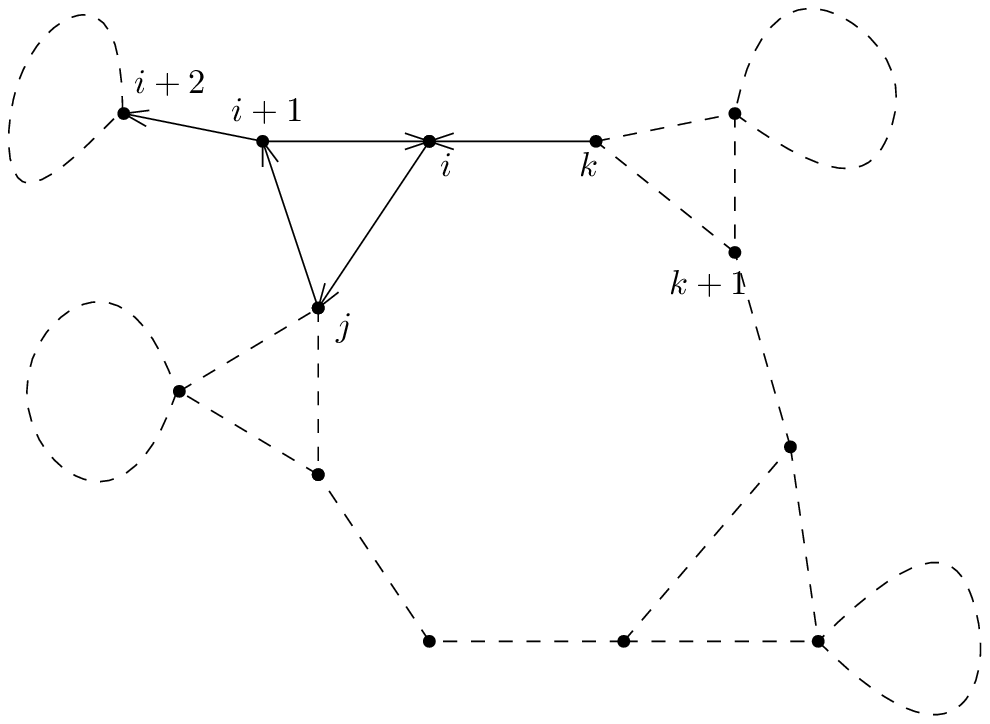}\hspace{25pt} \includegraphics[width=7cm]{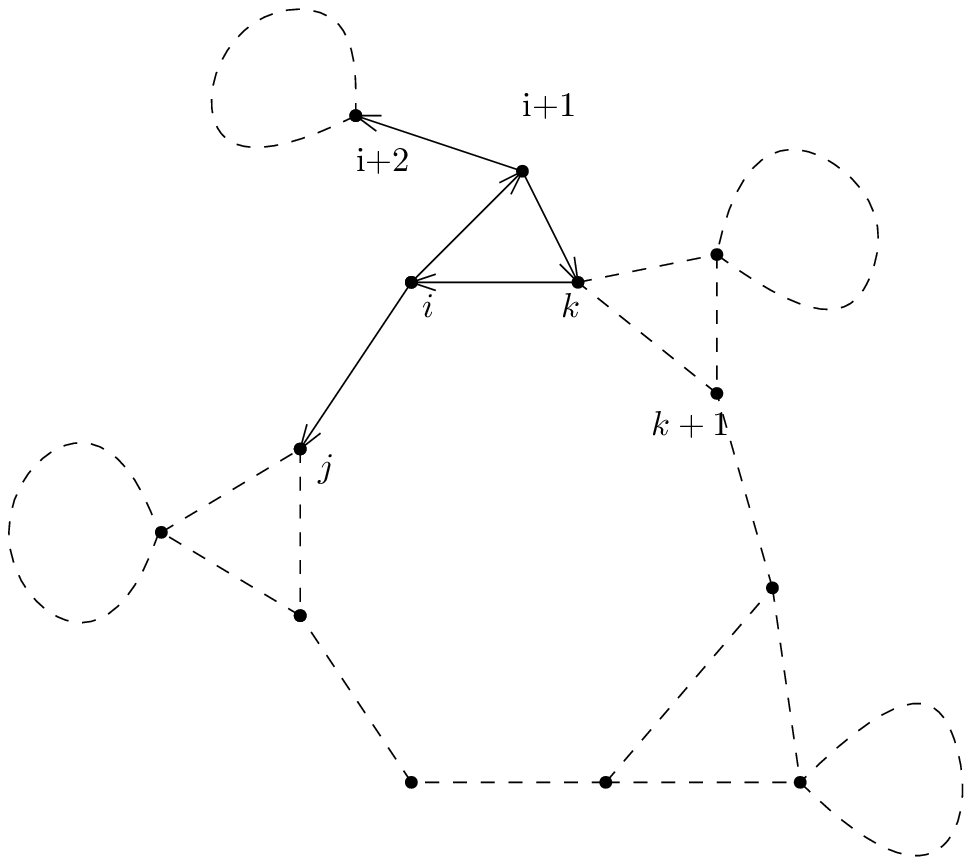}$$
Let $\Sigma=(\mathbf{x}, Q)$ and $\Sigma'=(\mathbf{x}, \mu_i(Q))$. Then we have:
$\mathcal{L}_P(\Sigma)=\mathcal{A}(\Sigma)$ $\Longrightarrow$
$\mathcal{L}_P(\Sigma')=\mathcal{A}(\Sigma')$.
\end{lemma}

\begin{proof}
We assume that $Q$ is a quiver as the one on the left, the proof for quivers as on the right is similar.
Let $\mathbf{x'}=\mathbf{x}\cup \{\mu_i(x_i)\}\backslash \{x_i\}$ and let
$g:\mathcal{A}(\mathbf{x}, Q)\rightarrow \mathcal{A}(\mathbf{x'}, Q)$ be the isomorphism from
Remark~\ref{rem:induced-isom}.

As in the previous proofs, we have to show that the cluster variables $x_i'=g(x_i)$
and $g(x_{P_i})$ are generated by the $\{x_j\}_j$ and by $\{x_{P'_j}\}_j$.
Here, $x_i'=\dfrac{x_j+x_kx_{i+1}}{x_i}$ and one can check that
$g(x_{P_i})=\dfrac{x_i+x_kx_{M}}{x_j}$, where $\rad(P'_j)=T\oplus M$ and $T$ is a representation of the subquiver $Q_{\alpha}$ of the quiver $\mu_i(Q)$ with $z_\alpha=i$.
We have $x_{P'_i}=\dfrac{x_kx_{i+1}+(1+x_Nx_i)(x_ix_{L}+x_j)}{x_ix_{i+1}x_k}$, where
$N=\rad(P'_{i+1})$ and $\rad(P_k)=P_i\oplus L$. Also,
$x_{P'_j}=\dfrac{x_i+x_Mx_{P'_{i+1}}x_j+x_Mx_k}{x_ix_j}$. Then
$x_{i+1}x_kx_{P'_i}=\dfrac{x_j+x_kx_{i+1}}{x_i}+x_{L}+x_ix_{N}x_L+x_Nx_j$ and
$x_ix_{P'_j}=\dfrac{x_i+x_kx_{M}}{x_j}+x_{M}x_{P'_{i+1}}$.
The same argument as in the first part of the proof of Lemma \ref{lemm5} shows that
$x_{L}, x_{M}$ and $x_N$ are generated by
$x_1, x_2, \cdots, x_{n+1}, x_{P'_1}, x_{P'_2}, \cdots, x_{P'_{n+1}}$ and the result follows.
\end{proof}

\begin{lemma}\label{lemm13} Let $Q\in \mathscr{Q}_n$ be one of the following types of quivers.
$$\includegraphics[width=7cm]{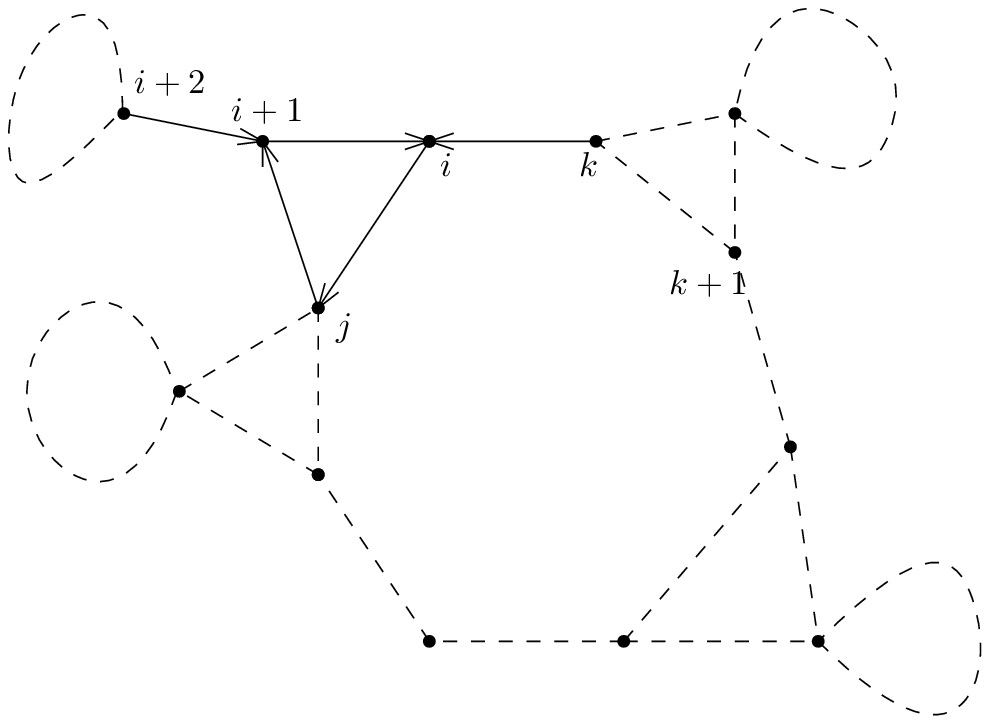} \hspace{25pt} \includegraphics[width=7cm]{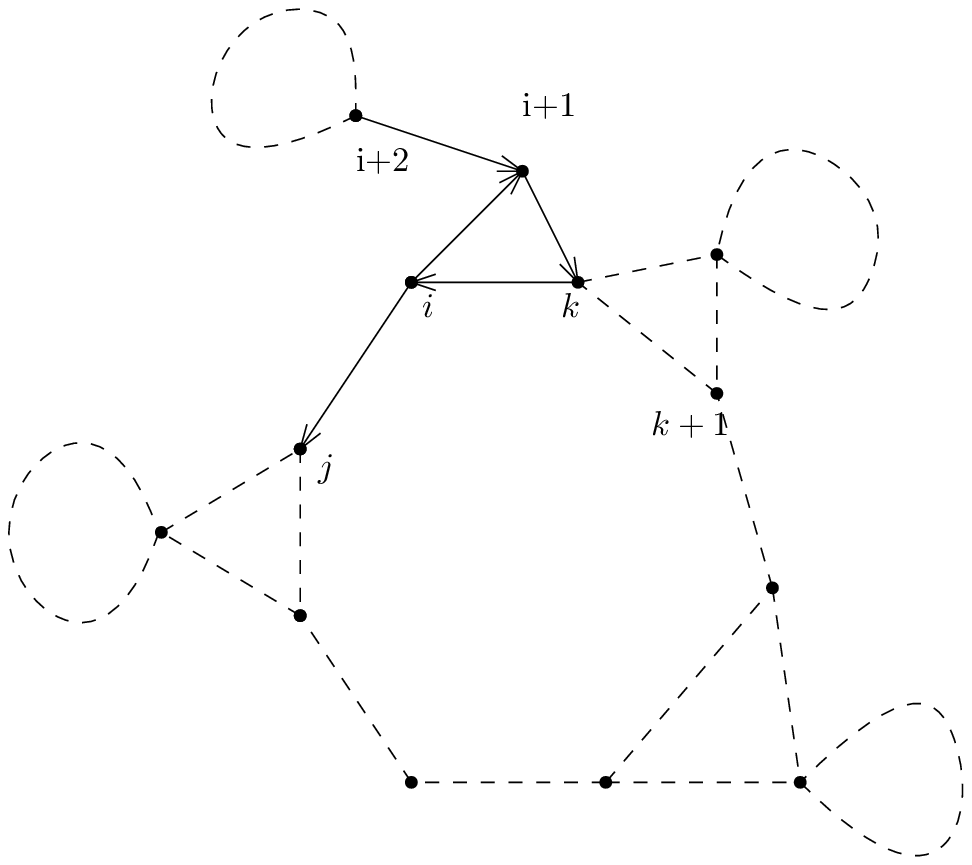} $$
Let $\Sigma=(\mathbf{x}, Q)$ and $\Sigma'=(\mathbf{x}, \mu_i(Q))$. If  $\mathcal{L}_P(\Sigma)=\mathcal{A}(\Sigma)$, then $\mathcal{L}_P(\Sigma')=\mathcal{A}(\Sigma')$.
\end{lemma}

\begin{proof} The proof is similar to the proof of Lemma \ref{lemm12}.
\end{proof}

\begin{lemma}\label{lemm14} Let $Q\in \mathscr{Q}_n$ be one of the following types of quivers.
$$
\includegraphics[width=4cm]{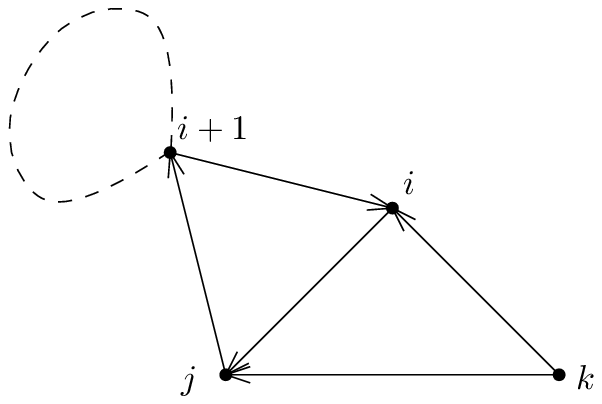}  \hspace{45pt}      \includegraphics[width=4cm]{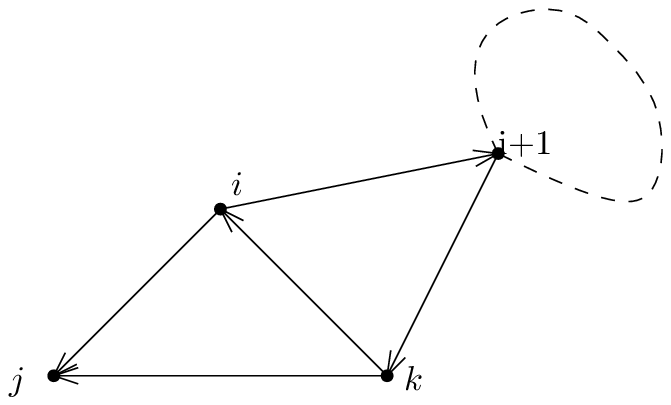}
$$
Let $\Sigma=(\mathbf{x}, Q)$ and $\Sigma'=(\mathbf{x}, \mu_i(Q))$. If
$\mathcal{L}_P(\Sigma)=\mathcal{A}(\Sigma)$, then $\mathcal{L}_P(\Sigma')=\mathcal{A}(\Sigma')$.
\end{lemma}

\begin{proof}
We prove the lemma for quivers of the type
$$
\includegraphics[width=4cm]{pic22-kb}
$$
 - the other case is similar.
Let $\mathbf{x'}=\mathbf{x}\cup \{\mu_i(x_i)\}\backslash \{x_i\}$ and
$g:\mathcal{A}(\mathbf{x}, Q)\rightarrow \mathcal{A}(\mathbf{x'}, Q)$ be the isomorphism from
Remark~\ref{rem:induced-isom}. As in the previous proofs, we have to show that the
cluster variables $g(x_i)=x_i'$ and $g(x_{P_i})$ are generated
by $x_1, x_2, \cdots, x_{n+1}$, $x_{P'_1}, x_{P'_2}, \cdots, x_{P'_{n+1}}$.
Now $g(x_i)=\dfrac{x_j+x_kx_{i+1}}{x_i}$ and
$g(x_{P_i})=\dfrac{x_i+x_k^2}{x_j}$. $x_{P'_j}=\dfrac{x_i+x_kx_{P'_{i+1}}x_j+x_k^2}{x_ix_j}$ and
hence $x_ix_{P'_j}=g(x_{P_i})+x_kx_{P'_{i+1}}$. We have
$x_{P'_i}=\dfrac{x_kx_jx_{P'_{i+1}}^2+x_k^2x_{P'_{i+1}}+x_ix_{P'_{i+1}}
+x_j^2x_{P'_{i+1}}+x_jx_k+x_ix_jx_k}{x_i^2x_{j}x_k}$.
Note that in $\mu_i(Q)$, there is no arrow between $i+1$ and $j$.
We have three cases. \\
Case (1): $i+1$ is a sink vertex in $\mu_i(Q)$ and
the arrow from $i$ is the only arrow in $\mu_i(Q)$ with target $i+1$.
In this case $x_{P'_{i+1}}=\dfrac{x_i+1}{x_{i+1}}$. One checks that
$x_ix_{i+1}x_kx_{P'_i}=\dfrac{x_j+x_kx_{i+1}}{x_i}+x_kx_{i+1}+x_j+(1+x_i)x_{P'_j}$
and the result follows. \\
Case (2): $i+1$ is a sink vertex in $\mu_i(Q)$ and there are exactly two arrows
$\alpha, \beta\in (\mu_i(Q))_1$ with target $i+1$.
By the same argument as in the case (1), we can see that
$\dfrac{x_j+x_kx_{i+1}}{x_i}\in \mathcal{L}_P(\Sigma')$. \\
Case (3): $i+1$ is not a sink vertex in $\mu_i(Q)$.
Since $Q\in \mathscr{Q}_n$, there is exactly one arrow in $\mu_i(Q)$ starting at $i+1$.
Let $i+2$ be its target.
Then we have $x_{P'_{i+1}}=\dfrac{x_{P'_{i+2}}x_i+1}{x_{i+1}}$ and so
$x_ix_{i+1}x_kx_{P'_i}=\dfrac{x_j+x_kx_{i+1}}{x_i}+x_ix_{P'_{i+2}}+x_kx_{i+1}+(1+x_ix_{P'_{i+2}})x_{P'_j}$
and the result follows.
\end{proof}

It remains to consider quivers where the non-oriented cycle consists of two arrows and where both
are part of a 3-cycle, possibly with further arrows. We do this in the following three lemmas.

\begin{lemma}\label{lemm15} Let $Q\in \mathscr{Q}_n$ be the following quiver.
$$
\includegraphics[width=3cm]{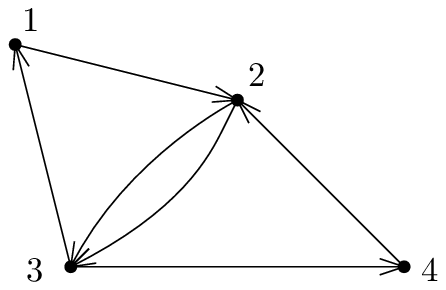}
$$
If $\Sigma=(\mathbf{x}, Q)$ then
$\mathcal{L}_P(\Sigma)=\mathcal{A}(\Sigma)$.
\end{lemma}

\begin{proof}
Let $Q''=\mu_1\mu_4(Q)$, i.e. $Q''$ is the quiver
$$
\includegraphics[width=3cm]{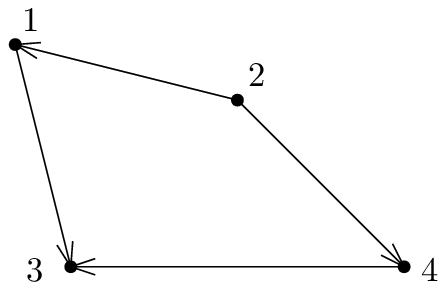}
$$
as in Example~\ref{ex:projective-cl-var}.
By Theorem~\ref{thm1}, $\mathcal{L}_P(\Sigma'')=\mathcal{A}(\Sigma'')$,
where $\Sigma''=(\mathbf{x}, Q'')$.
First we show that $\mathcal{L}_P(\Sigma')=\mathcal{A}(\Sigma')$, where $\Sigma'=(\mathbf{x}, Q')$ and $Q'=\mu_4(Q)$. Let $\mathbf{x'}=\mathbf{x}\cup \{\mu_4(x_4)\}\backslash \{x_4\}$ and
$g:\mathcal{A}(\mathbf{x}, Q'')\rightarrow \mathcal{A}(\mathbf{x'}, Q'')$ be the isomorphism from
Remark~\ref{rem:induced-isom}. As in the previous proofs, we have to show that the
cluster variables $g(x_4)=\dfrac{x_2+x_3}{x_4}$ and $g(x_{P_4})$ are generated
by $x_1, \cdots, x_{4}$, $x_{P'_1}, \cdots, x_{P'_{4}}$.
We calculate $g(x_{P_4})=\dfrac{x_4+x_1x_2}{x_3}$,
$x_{P'_3}=\dfrac{x_4+x_1x_2+x_1x_3}{x_3x_4}$ and
$x_{P'_4}=\dfrac{x_4x_{P'_1}+x_2+x_3}{x_2x_4}$.
Therefore $g(x_{P_4})=x_4x_{P'_3}-x_1$, $\dfrac{x_2+x_3}{x_4}=x_2x_{P'_4}-x_{P'_1}$ and hence
$\mathcal{L}_P(\Sigma')=\mathcal{A}(\Sigma')$.
Now let $\mathbf{x'}=\mathbf{x}\cup \{\mu_1(x_1)\}\backslash \{x_1\}$ and
$g:\mathcal{A}(\mathbf{x}, Q')\rightarrow \mathcal{A}(\mathbf{x'}, Q')$ be the isomorphism from
Remark~\ref{rem:induced-isom}. A similar argument as above shows that $g(x_1)=\dfrac{x_2+x_3}{x_1}$ and $g(x_{P_1})$ are generated
by $x_1, \cdots, x_{4}$, $x_{P'_1}, \cdots, x_{P'_{4}}$ and the result follows.

\end{proof}

Recall that if $\alpha$ is an arrow of the non-oriented cycle of $Q\in \mathscr{Q}_n$
we write $Q_{\alpha}$
for the subquiver attached to $\alpha$ (not including $\alpha$).

\begin{lemma}\label{lemm16}
Let $Q\in \mathscr{Q}_n$ be one of the following types of quivers
where $Q_{\alpha}$ and $Q_{\beta}$ are acyclic quivers of type $A$.
$$
\includegraphics[width=6cm]{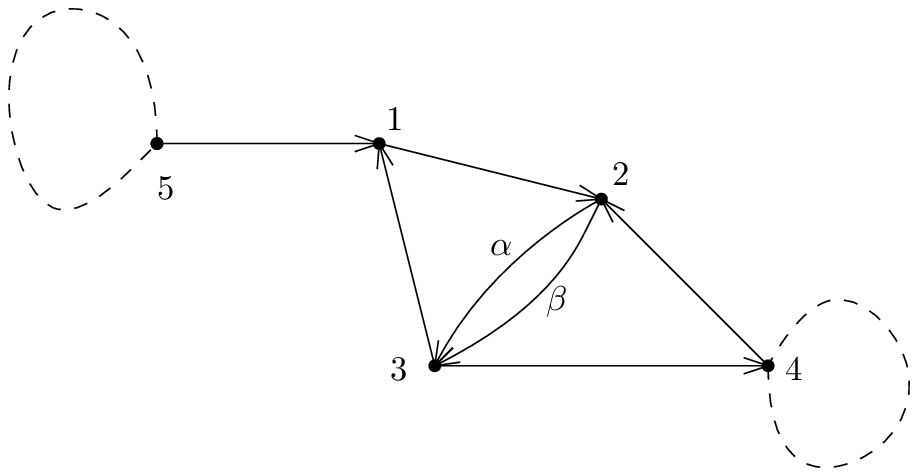}  \hspace{35pt}        \includegraphics[width=6cm]{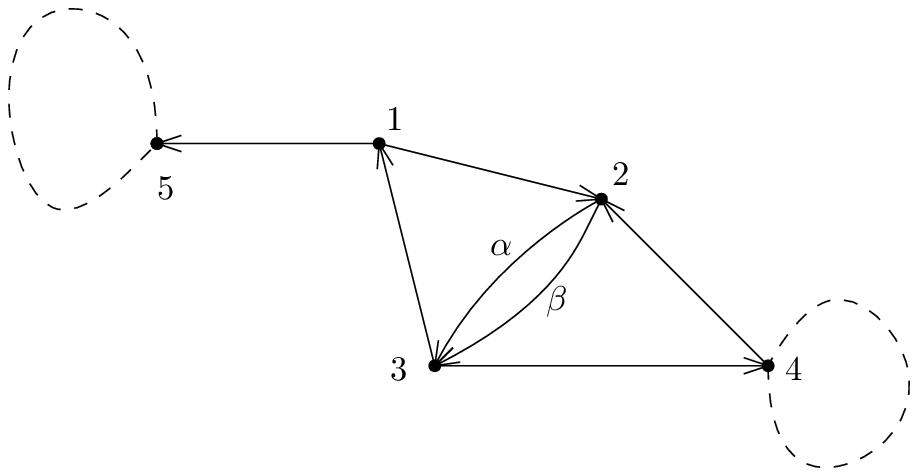}
$$
Let $\Sigma=(\mathbf{x}, Q)$, $\Sigma'=(\mathbf{x}, Q')$ and $Q'=\mu_1(Q)$. If $\mathcal{L}_P(\Sigma')=\mathcal{A}(\Sigma')$ then $\mathcal{L}_P(\Sigma)=\mathcal{A}(\Sigma)$.
\end{lemma}

\begin{proof}
It is enough to prove the lemma for the quiver on the left hand, as the other case is similar.
First assume that $Q_{\beta}=A_1$ is a vertex $4$. Let $\mathbf{x'}=\mathbf{x}\cup \{\mu_1(x_1)\}\backslash \{x_1\}$ and
$g:\mathcal{A}(\mathbf{x}, Q')\rightarrow \mathcal{A}(\mathbf{x'}, Q')$ be the isomorphism from
Remark~\ref{rem:induced-isom}. Then $g(x_1)=x_1'=\dfrac{x_2+x_3x_5}{x_1}$ and $g(x_{P_1})=\dfrac{x_1x_4+x_2(x_2+x_3)}{x_3x_4}$. As in the previous proofs, an direct
calculation shows that the cluster variables $g(x_1)$ and $g(x_{P_1})$ are generated
by $x_1, \cdots, x_{n+1}$, $x_{P'_1}, \cdots, x_{P'_{n+1}}$.

Now assume that $Q_{\beta}=A_t$, $t>1$. We have two possible cases:
Either there is an arrow $6\to 4$ or an arrow $4\to 6$ in $Q_{\beta}$, for some
vertex $6$ of $Q_{\beta}$.
In both cases one checks that the cluster variables $g(x_1)$ and $g(x_{P_1})$ are generated
by $x_1, \cdots, x_{n+1}$, $x_{P'_1}, \cdots, x_{P'_{n+1}}$ and the result follows.

\end{proof}

\begin{lemma}\label{lemm17}
Let $Q\in \mathscr{Q}_n$ be a quiver as follows where
$Q_{\beta}$ is acyclic of type $A$.
$$\includegraphics[width=4cm]{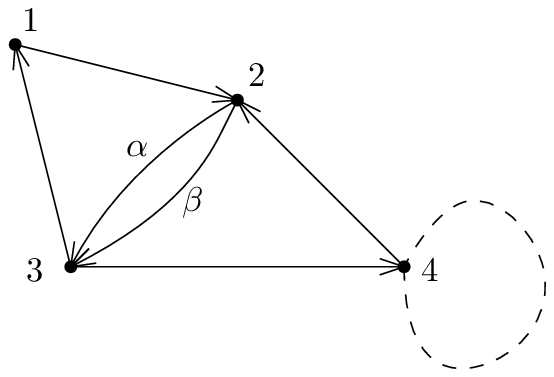}$$
Let $\Sigma=(\mathbf{x}, Q)$, $\Sigma'=(\mathbf{x}, \mu_1(Q))$. If $\mathcal{L}_P(\Sigma')=\mathcal{A}(\Sigma')$ then $\mathcal{L}_P(\Sigma)=\mathcal{A}(\Sigma)$.
\end{lemma}
\begin{proof} The proof is similar to the proof of the Lemma~\ref{lemm16}
\end{proof}

Now we are ready to prove the main theorem of this section.

\begin{theorem}\label{thm4} Let $Q$ be a quiver which is mutation equivalent to $\widetilde{A}_n$ and $\Sigma=(\mathbf{x}, Q)$. Then $\mathcal{L}_P(\Sigma)=\mathcal{A}(\Sigma)$.
\end{theorem}
\begin{proof} We use induction on the number of oriented $3$-cycles in $Q$. If $Q$ is acyclic,
the result follows by Theorem~\ref{thm1}.
Assume that $Q$ has $m$ oriented $3$-cycles, for some $m\ge 1$, and that the result follows for
every quiver $Q'\in \mathscr{Q}_n$ with $t<m$ oriented $3$-cycles.

Since $Q$ is of type $\widetilde{A}_n$, there exists precisely one full subquiver of $Q$ which is a
non-oriented cycle, it is of length $\ge 2$.
If we remove all vertices in the non-oriented cycle and their incident arrows, we obtain
a disjoint union of quivers $Q_1, Q_2, \cdots$ of type $A_{k_\alpha}$ for some $k_\alpha\geq1$.
Let $l$ be the minimum of all $k_\alpha$.
We use the induction on $l$.

\noindent
First assume that $l=1$. We have two cases.

Case (1): The non-oriented cycle in $Q$ has length greater than $2$.
Let $\alpha$ be an arrow of the non-oriented cycle with $k_{\alpha}=1$. Let $i$ be the
vertex forming the oriented 3-cycle with $\alpha$. Then $\mu_i(Q)$ is as in Lemma~\ref{lemm10}
and has $m-1$ oriented $3$-cycles.
By induction hypothesis, the cluster algebra given by $\mu_i(Q)$ has the desired property
and by Lemma~\ref{lemm10}, the result follows.

Case (2): The non-oriented cycle in $Q$ has a length $2$. Let $\alpha$ be an arrow of
the non-oriented cycle with $k_{\alpha}=1$ and $\beta$ be the other arrow in the non-oriented cycle.
We have two cases.

Case (i): If $\beta$ is not contained in any oriented 3-cycle of $Q$, then
$\mathcal{L}_P(\Sigma)=\mathcal{A}(\Sigma)$ by Lemma~\ref{lemm11}
and the result follows.

Case (ii): Assume that $\beta$ is contained in an oriented 3-cycle of $Q$. Then we have two cases.

Case (ii)(1): $k_{\beta}=1$. In this case,
$\mathcal{L}_P(\Sigma)=\mathcal{A}(\Sigma)$ by Lemma~\ref{lemm15}
and the result follows.

Case (ii)(2): $k_{\beta}>1$. If $Q_{\beta}$ contains an oriented $3$-cycle, then by the same argument
as in the proof of Theorem~\ref{thm3} (opening one $3$-cycle through an appropriate mutation) and
the induction hypothesis, $\mathcal{L}_P(\Sigma)=\mathcal{A}(\Sigma)$; the result follows.
So we can assume that $Q_{\beta}$ has no $3$-cycles. Let $i$ be the
vertex forming the oriented 3-cycle with $\alpha$. The quiver $\mu_i(Q)$ has $m-1$ oriented
$3$-cycles and so by induction hypothesis, the cluster algebra given by $\mu_i(Q)$ has the desired
property. Therefore the result follows by Lemma~\ref{lemm17}.

Now assume that $l>1$ and that the result holds for any quiver $Q'\in \mathscr{Q}_n$
with $m$ oriented $3$-cycles for which the minimum of all the $k_{\beta}$ in $Q'$ is less than $l$.
Let $\alpha$ be an arrow in $Q$ with $k_\alpha=l$.
If $Q_{\alpha}$ has an oriented $3$-cycle, then by the same argument as in the proof of
Theorem~\ref{thm3} (opening one $3$-cycle through an appropriate mutation) and
the induction hypothesis, $\mathcal{L}_P(\Sigma)=\mathcal{A}(\Sigma)$ and the result follows.
So we can assume that $Q_{\alpha}$  has no oriented $3$-cycles.
Let $i=z_{\alpha}$ be the vertex of $Q$ forming an oriented $3$-cycle with
$\alpha$. We will mutate at the vertex $i$. The new quiver $\mu_i(Q)$ still has $m$ oriented $3$-cycles
but now  the non-oriented cycle of $\mu_i(Q)$ is longer and in turn, the
minimum of all the $k_{\beta}$ in $\mu_i(Q)$ is $l-1$. We have two cases.

Case (1): The non-oriented cycle in $Q$ has length greater than $2$.
In this case by Lemmas \ref{lemm18}, \ref{lemm12} and \ref{lemm13}
(with the role of the vertex $i$ in these lemmas played by the mutated vertex $i$ here)
and induction hypothesis,
$\mathcal{L}_P(\Sigma)=\mathcal{A}(\Sigma)$ and the result follows.

Case (2): The non-oriented cycle in $Q$ has length $2$.
Let $\beta$ be the other arrow of the non-oriented cycle full subquiver of $Q$. If there is no oriented
$3$-cycle in $Q$ which is contain $\beta$, then in this case by Lemma \ref{lemm14} and
induction hypothesis,
$\mathcal{L}_P(\Sigma)=\mathcal{A}(\Sigma)$ and the result follows. Now assume that there is an
oriented $3$-cycle in $Q$ which contains $\beta$. If $Q_{\beta}$ has an oriented $3$-cycle,
then by the same argument as in the proof of
Theorem~\ref{thm3} (opening one $3$-cycle through an appropriate mutation) and
the induction hypothesis, $\mathcal{L}_P(\Sigma)=\mathcal{A}(\Sigma)$ and the result follows.
So we can assume that $Q_{\beta}$ has no oriented $3$-cycles. In this case
$\mathcal{L}_P(\Sigma)=\mathcal{A}(\Sigma)$ by
Lemma~\ref{lemm16} and the induction hypothesis, so the result follows.\\

\end{proof}

\section*{acknowledgements}
The authors would like to thank Siyang Liu, Greg Muller and Fan Qin for their comments on the first version of this paper. Part of this work was carried out during a visit of the second
author to the Institut f\"{u}r Mathematik und Wissenschaftliches
Rechnen, Universit\"{a}t Graz, Austria, with the financial support of the Joint Excellence in Science and Humanities (JESH) program of the Austrian Academy of Sciences. The second author would like to thank the Austrian Academy of Sciences for its support and this host
institution for its warm hospitality. The research of the second
author was in part supported by a grant from IPM (No. 98170412).
The first author was support through the FWF grants P 30549 and W1230. She is also supported by a
Royal Society Wolfson Fellowship.


\begin{thebibliography}{10}

\bibitem{B} \textsc{J. Bastian}, Mutation classes of $\widetilde{A}_n$-quivers and derived equivalence classification of cluster tilted algebras of type $\widetilde{A}_n$,
 \emph{Algebra Number Theory,} \textbf{5}(5) (2011), 567--594.

\bibitem{BFZ} \textsc{A. Berenstein, S. Fomin, A. Zelevinsky}, Cluster algebras III: Upper Bounds and double Bruhat cells, \emph{Duke Math. J.} \textbf{126}(1) (2005), 1--52.

\bibitem{BMRRT}  \textsc{A. Buan, R. Marsh, M. Reineke, I. Reiten, G. Todorov}, Tilting theory and cluster combinatorics,
\emph{Adv. Math.} \textbf{204} (2006) 572--618.

\bibitem{BMR1} \textsc{A. Buan, R. Marsh, I. Reiten},  Cluster-tilted algebras, \emph{Trans. Amer.
Math. Soc.}, \textbf{359}(1) (2007), 323-332.

\bibitem{BMR3} \textsc{A. Buan, R. Marsh, I. Reiten}, Cluster-tilted algebras of finite
representation type, \emph{J. Algebra} \textbf{306}(2) (2006), 412--431.

\bibitem{BMR2} \textsc{A. Buan, R. Marsh, I. Reiten}, Cluster mutation via quiver
representations, \emph{Comment. Math. Helv.} \textbf{83}(2) (2008),
143-177.

\bibitem{BMRT} \textsc{A. Buan, R. Marsh, I. Reiten, G. Todorov}, Clusters and seeds in acyclic cluster algebras, \emph{Proc. Amer. Math. Soc.} \textbf{135}(10) (2007), 3049--3060 (with appendix by the above authors, P. Caldero, and B. Keller).

\bibitem{BV} \textsc{A. Buan, D. F. Vatne},  Derived equivalence classification for cluster-tilted
algebras of type $A_n$, \emph{J. Algebra} \textbf{319} (2008),
2723-2738.

\bibitem{CC} \textsc{P. Caldero, F. Chapoton}, Cluster algebras as Hall algebras of quiver representations, \emph{Comment. Math. Helv.} \textbf{81}(3) (2006),  595--616.

\bibitem{CCS2} \textsc{P. Caldero, F. Chapoton, R. Schiffler},  Quivers with relations and
cluster-tilted algebras, \emph{Algebr. Represent. Theory}
\textbf{9}(4) (2006), 359-376.

\bibitem{CCS}  \textsc{P. Caldero, F. Chapoton, R. Schiffler}, Quivers with relations arising from clusters ($A_n$ case), \emph{Trans. Amer. Math. Soc.} \textbf{358} (2006), 1347--1364.

\bibitem{CK} \textsc{P. Caldero, B. Keller}, From triangulated categories to cluster algebras, \emph{Invent. math.} \textbf{172} (2008), 169--211.

\bibitem{CK1} \textsc{P. Caldero, B. Keller}, From triangulated categories to cluster algebras
II, \emph{Ann. Sci. \'{E}cole Norm. Sup.} (4) \textbf{39}(6) (2006), 983--1009.

\bibitem{DM} \textsc{B. Davison, T. Mandel}, Strong positivity for quantum theta bases of quantum cluster algebras,
arXiv, 1910.12915.

\bibitem{FZ1} \textsc{S. Fomin, A. Zelevinsky},  Cluster algebras I: foundations, \emph{J. Amer.
Math. Soc.} \textbf{15}(2) (2002), 497--529.

\bibitem{FZ2}\textsc{S. Fomin, A. Zelevinsky}, Cluster algebras II: Finite type
classification, \emph{Invent. Math.} \textbf{154}(1) (2003),
63--121.

\bibitem{happel}\textsc{D. Happel}, Triangulated categories in the representation theory of
finite-dimensional algebras, Cambridge University Press, \textbf{119} (1988).

\bibitem{HU}\textsc{D. Happel, L. Unger}, Almost complete tilting modules, \emph{Proc. Amer. Math. Soc.} \textbf{107}(3)
(1989), 603--610.

\bibitem{K}\textsc{B. Keller}, Quiver mutation in Java, Java applet available at the Keller's home page.

\bibitem{KR1}\textsc{B. Keller, I. Reiten}, Acyclic Calabi-Yau categories, \emph{Compos. Math.} \textbf{144}(5) (2008), 1332--1348.

\bibitem{KR}\textsc{B. Keller, I. Reiten}, Cluster-tilted algebras are Gorenstein and stably Calabi-Yau, \emph{Adv. Math.}, \textbf{211}(1)  (2007), 123--151.

\bibitem{KQ}\textsc{Y. Kimura, F. Qin}, Graded quiver varieties, quantum cluster algebras and dual canonical basis,
\emph{Adv. Math.}, \textbf{262} (2014), 261--312.

\bibitem{L}\textsc{G. Lusztig}, Canonical bases arising from quantized enveloping algebras, \emph{J. Amer. Math. Soc.}
\textbf{3} (1990), 447--498.

\bibitem{L1}\textsc{G. Lusztig}, Total positivity in reductive groups, in: Lie theory and geometry: in honor of
Bertram Kostant, Progress in Mathematics 123, Birkh$\ddot{a}$user, 1994.

\bibitem{M} \textsc{G. Muller}, Locally acyclic cluster algebras, \emph{Adv. Math.} \textbf{233} (2013), 207--247.

\bibitem{MRZ} \textsc{G. Muller, J. Rajchgot, B. Zykoski}, Lower bound cluster algebras: presentations, Cohen-Macaulayness, and normality, \emph{Algebraic Combinatorics} \textbf{1}(1) (2018), 95--114.

\bibitem{Pa} \textsc{Y. Palu}, Cluster characters for $2$-Calabi-Yau triangulated categories, \emph{Ann. Inst. Fourier} \textbf{58} (2008), 2221--2248.


\end{thebibliography}
\end{document}